\documentclass[a4paper,12pt]{amsart}

\usepackage[english]{babel}
\usepackage{amsmath,amsthm,amssymb,amsfonts}
\usepackage[all]{xy}
\usepackage{multicol}
\usepackage{todonotes}
\usepackage[shortlabels]{enumitem} 
\usepackage[dvipsnames]{xcolor}
\usepackage{pdflscape}

\usepackage{mathtools}
\usepackage{extarrows}

\usepackage{thmtools}

\usepackage[export]{adjustbox}
\usepackage{graphicx}
\usepackage{subcaption}
\usepackage{tikz}
\usetikzlibrary{mindmap,backgrounds, calc}
\usetikzlibrary{matrix,chains,scopes,positioning,arrows,fit}
\usetikzlibrary{positioning, shapes, patterns}
\usetikzlibrary{automata} 
\usetikzlibrary{shapes.geometric, arrows, arrows.meta}
\usetikzlibrary{positioning,decorations.markings}
\usepackage{caption}
\usepackage[normalem]{ulem}
\usepackage{marginnote}
\usepackage{cancel}
\usepackage{bbm}
\usepackage{comment}

\setlist[enumerate]{itemsep=0.2em, topsep=0.25em}

\makeatletter
\newcommand*\bigcdot{\mathpalette\bigcdot@{.5}}
\newcommand*\bigcdot@[2]{\mathbin{\vcenter{\hbox{\scalebox{#2}{$\m@th#1\bullet$}}}}}
\makeatother

\newcommand{\vertiii}[1]{{\left\vert\kern-0.25ex\left\vert\kern-0.25ex\left\vert #1 
		\right\vert\kern-0.25ex\right\vert\kern-0.25ex\right\vert}}

\makeatletter
\renewcommand{\tocsection}[3]{%
	\indentlabel{\@ifnotempty{#2}{\bfseries\ignorespaces#1 #2\quad}}\bfseries#3}
\renewcommand{\tocsubsection}[3]{%
	\indentlabel{\@ifnotempty{#2}{\ignorespaces#1 #2\quad}}#3}

\def\@tocline#1#2#3#4#5#6#7{\relax
	\ifnum #1>\c@tocdepth 
	\else
	\par \addpenalty\@secpenalty\addvspace{#2}%
	\begingroup \hyphenpenalty\@M
	\@ifempty{#4}{%
		\@tempdima\csname r@tocindent\number#1\endcsname\relax
	}{%
		\@tempdima#4\relax
	}%
	\parindent\z@ \leftskip#3\relax \advance\leftskip\@tempdima\relax
	\rightskip\@pnumwidth plus1em \parfillskip-\@pnumwidth
	#5\leavevmode\hskip-\@tempdima{#6}\nobreak
	\leaders\hbox{$\m@th\mkern \@dotsep mu\hbox{.}\mkern \@dotsep mu$}\hfill
	\nobreak
	\hbox to\@pnumwidth{\@tocpagenum{\ifnum#1=1\bfseries\fi#7}}\par
	\nobreak
	\endgroup
	\fi}
\AtBeginDocument{%
	\expandafter\renewcommand\csname r@tocindent0\endcsname{0pt}
}
\def\l@subsection{\@tocline{2}{0pt}{2.5pc}{5pc}{}}
\makeatother

\usepackage{tikz}
\usetikzlibrary{mindmap,backgrounds, calc}
\usetikzlibrary{matrix,chains,scopes,positioning,arrows,fit}
\usetikzlibrary{positioning, shapes, patterns}
\usetikzlibrary{automata} 
\usetikzlibrary{shapes.geometric, arrows, arrows.meta}
\usetikzlibrary{positioning,decorations.markings}

\usepackage[a4paper,left=3cm,right=3cm,top=2.5cm,bottom=3cm]{geometry}

\theoremstyle{definition}
\newcounter{maincoro}

\newtheorem{theorem}{Theorem}[section]
\newtheorem{lemma}[theorem]{Lemma}

\newtheorem{proposition}[theorem]{Proposition}
\newtheorem{corollary}[theorem]{Corollary}

\theoremstyle{definition}
\newcounter{maintheorem}

\newtheorem{definition}[theorem]{Definition}
\newtheorem{example}[theorem]{Example}
\newtheorem{problem}[theorem]{Problem}

\theoremstyle{remark}
\newtheorem{remark}[theorem]{Remark}

\numberwithin{equation}{section}

\newcommand{\R}{\mathbb{R}}
\newcommand{\C}{\mathbb{C}}
\newcommand{\N}{\mathbb{N}}
\newcommand{\Z}{\mathbb{Z}}
\newcommand{\K}{\mathbb{K}}
\newcommand{\Q}{\mathbb{Q}}
\newcommand{\T}{\mathbb{T}}

\newcommand{\clco}{\mathop{\overline{\mathrm{co}}}\nolimits}

\newcommand{\Iso}{\mathrm{Iso}}

\makeatletter
\tikzset{
	arrow/.style   = {-{Implies}, double equal sign distance, line width=0.65pt},
	iffarrow/.style= {{Implies}-{Implies}, double equal sign distance, line width=0.65pt},
}

\renewcommand{\tocsection}[3]{%
	\indentlabel{\@ifnotempty{#2}{\bfseries\ignorespaces#1 #2\quad}}\bfseries#3}
\renewcommand{\tocsubsection}[3]{%
	\indentlabel{\@ifnotempty{#2}{\ignorespaces#1 #2\quad}}#3}

\def\@tocline#1#2#3#4#5#6#7{\relax
	\ifnum #1>\c@tocdepth 
	\else
	\par \addpenalty\@secpenalty\addvspace{#2}%
	\begingroup \hyphenpenalty\@M
	\@ifempty{#4}{%
		\@tempdima\csname r@tocindent\number#1\endcsname\relax
	}{%
		\@tempdima#4\relax
	}%
	\parindent\z@ \leftskip#3\relax \advance\leftskip\@tempdima\relax
	\rightskip\@pnumwidth plus1em \parfillskip-\@pnumwidth
	#5\leavevmode\hskip-\@tempdima{#6}\nobreak
	\leaders\hbox{$\m@th\mkern \@dotsep mu\hbox{.}\mkern \@dotsep mu$}\hfill
	\nobreak
	\hbox to\@pnumwidth{\@tocpagenum{\ifnum#1=1\bfseries\fi#7}}\par
	\nobreak
	\endgroup
	\fi}
\AtBeginDocument{%
	\expandafter\renewcommand\csname r@tocindent0\endcsname{0pt}
}
\def\l@subsection{\@tocline{2}{0pt}{2.5pc}{5pc}{}}
\makeatother

\DeclareMathOperator{\dist}{dist\,}
\DeclareMathOperator{\diam}{diam\,}
\DeclareMathOperator{\co}{co}

\DeclareMathOperator{\re}{Re}

\DeclareMathOperator{\id}{Id}

\DeclareMathOperator{\sign}{sign}
\DeclareMathOperator{\iso}{Iso}

\newcommand{\nn}[1]{{\left\vert\kern-0.25ex\left\vert\kern-0.25ex\left\vert #1 
		\right\vert\kern-0.25ex\right\vert\kern-0.25ex\right\vert}}

\renewcommand{\geq}{\geqslant}
\renewcommand{\leq}{\leqslant}

\newcommand{\norm}[1]{\left\Vert#1\right\Vert}
\newcommand{\abs}[1]{\left\vert#1\right\vert}

\newcommand{\spann}{\operatorname{span}}

\newcommand{\dent}[1]{\operatorname{dent}\left(#1\right)}

\newcommand{\e}{\varepsilon}

\newcommand{\ext}{\operatorname{ext}}

\newcommand{\eps}{\e}

\newcommand{\clspan}{\overline{\operatorname{span}}}

\usepackage[bookmarks=true,hyperindex,pdftex,colorlinks,citecolor=cyan]{hyperref}
\hypersetup{colorlinks=true,  linkcolor=blue, citecolor=red, urlcolor=cyan}
\usepackage{cleveref}

\begin{document}

	\title[A group action approach to the Daugavet property]{A group action approach to the Daugavet property}
	
	\author[S.~Dantas]{Sheldon Dantas}
	\address[S.~Dantas]{Czech Technical University in Prague, FEE, Department of Mathematics, Technická 2, 16627, Prague 6, Czech Republic \newline
		\href{https://orcid.org/0000-0001-8117-3760}{ORCID: \texttt{0000-0001-8117-3760}}}
	\email{\texttt{sheldon.dantas@fel.cvut.cz}}
	\urladdr{sheldondantas.com}
	
	\author[H.~Del R\'{\i}o]{Helena del R\'{\i}o} 
	\address[H.~Del R\'{i}o]{Department of Mathematical Analysis and Institute of Mathematics (IMAG), University of Granada, E-18071 Granada, Spain \newline
		\href{https://orcid.org/0009-0004-5078-6993}{ORCID: \texttt{0009-0004-5078-6993} }}
	\email[]{\texttt{helenadelrio@ugr.es}}
	
	\author[T.~Raunig]{Tomáš Raunig}
	\address[T.~Raunig]{Institute of Mathematics, Czech Academy of Sciences, Žitná 25, 115 67, Czech Republic\newline second address: Faculty of Mathematics and Physics, Charles University, Sokolovská 83, Prague, 186 00 \newline
		\href{https://orcid.org/0009-0001-3425-8726}{ORCID: \texttt{0009-0001-3425-8726}}}
	\email{raunig@karlin.mff.cuni.cz}

	\begin{abstract} We introduce the $G$-Daugavet property ($G$-DPr, for short) for Banach spaces endowed with an action of a group $G$ by surjective linear isometries. This notion provides a common framework for the classical Daugavet property and the alternative Daugavet property, which correspond respectively to the trivial action and to the scalar action of $S_{\K}$. We establish several characterizations of the $G$-DPr in terms of $G$-slices and closed convex $G$-invariant hulls, recovering the usual slice descriptions of the DPr and the aDPr as particular cases. We show that the presence of a group action leads to new behavior in Daugavet theory. In particular, the $G$-DPr may hold on classical reflexive spaces in sharp contrast with the classical Daugavet property. We relate this phenomenon to convex transitivity, almost transitivity and finite-dimensional rotation problems. We also prove group-action versions of the classical characterizations for $L^1(\mu, X)$- and $C(K,X)$-spaces. The paper also studies group separable determination, $G$-versions of numerical radius and numerical index, and connections between the $G$-DPr and strong Radon-Nikodým and SCD operators.
		Finally, we introduce a parameter which measures how far the $G$-DPr is from the classical DPr in a quantitative manner. As a consequence of these results, we obtain conditions under which the $G$-DPr recovers several classical implications, including the failure of the RNP for both $X$ and $X^*$, the presence of copies of $\ell_1$ and the failure of the unit ball to be an SCD set. 
	\end{abstract}

	\thanks{ }
	
	
	\maketitle
	
	\tableofcontents
	
	\thispagestyle{plain}

	\section{Introduction} The starting point of the present work is the classical theorem of I.K. Daugavet, which asserts that every compact operator $T$ on $C[0,1]$ satisfies the norm identity
	\begin{equation} \label{eq1}
		\|\id_X+T\|=1+\|T\|.
	\end{equation}
	This identity leads to the definition of the Daugavet property. A Banach space $X$ is said to have the \emph{Daugavet property} $($DPr, for short$)$ if every rank-one operator $T:X\to X$ satisfies (\ref{eq1}). Since its definition, the Daugavet property has become a central object in the geometry of Banach spaces. There are many classical examples of Banach spaces with the DPr. Among them one finds $C(K)$-spaces when $K$ is perfect \cite{Daugavet1963}, $L^1(\mu)$-spaces when $\mu$ is atomless \cite{Lozanovskii66}, Lipschitz-free spaces and Lipschitz spaces over length metric spaces \cite{IKW07,GLPRZ18}, non-atomic $C^*$-algebras and preduals of diffuse von Neumann algebras \cite{BM05,Oik02}, as well as certain Banach algebras of holomorphic functions on arbitrary Banach spaces \cite{Jung2023,Werner97,Wojtaszcyk92}. It was proved recently that the DPr is equivalent to its polynomial version \cite{DantasMartinPerreau2025}. For a detailed presentation of the theory and its background, we refer the reader to the recent monograph \cite{KMRW}.

	The Daugavet property has received considerable attention not only as an isometric property, but also because of its strong consequences for the linear structure of the space. For instance, a Banach space with the DPr cannot be isomorphically embedded into a space with an unconditional Schauder basis \cite{Kadets2000}, extending a classical theorem of A. Pe{\l}czýnski for $L^1[0,1]$. Further structural consequences are that such a Banach space fails the Radon-Nikodým property (RNP, for short) and contains many isomorphic copies of $\ell_1$. In particular, spaces with the Daugavet property cannot be Asplund \cite{KMRW, Kadets2000}.

	A closely related notion is the \emph{alternative Daugavet property} $($aDPr, for short$)$, which was introduced and systematically studied in \cite{MO2004}. Let us recall its definition. A Banach space $X$ is said to have the aDPr if every rank-one operator $T:X\to X$ satisfies
	\begin{equation*}
		\max_{\theta\in S_{\K}}\|\id_X+\theta T\|=1+\|T\|.
	\end{equation*}
	Clearly, the aDPr is weaker than the DPr, but it still shares many of its geometric features. Also it is deeply connected with numerical ranges and numerical index (see \cite{MO2004} and also \cite{BonsallDuncanV1,BonsallDuncanV2}). Several results for vector-valued function spaces have parallel formulations for the DPr and the aDPr \cite{MO2004, MartinPaya2000}. Nevertheless, the aDPr can be drastically different from the classical DPr: for instance, in the real case it may occur in finite-dimensional spaces \cite{LMP1999,MO2004}.
	
	The parallel behavior of the DPr and the aDPr suggests that both properties should be understood as particular cases of a more general principle. Indeed, the alternative Daugavet equation already contains the action of the scalar group $S_{\K}$ on the space. This observation naturally leads to the following broader point of view: instead of only allowing multiplication by scalars, one may allow an arbitrary group of surjective linear isometries to act on the space. Group actions have become increasingly relevant in Banach space theory, appearing in the study of transitivity, invariant renormings, rigidity, variational principles, twisted sums and related geometric questions (see, for instance, \cite{AronFalcoMaestre2018,BaderFurmanGelanderMonod2007,BekkaDeLaHarpeValette2008,CastilloFerenczi2023,DantasDouchaJungRaunig2025,DantasFalcoJung2023,  Doucha2026,Falco2021GroupInvariantBishopPhelps,FalcoGarciaJungMaestre2022, FalcoIsert2024, FalcoIsert2026,FerencziRosendal2013,Raunig2025}).
	
	The aim of this paper is to develop such a group-action approach to Daugavet-type properties. Given a Banach space $X$ and a group $G$ acting on $X$ by surjective linear isometries, we introduce the \emph{$G$-Daugavet property} $($$G$-DPr, for short$)$. We say that $X$ has the $G$-DPr if every rank-one operator $T:X\to X$, not assumed to commute with the action of $G$, satisfies the equation
	\begin{equation*}
		\sup_{g\in G}\|\id_X+g\circ T\|=1+\|T\|.
	\end{equation*}
	Let us notice that this definition contains the two classical cases. If $G=\{\id_X\}$, then one recovers exactly the classical DPr. If $G=S_{\K}$ acts by scalar multiplication, then one recovers the alternative Daugavet property. Thus the $G$-DPr provides a common framework in which the DPr and the aDPr can be treated at the same time.
	
	A substantial part of the classical theory can be recovered in this wider setting. In particular, the usual slice characterizations of the DPr and the aDPr admit natural $G$-versions once ordinary slices are replaced by $G$-slices and closed convex hulls are replaced by closed convex $G$-invariant hulls. This allows several arguments which are traditionally written separately for the DPr and the aDPr to be viewed as instances of the same general mechanism.
	
	One of the main points of the paper is that this is not merely a formal generalization. Indeed, the presence of a group action leads to a new behavior in the Daugavet theory. In the classical theory, the Daugavet property is incompatible with reflexivity and with several regularity properties of Banach spaces. By contrast, the $G$-DPr may hold for classical spaces which are reflexive (or, as expected, even finite-dimensional as the aDPr). This shows that Daugavet-type behavior may arise not only from the geometry of the space itself, but also from the way the group moves points of the unit sphere. In the second half of this paper, our goal is therefore to identify which parts of the classical theory survive in the general group-action setting and which parts fail without further restrictions on the action.

	\subsection{Outline of the results}
	
	The paper is organized as follows. In Section \ref{sec: definition}, we introduce the main notion of the paper, namely the $G$-Daugavet property. This property is defined for Banach spaces endowed with an action (not necessarily continuous) of a group $G$ by linear isometries and requires every rank-one operator to satisfy the corresponding $G$-Daugavet equation. As we have mentioned before, this new definition contains the classical Daugavet property as the particular case $G=\{\id_X\}$ and the alternative Daugavet property as the case in which $G=S_{\K}$ acts by scalar multiplication. We first record in Proposition \ref{fact:G-DP-equation-commutes} that the different natural ways of inserting the group element in the Daugavet equation lead to the same supremum. We also prove the basic monotonicity property in Proposition \ref{fact-DP-implies-GDP}, from which Corollary \ref{DPr-implies-G-DPr} shows that both the DPr and the aDPr imply the corresponding $G$-DPr in the expected cases.
	
	After introducing $G$-slices in Definition \ref{def:G-slice}, we establish the main geometric characterizations of the $G$-DPr in Proposition \ref{prop:characterization-G-DP}. These characterizations are formulated in terms of $G$-slices and closed convex $G$-invariant hulls. They provide the group counterpart of the classical slice descriptions of the DPr and the aDPr, and they are used repeatedly throughout the paper. We also prove in Proposition \ref{prop:inheritanceFromDual} that the $G$-DPr is inherited from the dual space when $X^*$ is equipped with the canonical dual action.
	
	Section \ref{section:far-classical} is devoted to examples showing that the $G$-DPr can behave very differently from the classical DPr. We first prove in Proposition \ref{cor:convexTranstiveImpliesGDP} that convex-transitive actions always yield the $G$-DPr. In the LUR setting, Theorem \ref{theorem:GDPrOnUniformlyConvexSpace} gives a converse: an LUR $G$-Banach space has the $G$-DPr if and only if the action of $G$ on the unit sphere is almost transitive. This leads to the examples in Corollary \ref{cor:Lp-ellp-examples}, including the fact that $L^p[0,1]$ has the $\Iso(L^p[0,1])$-DPr, whereas $\ell_p$ fails the $G$-DPr for every $G\leq \Iso(\ell_p)$ when $p\neq2$. In finite dimensions, Corollary \ref{corollary:finite-dimensional-strictly-convex} relates the $G$-DPr for strictly convex spaces to transitivity of the closure of the group and to the finite-dimensional Mazur rotation problem. We also provide more infinite-dimensional reflexive spaces with the $G$-DPr in Example \ref{inf-dim-reflexive-with-GDPr-not-almost-transitive}.
	
	We then turn to stability and vector-valued function spaces. The $\ell_1$-sum stability results are given in Proposition \ref{stability-ell1} and Proposition \ref{prop:diagonal-ell1-stability}, while the corresponding $\ell_\infty$-sum result is obtained in Proposition \ref{prop:ellInftySumStability}. These results are used to prove the group analogues of the classical characterizations for $L^1(\mu,X)$ and $C(K,X)$. More precisely, Theorem \ref{L1(X,mu)-iff} shows that, for the pointwise action induced by the action on $X$, the space $L^1(\mu,X)$ has the $G$-DPr if and only if $\mu$ is atomless or $X$ has the $G$-DPr. Similarly, Theorem \ref{thm:c-of-K-X} proves that $C(K,X)$ has the $G$-DPr if and only if $K$ is perfect or $X$ has the $G$-DPr, recovering once again classical results. Examples \ref{examples-L1-C(K)-independent-actions} and \ref{C(K,X)-action} show that the compatibility between the action on the function space and the action on $X$ is essential.
	
	The next part of the paper studies separable determination. Proposition \ref{prop:denseSubsetG-DP} gives a dense-subset criterion for the $G$-DPr and Lemma \ref{lem:sepDeterminationLemma} provides the main technical tool for passing from increasing families of separable invariant subspaces and subgroups to the ambient space. These results lead to group-action separable determination theorems for the $G$-DPr, both for arbitrary topological groups and, in a simpler form, for separable or $\sigma$-compact groups with continuous actions.
	
	We also introduce a group version of the numerical radius $v_G(T)$ of an operator $T$ and numerical index $n_G(X)$ of a $G$-Banach space $X$. Lemma \ref{lem:G-DP-numerical-radius} proves that the $G$-Daugavet equation for an operator is equivalent to the equality $v_G(T)=\|T\|$. This allows us to connect the $G$-DPr with numerical-index theory. Among other results, Proposition \ref{prop:convex-transitivity-G-numerical-index} shows that convex-transitive actions have $G$-numerical index one. We further relate this circle of ideas to strong Radon-Nikodým operators in Proposition \ref{prop:sRN-Operators}, obtaining in Theorem \ref{thm:G-DPandRNP} that if $X$ has both the RNP and the $G$-DPr, then $n_G(X)=1$. The more general group version involving two group actions is given in Theorem \ref{thm:GDPandGRNP} where we estabilish a connection between the $G$-DPr and the $G$-RNP recently introduced in \cite{DantasDouchaJungRaunig2025}.
	
	The relation with SCD sets and operators is developed next. Proposition \ref{prop:GDPandSCDslices} gives an SCD-type characterization of the $G$-DPr in terms of weak-star dense subsets. Proposition \ref{prop:G-DPandSCDoperators} shows that every SCD-operator on a space with the $G$-DPr satisfies the $G$-Daugavet equation. This is extended to certain nonseparable-range operators in Corollary \ref{cor:GDP-SCD-nonseparable}. As a consequence, Corollary \ref{result-contain-copies-of-l1} shows that, under the assumption that the action is continuous, a $G$-Banach space with the $G$-DPr that contains no copy of $\ell_1$ must satisfy $n_G(X)=1$. We also include a group version of lushness in Proposition \ref{G-DPr-G-lush}, which says that if $X$ has the $G$-DPr but it is not $G$-lush, then its unit ball is not an SCD set. This once again recovers classical results in the theory.
	
	In the final section, we introduce a parameter denoted by $\alpha_G(X)$. It measures how far the action is from producing antipodal behavior on the unit sphere. Our main goal in this section is to provide a quantitative way to say when the $G$-DPr is close to the DPr. Theorem \ref{thm:GDPr-alphaG} shows that if $X$ has the $G$-DPr and $\alpha_G(X)<2$, then $X$ fails the RNP and, moreover, if $X$ has an unconditional basis, then it contains a copy of $c_0$. Proposition \ref{relation-between-alpha-and-numerical-index} relates $\alpha_G(X)$ to the $G$-numerical index and Theorem \ref{theorem:all-together} recovers several classical results: under the assumption that $\alpha_G(X)<2$, the $G$-DPr forces the presence of $\ell_1$ under continuity of the action, the failure of the RNP and the failure of $B_X$ being an SCD set. We also prove the quantitative estimate of Proposition \ref{proposition-estimation-alpha-G}, which gives a lower bound for $\|\id_X+T\|$ for rank-one operators in terms of $\alpha_G(X)$. The last quantitative results concern Daugavet indices of thickness. For real $G$-Banach spaces with the $G$-DPr, Theorem \ref{estimate-T-s-alpha-G} proves that $\mathcal T^s(X)\geq 2-\alpha_G(X)$, while Theorem \ref{theorem:relation-between-alphaG-dual} gives the weak-star dual estimate $\mathcal T^s_{w^*}(X)\geq 2-\alpha_G(X)$. These estimates show that when $\alpha_G(X)<2$, the $G$-DPr recovers quantitative features of the classical DPr. Finally, in the complex case, we obtain Corollary \ref{complex-reflexive-aDPr-problem}, which provides a complete picture of $S_{\C}$-Banach spaces under the scalar multiplication. Indeed, if $H \leq S_{\C}$, then $X$ cannot have the $H$-DPr whenever $X$ has the RNP and $H$ is finite, and if $H$ is infinite, then to say that $X$ has the $H$-DPr is equivalent to say that $X$ has the aDPr.

	\subsection{Notation} In the present paper, all Banach spaces will be considered over the scalar field $\K$, where $\K=\R$ or $\C$. Given a Banach space $X$, we denote by $B_X$ its closed unit ball, by $S_X$ its unit sphere and by $X^*$ its topological dual. The symbol $S_{\K}$ stands for the unit sphere of the scalar field $\K$; that is, $S_{\R}=\{-1,1\}$ in the real case and $S_{\C}=\{\lambda\in\C:|\lambda|=1\}$ in the complex case. When convenient, we identify $S_{\K}$ with the subgroup $\{\lambda\id_X:\lambda\in S_{\K}\}\leq \Iso(X)$.
	
	For a subset $A$ of a Banach space, we denote by $\diam(A)$ its diameter, by $\co(A)$ its convex hull and by $\overline{\co}(A)$ its closed convex hull. Given $z\in X$ and $\e>0$, we denote by $B(z,\e)$ the open ball centered at $z$ with radius $\e$. We denote by $\mathcal L(X,Y)$ the Banach space of all bounded linear operators from $X$ into $Y$ and, in order to simplify the notation, we write $\mathcal L(X)=\mathcal L(X,X)$. Throughout the paper, the word operator always means bounded linear operator. The identity operator on $X$ will be denoted by $\id_X$ and $\Iso(X)$ stands for the group of all surjective linear isometries from $X$ onto itself.
	
	We will use $\re$ to denote the real part of a complex number. If $x^*\in S_{X^*}$ and $\e>0$, the slice of $B_X$ determined by $x^*$ and $\e$ is the set
	\begin{equation*} 
		S(B_X,x^*,\e) = \{x\in B_X:\re x^*(x)>1-\e\}.
	\end{equation*}

	We refer to a $G$-Banach space as a Banach space endowed with an action of a group $G$ by surjective linear isometries, or equivalently, with a group homomorphism from $G$ into $\Iso(X)$. We write $g\cdot x$, or simply $gx$ when the context is clear, for the image of $x\in X$ under the action of $g\in G$. Whenever we need the continuity of the action, we will be explicit about it. We will use the phrase ``Let $X$ be a $G$-Banach space'' as a shorthand for ``Let $X$ be a Banach space, let $G$ be a group and let $G$ act on $X$ by surjective linear isometries''. If $X$ is a $G$-Banach space and $T\in\mathcal L(X)$, we will use the notation $T\circ g$ and $g\circ T$ to denote the bounded linear operators given by $(T\circ g)(x)=T(gx)$ and $(g\circ T)(x)=g(Tx)$ for every $x\in X$.
	
	Let $X$ and $Y$ be $G$-Banach spaces. We say that a linear operator $T:X\to Y$ is $G$-equivariant, or simply equivariant when the context is clear, if $T(gx)=gT(x)$ for every $g\in G$ and every $x\in X$. We denote by $\mathcal L^G(X,Y)$ the Banach space of all $G$-equivariant bounded linear operators from $X$ into $Y$ and we write $\mathcal L^G(X)=\mathcal L^G(X,X)$. In the language of representation theory, such operators are often called intertwining operators.
	
	A subset $C$ of a $G$-Banach space $X$ is said to be $G$-invariant if $gC\subseteq C$ for every $g\in G$. Since $G$ is a group, this is equivalent to $gC=C$ for every $g\in G$. For every $C\subseteq X$, we denote by $\clco_G(C)$ to be the smallest closed convex $G$-invariant subset of $X$ containing $C$. Equivalently, $\clco_G(C)=\overline{\co}(G\cdot C)$, where we use the notation $G\cdot C=\{gx:g\in G,\ x\in C\}$. Finally, whenever $x^*\in X^*$ and $x\in X$, we denote by $x^*\otimes x$ the rank-one operator on $X$ defined by $(x^*\otimes x)(y)=x^*(y)x$ for every $y\in X$.

	\section{Main definition, first results and characterizations} \label{sec: definition}
	
	In this section, we introduce the main definition of the paper and present the first results. We begin by recalling the definition of the Daugavet property, which was briefly discussed in the introduction and state it here in a formal manner.
	\begin{definition}
		A Banach space $X$ is said to have the {\it Daugavet property} (DPr, for short) if every rank-one operator $T:X\to X$ satisfies
		\begin{equation} \label{Daugavet-equation}
			\norm{\id+T}=1+\norm{T}.
		\end{equation}
	\end{definition}
	A similar property, strictly weaker than the DPr, with roots in the theory of numerical ranges has been also studied (see \cite{MO2004} - see also the monograph \cite{KadetsMartinMeriPerez2018}). This property has became known as the alternative Daugavet property for its resemblance to the Daugavet property and is formally defined as follows.
	\begin{definition} A Banach space $X$ is said to have the {\it alternative Daugavet property} (aDPr, for short) if every rank-one operator $T:X\to X$ satisfies 
		\begin{equation} \label{alternative-Daugavet-equation}
			\max_{\theta \in S_{\K}}\norm{\id+\theta T}=1+\norm{T}.
		\end{equation}
		
	\end{definition}

	Throughout the text, whenever convenient, we will refer to (\ref{Daugavet-equation}) and (\ref{alternative-Daugavet-equation}) as the Daugavet equation and the alternative Daugavet equation, respectively. While the similarities between both properties have been long recognized and it is known that often proofs for one property can be very easily adapted to the other (see, for instance, \cite[Section 12.3]{Kadets2000}), both of these properties have been treated separately in the literature. In what follows, we introduce the $G$-Daugavet property, of which both the Daugavet and the alternative Daugavet properties are particular instances. In other words, this new property will allow us to treat both properties in a unified way.

	\begin{definition} \label{definition-G-DPr}
		Let $X$ be a $G$-Banach space. We say that $X$ has the \emph{$G$-Daugavet property} ($G$-DPr, for short) if every rank-one operator $T:X\to X$ satisfies the \emph{$G$-Daugavet equation}
		\begin{equation} \label{G-Daugavet-equation-1}
			\sup_{g\in G}\norm{\id + \; g \circ T}=1+\norm{T}.
		\end{equation}
	\end{definition}

	It is clear that if we take $G=\{\id_X \}$, then the $G$-Daugavet property becomes the Daugavet property and if we take $G=S_{\K}$, then the $G$-Daugavet property becomes the alternative Daugavet property since $S_{\K}$ is compact and thus the supremum on $G$ in (\ref{G-Daugavet-equation-1}) is attained. Notice also that, as in the classical cases, $T$ satisfies the $G$-Daugavet equation if and only if $aT$ satisfies the $G$-Daugavet equation for every real number $a > 0$.

	Immediately one should wonder about our choice of the expression in the supremum (\ref{G-Daugavet-equation-1}), that is, why do we consider $g\circ T$ instead of $T\circ g$? For the alternative Daugavet property, both expressions are equivalent since we consider $T$ to be a linear operator. However, for a~general group $G$, if $T$ is not $G$-equivariant, then we do not have in general that the identity $g\circ T=T\circ g$ holds true for every $g\in G$. However, as we show in the following proposition, the $G$-Daugavet equation is in some sense ``$G$-equivariant'' in its own right.

	\begin{proposition} \label{fact:G-DP-equation-commutes}
		Let $X$ be a $G$-Banach space. Let $T: X \to X$ be a linear bounded operator. Then, we have that
		\begin{equation*}
			\sup_{g \in G} \norm{g + T}
			= \sup_{g \in G} \norm{\id + T \circ g}
			= \sup_{g \in G} \norm{\id + g \circ T}.
		\end{equation*}
	\end{proposition}
	
	\begin{proof} Consider the conjugation on $\mathcal{L}(X)$, that is, for every $g \in G$, consider the map $\alpha_g: \mathcal{L}(X) \rightarrow \mathcal{L}(X)$ defined by $\alpha_g(S) := g \circ S \circ g^{-1}$ for every $S \in \mathcal{L}(X)$. Then, $\alpha_g$ defines an action of $G$ on $\mathcal{L}(X)$. Notice also that it is linear and isometric. As
		\begin{equation*}
			\alpha_g(\id_X) = g \circ \id_X \circ g^{-1} = \id_X, 
		\end{equation*}
		the identity is $G$-invariant, that is, it is fixed by the action. Finally, the action of $g$ on $\id_X + T \circ g$ is exactly $\id_X + g \circ T$ as 
		\begin{equation*}
			\alpha_g(\id_X + T \circ g) = g \circ (\id_X + T \circ g) \circ g^{-1} = \id_X + g \circ T
		\end{equation*}
		and therefore the norms are the same as the action is isometric. Now we take supremum over $G$ and we are done. For the first equality, simply note that for every $g\in G$, since $g^{-1}$ is an isometry, $\|g+T\|=\|(g+T)\circ g^{-1}\|=\|\id_X+T\circ g^{-1}\|$. Then  we are done as we can take supremum once again by using the fact that inversion is a bijection on $G$.
	\end{proof}

	From \Cref{fact:G-DP-equation-commutes}, we then can say that a bounded linear operator $T:X\to X$ satisfies the $G$-Daugavet equation whenever
	\begin{equation} \label{G-Daugavet-equation}
		\sup_{g\in G}\norm{\id + \; g \circ T}= \sup_{g \in G} \norm{g + T}
		= \sup_{g \in G} \norm{\id + T \circ g} = 1+\norm{T} .
	\end{equation}

	Since the triangle inequality always guarantees that for an action by isometries we have $\norm{\id + g \circ T} \leq 1 + \norm{T}$ for all $g \in G$, the following fact is an immediate consequence of the definition of the $G$-Daugavet property.
	
	\begin{proposition} \label{fact-DP-implies-GDP}
		Let $X$ be a $G$-Banach space and let $H \leq G$ be a~subgroup of $G$. If $X$ has the $H$-DPr, then $X$ has the $G$-DPr. 
	\end{proposition}
	
	As an immediate consequence of \Cref{fact-DP-implies-GDP}, we the have \Cref{DPr-implies-G-DPr} below. As we will see throughout the entire paper (see, for instance, \Cref{section:far-classical}), the converse of (a) and (b) in Corollary \ref{DPr-implies-G-DPr} below are far for holding true. For (a) see, for instance, \Cref{cor:Lp-ellp-examples}. A straightforward example which shows that satisfying the $G$-DPr with $S_{\K} \subseteq G$ is not enough to get back the aDPr is the following. Let $X=\ell_2(\C)$ and $G=\iso(X)$, the full unitary group, which acts on $X$ as $g \cdot x = g(x)$ for every $g \in G$ and $x \in X$. Then, $S_{\C} \subseteq G$ as every modulus-one scalar $\omega$ acts as the unitary $\omega \id_X$. Moreover, the action of $G$ on $X$ is transitive and, in particular, almost transitive. By \Cref{theorem:GDPrOnUniformlyConvexSpace}, $X$ satisfies the $G$-DPr. However, the numerical index of the complex Hilbert space is 1/2 (see \cite{Halmos1967}) and this cannot happen with spaces that satisfy the aDPr (see, for instance, \cite[Corollary 12.3.2]{KMRW}).
	
	\begin{corollary} \label{DPr-implies-G-DPr} Let $X$ be a $G$-Banach space.
		\begin{itemize}
			\item[(a)] If $X$ has the DPr, then $X$ has the $G$-DPr.
			\item[(b)] If $X$ has the aDPr, then $X$ has the $G$-DPr whenever $S_{\K} \subseteq G$.
		\end{itemize}
	\end{corollary}

	Next, we show that the well-known geometric characterizations of the DPr and aDPr admit straightforward generalizations to the $G$-DPr by placing $g \in G$ conveniently. Before doing so, since the theory of the Daugavet property relies heavily on the notion of slices of the unit ball, we first introduce the corresponding $G$-version of this concept and provide a simple example.
	
	\begin{definition}\label{def:G-slice} Let $X$ be a~$G$-Banach space, $\e>0$ and $x^* \in S_{X^*}$. A \emph{$G$-slice of $B_X$} is a set of the form 
		\begin{equation*}
			S_G(B_X, x^*, \e) = \{y \in B_X: (\exists g \in G) \ \re x^*(gy) > 1 - \e \} = \bigcup_{g \in G} g^{-1} S(B_X, x^*, \e).
		\end{equation*}
	\end{definition}

	The following simple example illustrates how different the classical notion of a slice is from the group one.
	
	\begin{example} \label{ex:G-Slice}
		Let $X=\mathbb R^2$ endowed with the Euclidean norm so that $B_X$ is the unit disk and let $x^*((a,b))=a$ for every $(a,b)\in \R^2$. Fix $0<\e<1$. Then the classical slice of $B_X$ determined by $x^*$ and $\e$  is given by $S(B_X,x^*,\e)=\{(a,b)\in B_X:\ a>1-\e\}$, that is, the cap on the right-hand side of the unit disk. Now let $G=\langle R_{\pi/2}\rangle\cong \mathbb Z_4$, where $R_{\pi/2}(a,b)=(-b,a)$. Then, by definition of $G$-slice, we have that
		\begin{equation*} 
			S_G(B_X,x^*,\e)
			=
			\{(a,b)\in B_X:\ (\exists g\in G) \ x^*(g(a,b))>1-\e\}.
		\end{equation*} 
		Since the orbit of $(a,b)$ under $G$ consists of $(a,b), (-b,a), (-a,-b)$ and $(b,-a)$, we get that
		\begin{equation*} 
			S_G(B_X,x^*,\e)
			=
			\{(a,b)\in B_X:\ \max\{a,-a,b,-b\}>1-\e\}.
		\end{equation*}
		Hence $S_G(B_X,x^*,\e)$ is exactly the union of the four rotated copies of the classical slice (see Figure \ref{fig:G-slice-R2}).
		
		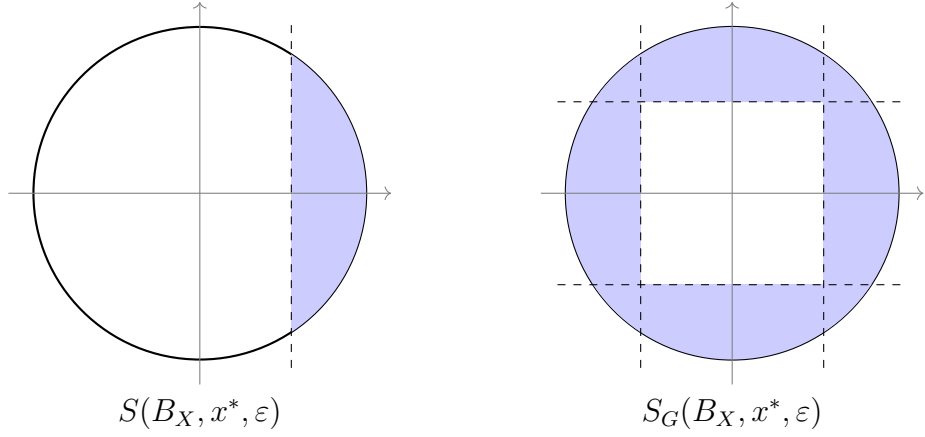
\begin{figure}[ht]
			\centering
			\begin{tikzpicture}[scale=2.2]
				\def\c{0.55} 
				
				\begin{scope}
					\draw[thick] (0,0) circle (1);
					\begin{scope}
						\clip (0,0) circle (1);
						\fill[blue!20] (\c,-1.2) rectangle (1.2,1.2);
					\end{scope}
					\draw[dashed] (\c,-1.05) -- (\c,1.05);
					\draw[->,gray] (-1.15,0) -- (1.15,0);
					\draw[->,gray] (0,-1.15) -- (0,1.15);
					\node at (0,1.25) {};
					\node at (0,-1.32) {$S(B_X,x^*,\e)$};
				\end{scope}
				
				\begin{scope}[xshift=3.2cm]
					\draw[thick] (0,0) circle (1);
					\begin{scope}
						\clip (0,0) circle (1);
						\fill[blue!20] (\c,-1.2) rectangle (1.2,1.2);
						\fill[blue!20] (-1.2,-1.2) rectangle (-\c,1.2);
						\fill[blue!20] (-1.2,\c) rectangle (1.2,1.2);
						\fill[blue!20] (-1.2,-1.2) rectangle (1.2,-\c);
					\end{scope}
					\draw[dashed] (\c,-1.05) -- (\c,1.05);
					\draw[dashed] (-\c,-1.05) -- (-\c,1.05);
					\draw[dashed] (-1.05,\c) -- (1.05,\c);
					\draw[dashed] (-1.05,-\c) -- (1.05,-\c);
					\draw[->,gray] (-1.15,0) -- (1.15,0);
					\draw[->,gray] (0,-1.15) -- (0,1.15);
					\node at (0,1.25) {};
					\node at (0,-1.32) {$S_G(B_X,x^*,\e)$};
				\end{scope}
			\end{tikzpicture}
			\caption{On the left the classical slice of the Euclidean unit disk and on the right the $G$-slice for the group of quarter-turn rotations.}
			\label{fig:G-slice-R2}
		\end{figure}
		
	\end{example}
	
	In what follows, we will be using the following notation borrowed from the monograph \cite{Kadets2000}. Given a $G$-Banach space $X$, $x \in S_X$ and $\eps>0$, we set
	\begin{equation*}
		Q_G(x,\e):= \{y \in B_X: (\exists g \in G) \ \norm{x + gy} > 2 - \e\}.
	\end{equation*}

	We are now in a position to state the promised characterizations of the \(G\)-DPr. This theorem is the counterpart of \cite[Theorem 3.1.11]{KMRW} (for the DPr) and \cite[Proposition 2.1]{MO2004} (for the aDPr). Its proof requires inserting the $g$'s in the appropriate places. Nevertheless, we include it for the sake of completeness and also as a warm-up for readers who are not yet familiar with group-action language.

	\begin{proposition} \label{prop:characterization-G-DP} Let $X$ be a $G$-Banach space. The following are equivalent.
		\begin{enumerate}
			\item $X \in G$-DPr.
			\item For every $\e>0$, $x \in S_X$ and $G$-slice $S_G$ of $B_X$, there is $y \in S_G$ such that $\|x+y\| > 2 -\e$.
			\item For every $\e>0$, $x \in S_X$ and $G$-slice $S_G$ of $B_X$, there is $y \in S_G$ such that $\|x-y\| > 2 -\e$.
			\item For every $\e>0$ and $x \in S_X$, we have that $\overline{\co}_G(Q(x,\e)) = B_X$, where $Q(x,\e) = \{y \in B_X: \|x+y\| > 2 - \e\}$.
			\item Every $G$-slice $S_G$ of $B_X$ is such that $S_G \cap Q(x, \e) \not= \emptyset$ for every $\e>0$ and every $x \in S_X$.
			\item For every $x \in S_X$, $G$-slice $S_G$ of $B_X$ and $\delta>0$, there is a further $G$-slice $\tilde{S}_G$ of $B_X$ such that $\tilde{S}_G \subseteq S_G$ and $z \in Q_G(x, \delta)$ for every $z \in \tilde{S}_G$.
			\item For every $x \in S_X$ and $\eps > 0$, we have that $\clco_G(S_X \setminus (x + (2-\eps) B_X)) = B_X$.
		\end{enumerate}
	\end{proposition}
	
	\begin{proof} $(2) \implies (1)$: Let $T = x^* \otimes x$, $x \in S_X$, $x^* \in S_{X^*}$, be a~rank-one operator and $\eps > 0$.
		Denote $S_G = S_G(B_X, x^*, \eps)$.
		By (2), there is $y \in S_G$ such that $\norm{x + y} > 2 - \eps$. If $g \in G$ is such that $\re x^*(gy) > 1 - \eps$, then $\abs{x^*(gy)} \leq \norm{y} \leq 1$ and hence $\abs{1-x^*(gy)} \leq \sqrt{2\eps}$.
		Thus,
		\begin{align*}
			\norm{g^{-1} + T}
			&\geq \norm{(g^{-1} + T)(gy)}
			= \norm{y + x^*(gy) x}\\
			&\geq \norm{x + y} - \abs{1 - x^*(gy)} \norm{x}
			> 2 - \eps - \sqrt{2\eps}.
		\end{align*}
		\noindent
		$(1) \implies (2)$: Let $x \in S_X$, $\eps > 0$ and let $S_G = S_G(B_X, x^*, \delta)$ be a~$G$-slice of $B_X$.
		Define $T = x^* \otimes x$ and use (1) to find $g \in G$ such that $\norm{g + T} > 2 - \eps/2$. Define $\eta = \min\{\eps/2, \delta\}$.
		Then, there is $y_0 \in S_X$ such that
		\begin{equation*}
			\norm{gy_0 + Ty_0} = \norm{g y_0 + x^*(y_0) x} > 2 - \frac{\eps}{2} - \frac{\eta}{2}.
		\end{equation*}
		It follows that $\abs{x^*(y_0)} > 1 - \eta/2$.
		Let $\theta \in S_\K$ be such that $x^*(\theta y_0) = |x^*(y_0)|$ and define $y = \theta y_0$.
		Then, $\re x^*(y) = x^*(y) = |x^*(y_0)| > 1 - \frac{\eta}{2} > 1 - \delta$. In particular, $\abs{1-x^*(y)} < \eta/2$ and hence
		\begin{align*}
			\norm{x + gy}
			&\geq \norm{gy + x^*(y) x} - \abs{1 - x^*(y)} \norm{x} \\
			&> 2 - \frac{\eps}{2} - \frac{\eta}{2} - \frac{\eta}{2} \geq 2 - \eps.
		\end{align*}
		Since $y \in S(B_X, x^*, \delta)$, we have that $gy \in S_G(B_X, x^*, \delta)$.

		\noindent
		$(2) \iff (3)$: As in the classical case the equivalence here is obtained simply by replacing $x$ by $-x$.

		\noindent
		$(2) \implies (4)$: For the sake of contradiction, suppose that (2) holds but (4) does not.
		Then, there are $\e_0>0$ and $x_0 \in S_X$ such that $B_X \setminus \overline{\co}_G(Q(x_0, \e_0)) \not= \emptyset$. Take $y \in B_X \setminus \overline{\co}_G(Q(x_0, \e_0))$. By the Hahn-Banach separation theorem, we can get $x^* \in X^*$ and $\alpha > 0$ to be such that $y \in S(B_X, x^*, \alpha)$ and $\overline{\co}_G(Q(x_0, \e_0)) \subseteq X \setminus S(B_X, x^*, \alpha)$. In fact, we have that $\overline{\co}_G(Q(x_0, \e_0)) \subseteq X \setminus S_G(B_X, x^*, \alpha)$; indeed, if there were $z \in \overline{\co}_G(Q(x_0, \eps_0)) \cap S_G(B_X, x^*, \alpha)$, then there would exists $g \in G$ such that $gz \in S(B_X, x^*, \alpha) \cap \overline{\co}_G(Q(x_0, \eps_0))$, which is a contradiction. By hypothesis, there exists $z \in S_G$ such that $\|x+z\| > 2 - \alpha$, which means that 
		\begin{equation*}
			z \in Q(x_0, \alpha) \subseteq \overline{\co}_G(Q(x_0, \alpha)) \subseteq X \setminus S_G(B_X, x^*, \alpha),
		\end{equation*}
		which is a contradiction once again.
		
		\noindent
		$(4) \implies (2)$: Suppose that (4) holds. Then, for every $x \in S_X$ and for every slice $S = S(B_X, x^*, \delta)$ of $B_X$, we have that $Q_G(x, \e) \cap S(B_X, x^*, \delta) \not= \emptyset$. This means that, for every $x \in S_X$ and for every slice $S = S(B_X, x^*, \delta)$, there are $g \in G$ and $y \in S$ such that $\|x+gy\| > 2 - \e$. This is the same as saying that there exists $z \in S_G = S_G(B_X, x^*, \delta)$ such that $\|x+z\| > 2 - \e$.

		\noindent
		$(5) \iff (2)$: The equivalence is immediate from the definitions.

		\noindent
		$(5) \implies (6)$:
		Fix $x \in S_X$. Let $\delta > 0$. Assume that $S_G = S_G(B_X, x_0^*, \e_0)$ with $\|x_0^*\| = 1$ and $\e_0 \in (0,1)$. As in \cite{KMRW}, consider $\e>0$ such that
		\begin{equation*}
			\e \in \left( 0, \min \left\{ \frac{\delta}{3}, \frac{\e_0}{2} \right\} \right).
		\end{equation*}
		By hypothesis applied to the smaller $G$-slice $S_G(B_X, x_0^*, \e)$, we have that the intersection $S_G(B_X, x_0^*, \e) \cap Q_G(x, \e)$ is nonempty. This means that there exists $x_0 \in S_X$ in this intersection and satisfying without loss of generality that $\re x_0^*(x_0) > 1 - \e$. Also, for being an element of this intersection, there exists $g_0 \in G$ such that $\|x + g_0 x_0\| > 2 - \e$. By \cite[Lemma 2.6.8]{KMRW} applied to $x \in S_X$, $g_0 x_0 \in S_X$ and $\e>0$, there exists a functional $x^* \in S_{X^*}$ such that $\re x^*(x) > 1 - \e$ and $\re x^*(g_0 x_0) > 1- \e$. Now, let us define the functional $x_2^* \in S_{X^*}$ by $x_2^*(y):= x^*(g_0 y)$ for every $y \in X$ and let us consider the $G$-slice of $B_X$ given by 
		\begin{align*}
			\tilde{S}_G
			&:= \{ z \in B_X: (\exists g \in G) \ \re (x_0^* + x_2^*)(gz) > 2 - 2 \e \}\\
			&= S_G\left(B_X,\frac{x_0^*+x_2^*}{\norm{x_0^*+x_2^*}}, 1 - \frac{2-2\eps}{\norm{x_0^* + x_2^*}}\right).
		\end{align*}
		Let us show that $\tilde{S}_G$ is nonempty and since $\tilde{S}_G$ is a halfspace of $B_X$, we will have that $\tilde{S}_G$ is a nonempty $G$-slice. Indeed, to see that $x_0 \in \tilde{S}_G$ we calculate
		\begin{eqnarray*}
			\re (x_0^*+x_2^*)(x_0)
			&=& \re x_0^*(x_0) + x_2^*(x_0) \\
			&=& \re x_0^*(x_0) + x^*(g_0 x_0) > 2 - 2\e.
		\end{eqnarray*}
		We have to prove now that $\tilde{S}_G \subseteq S_G$ and that, for every $z \in \tilde{S}_G$, we have that $z \in Q_G(x,\delta)$. Indeed, let $z \in \tilde{S}_G$. Then, we have that there exists $g \in G$ such that $\re (x_0^*+x_2^*)(gz) > 2 - 2 \e$. Thus,
		\begin{equation*}
			\re  x_0^*(gz) > 2 - 2\e - x_2^*(gz) \geq 1 - 2\e > 1 - \e_0,
		\end{equation*}
		that is, $z \in S_G$ and, therefore, $\tilde{S}_G \subseteq S_G$.  On the other hand,
		\begin{equation*}
			\re x_2^*(gz) > 2- 2\e - x_0(gz) \geq 1 - 2\e 
		\end{equation*}
		and thus
		\begin{eqnarray*}
			\|x+g_0gz\| = \|g_0^{-1}x + gz\| &\geq& \re x_2^*(g_0^{-1}x + gz) \\
			&=& \re x^*(g_0g_0^{-1}x) + \re x_2^*(gz) \\
			&>&2-3\e > 2- \delta.
		\end{eqnarray*}
		
		\noindent
		$(6) \implies (5)$: This implication is straightforward using that $G$-slices are $G$-invariant.

		\noindent
		$(4) \iff (7)$: Let $y \in S_X$ and $\eps > 0$. Then we have that $y \in S_X \setminus (x + (2-\eps) B_X)$ is equivalent to $\norm{x - y} > 2 - \eps$, which in turn is equivalent to $y \in Q(-x, \eps)$.
	\end{proof}

	We also have the analogue of the fact that Daugavet property is inherited from dual spaces endowed with the canonical dual action. Notice that the action defined below does not need to be continuous.

	\begin{proposition} \label{prop:inheritanceFromDual}
		Let $X$ be a $G$-Banach space. Equip $X^*$ with the canonical dual action, that is, $(gx^*)(x) = x^*(g^{-1}x)$ for every $g \in G$, $x \in X$ and $x^* \in X^*$. If $X^*$ has the $G$-DPr with respect to this action, then so does $X$.
	\end{proposition}
	\begin{proof}
		Let $T:X \to X$ be a rank-one bounded linear operator. Consider the adjoint operator $T^*: X^* \to X^*$. Since $T$ is a rank-one operator, $T^*$ is also a rank-one operator, and we have $\norm{T^*} = \norm{T}$.
		By assumption, $X^*$ has the $G$-Daugavet property with respect to the dual action. Therefore, $T^*$ satisfies the $G$-Daugavet equation on $X^*$. By \Cref{fact:G-DP-equation-commutes}, this allows us to write
		$$
		\sup_{h \in G} \norm{\id + T^* \circ h} = 1 + \norm{T^*}
		$$
		where we identify $h \in G$ with the linear operator on $X^*$ given by the dual action. We obtain that
		$$
		\norm{\id + g \circ T} = \norm{(\id + g \circ T)^*} = \norm{\id + T^* \circ g^*}.
		$$
		Let us analyze the adjoint of the action of $g$ on $X$. For any $x \in X$ and $x^* \in X^*$, by the definition of the dual action we have $(g^{-1}x^*)(x) = x^*(g x) = (g^* x^*)(x)$. This shows that the adjoint operator $g^*$ coincides exactly with the dual action of $g^{-1}$ on $X^*$. Substituting this into our norm equality, it yields
		$$
		\norm{\id + g \circ T} = \norm{\id + T^* \circ g^{-1}}.
		$$
		Taking the supremum over all $g \in G$ and using the fact that the map $g \mapsto g^{-1}$ is a bijection on the group $G$, we get
		$$
		\begin{aligned}
			\sup_{g \in G} \norm{\id + g \circ T} &= \sup_{g \in G} \norm{\id + T^* \circ g^{-1}} \\
			&= \sup_{h \in G} \norm{\id + T^* \circ h} 
			= 1 + \norm{T^*} 
			= 1 + \norm{T}.
		\end{aligned}
		$$
		Thus, $X$ has the $G$-Daugavet property.
	\end{proof}
	
	The use of the dual action in \Cref{prop:inheritanceFromDual} is essential. Indeed, the conclusion may fail if one equips $X^*$ with an arbitrary action of $G$ unrelated to the given action on $X$. For a simple example, let $X=\R^2$ with the Euclidean norm and let $G=\Z$. Let $G$ act trivially on $X$, so that $n\cdot x=x$ for every $n\in\Z$ and $x\in X$. Then the $G$-DPr on $X$ is just the classical DPr, since for every rank-one operator $T$ one has $\sup_{g\in G}\|\id_X+g\circ T\|=\|\id_X+T\|$. As $\R^2$ does not have the DPr, it follows that $X$ fails the $G$-DPr. On the other hand, identify $X^*$ with $\R^2$ and let $G$ act on $X^*$ by irrational rotations, that is, let the generator $1\in\Z$ act as the rotation $R_\theta$ through an angle $\theta$ such that $\theta/\pi\notin\mathbb Q$. Then the orbit of every point of the unit circle is dense, so the action is almost transitive. Since $\R^2$ is LUR, Theorem \ref{theorem:GDPrOnUniformlyConvexSpace} below yields that $X^*$ has the $G$-DPr with respect to this action. Thus, $X^*$ may have the $G$-DPr for a non-dual action of $G$ while $X$ fails the $G$-DPr for its original action.

	
	
	\section{Examples exhibiting non-classical behavior} \label{section:far-classical}

	As mentioned in the introduction, the classical Daugavet property has strong consequences for the isomorphic structure of the underlying Banach spaces. Although we have mentioned them in the introduction, we next recall those consequences that will be most relevant for our purposes. They can be derived, for instance, from \cite{Kadets2000}. If $X$ has the Daugavet property, then the following assertions hold true.
	
	\begin{itemize}
		\item $X$ cannot be embedded into a Banach space with an unconditional basis.
		\item $X$ does not have the RNP and, consequently, $X$ cannot be reflexive.
		\item $X$ contains many isomorphic copies of $\ell_1$ and so $X$ cannot be Asplund.
	\end{itemize}
	
	It is therefore natural to ask whether one can impose conditions on the acting group $G$ under which the $G$-DPr would yield, as in the classical setting, non-reflexivity, failure of the Radon-Nikodým property, non-Asplundness and the absence of an unconditional Schauder basis. As we shall see, however, the results that follow show that such implications cannot be expected in full generality. Depending on the action under consideration, one may in fact obtain results which point in the opposite direction from those associated with the classical DPr.
	
	\subsection{Transitivity} \label{section:transitivity} We say that the action of $G$ on a Banach space is {\it convex-transitive} if every point of its unit sphere is a {\it big point} with respect to this action, that is, for every $x \in S_X$, we have that $\overline{\co}(G \cdot x) = B_X$. For more information on convex-transitive Banach spaces and big points as well as almost-transitive Banach spaces we refer the reader to \cite{BG-RP99,BG-RP02}.

	By using item (4) of \Cref{prop:characterization-G-DP}, we immediately get the following result.
	\begin{proposition} \label{cor:convexTranstiveImpliesGDP}
		Let $X$ be a~Banach space equipped with a convex-transitive action of a~group $G$. Then, $X$ has the $G$-DPr.
	\end{proposition}
	\begin{proof}
		For every $x \in S_X$ and $\eps > 0$, we have that
		\begin{equation*}
			B_X = \overline{\co}(G \cdot x)
			\subseteq \overline{\co}(G \cdot Q(x, \eps))
			= \overline{\co}_{G}(Q(x, \eps))
			\subseteq B_X.
		\end{equation*}
		This shows that $X$ has the $G$-DPr.
	\end{proof}
	
	One might wonder if the converse of \Cref{cor:convexTranstiveImpliesGDP} above is also true but it is not the case as we can see in the next example.

	\begin{example} The real Banach space $X=\ell_\infty^2$ has the aDPr (see, for instance, \cite[Proposition 1.4.8]{KMRW}), but no action of a group $G$ on $X$ by surjective linear isometries is convex-transitive. For completeness, we include a proof of the latter assertion, although it is well-known to specialists. It is enough to show that $X$ is not convex-transitive under the action of its full isometry group $\Iso(X)$. Since $B_X$ has exactly four extreme points and linear isometries preserve extreme points, every element of $\Iso(X)$ induces a permutation of these four points. Consequently, $\Iso(X)$ is finite. Now, let $x \in S_X$ be a point which is not an extreme point of $B_X$. Then its orbit $\Iso(X)\cdot x$ is finite and therefore $\overline{\co}(\Iso(X)\cdot x)=\co(\Iso(X)\cdot x)$. Since $x$ is not extreme and linear isometries preserve non-extreme points, every point in the orbit $\Iso(X)\cdot x$ is again non-extreme. Hence $\co(\Iso(X)\cdot x)$ cannot contain any extreme point of $B_X$ and so it cannot coincide with $B_X$.
	\end{example}

	We shall see next that the $G$-DPr is closely related to strong transitivity assumptions when the underlying space is locally uniformly rotund. More precisely, if the action of $G$ on $X$ is almost transitive, then $X$ has the $G$-DPr.  Conversely, if an LUR $G$-Banach space has the $G$-DPr, then the action of $G$ must be almost transitive (see Theorem \ref{theorem:GDPrOnUniformlyConvexSpace} below). Recall that the action of $G$ on $X$ is said to be \emph{almost transitive} if, for every $x \in S_X$, the orbit $G \cdot x$ is dense in $S_X$. Every almost-transitive action is convex-transitive, whereas $C(K)$, with $K$ the Cantor set, is a classical example of a convex-transitive Banach space which is not almost transitive (see \cite{PR1962}). We also recall that a Banach space $X$ is said to be \emph{locally uniformly rotund} (LUR, for short) if, whenever $x \in X$ and $(x_k)$ is a sequence in $X$ such that $\|x_k\|\leq \|x\|$ for every $k\in\N$ and $\lim_k \left\| \frac{x_k+x}{2} \right\| = \lim_k \|x_k\| = \|x\|$, it follows that $\lim_k \|x_k-x\|=0$.
	
	\begin{theorem} \label{theorem:GDPrOnUniformlyConvexSpace} Let $X$ be a $G$-Banach space and assume that $X$ is LUR. Then, $X$ has the $G$-DPr if and only if the action of $G$ on $X$ is almost transitive.
	\end{theorem}
	\begin{proof}
		If the action is almost transitive, then it is also convex-transitive and so, by \Cref{cor:convexTranstiveImpliesGDP}, $X$ has the $G$-DPr.
		
		Conversely, suppose that $X$ has the $G$-DPr. We will show that the action is almost transitive. Choose $x,y \in S_X$ and $\eps > 0$ arbitrarily. We need to show that $Gx \cap B(y, \eps)$ is non-empty. Using the Hahn-Banach theorem, we can find $x^* \in S_{X^*}$ such that $x^*(x) = 1$. Consider the operator $T = x^* \otimes y$.
		Since $X$ has the $G$-DPr, for any $n \in \N$, there are $g_n \in G$ and $z_n' \in S_X$ such that
		\begin{equation*}
			\norm{g_n z_n' + x^*(z_n')y} = \norm{g_n z_n' + T z_n'} > 2 - \frac{1}{n}.
		\end{equation*}
		It follows that 
		\begin{equation*} 
			\abs{x^*(z_n')} > \norm{g_n z_n' + x^*(z_n')y} - \norm{g_n z_n'} > 1 - \frac{1}{n}.
		\end{equation*} 
		Find $\theta_n \in \T$ such that $x^*(\theta_n z_n') = |x^*(z_n')|$ and define $z_n = \theta_n z_n'$. Then, 
		\begin{equation*} 
			\re x^*(z_n) = x^*(z_n) = |x^*(z_n')| > 1 - \frac{1}{n}
		\end{equation*}     
		implying that $\abs{1 - x^*(z_n)} < \frac{1}{n}$. Moreover, we have
		\begin{equation*}
			\norm{g_n z_n + x^*(z_n) y} = \norm{\theta_n(g_n z_n' + x^*(z_n')y)} > 2 - \frac{1}{n}.
		\end{equation*}
		So we see that for every $n \in \N$ one has
		\begin{eqnarray*}
			\norm{g_n z_n + y} &\geq& \norm{g_n z_n + x^*(z_n)y} - \norm{x^*(z_n)y - y} \\
			&>& 2 - \frac{1}{n} - \abs{x^*(z_n) - 1}\norm{y} \\
			&\geq& 2 - \frac{1}{n} - \frac{1}{n}
			= 2 - \frac{2}{n}
		\end{eqnarray*}
		yielding the convergence
		\begin{equation} \label{eq:GDPOnLUR:gzPLUSy}
			\lim_{n \to \infty}\norm{g_n z_n + y} = 2.
		\end{equation}
		Moreover, from the previous estimates we also get
		\begin{equation*}
			\abs{x^*(x+z_n)}
			= x^*(x) + x^*(z_n)
			> 1 + \left(1 - \frac{1}{n}\right)
			= 2 - \frac{1}{n}.
		\end{equation*}
		As $\norm{x^*} = 1$, we deduce that
		\begin{equation} \label{eq:GDPOnLUR:xPLUSz}
			\lim_{n \to \infty}\norm{x + z_n} = 2.
		\end{equation}
		As $X$ is LUR, we deduce from \eqref{eq:GDPOnLUR:gzPLUSy} and \eqref{eq:GDPOnLUR:xPLUSz} that
		\begin{equation*}
			\lim_{n \to \infty}\norm{g_n z_n - y} = 0
			\ \ \ \text{and} \ \ \ 
			\lim_{n \to \infty}\norm{z_n - x} = 0.
		\end{equation*}
		Now, choose $n$ such that $\norm{g_n z_n - y} < \frac{\eps}{2}$ and $\norm{z_n - x} < \frac{\eps}{2}$.
		From this, however, as $G$ acts by isometries, it follows that
		\begin{equation*}
			\norm{g_n x - y}
			\leq \norm{g_n x - g_n z_n} + \norm{g_n z_n - y}
			< \frac{\eps}{2} + \frac{\eps}{2}
			= \eps
		\end{equation*}
		and hence $g_n x \in B(y, \eps)$.
	\end{proof}

	As a consequence of \Cref{theorem:GDPrOnUniformlyConvexSpace}, there are infinite dimensional reflexive Banach spaces $X$ satisfying the $\iso(X)$-DPr as we can see below. In fact, there are Banach spaces which fail the $G$-DPr for every subgroup $G\leq \Iso(X)$ and hence fail even the weakest possible group version, namely the $\Iso(X)$-DPr (see item (c) below).

	\begin{corollary} \label{cor:Lp-ellp-examples}
		Let $1<p<\infty$ be given.
		\begin{itemize}
			\item[(a)] $L^p[0,1]$ has the $\Iso(L^p[0,1])$-Daugavet property.
			
			\item[(b)]  $L^p[0,1]$ fails the $\Iso^+(L^p[0,1])$-Daugavet property.
			
			\item[(c)] If $p\not=2$ and $G\leq \Iso(\ell_p)$, then \(\ell_p\) fails the $G$-Daugavet property.
		\end{itemize}
	\end{corollary}
	
	\begin{proof} It is well-known that the action of $\Iso(L^p[0,1])$ on $L^p[0,1]$ is almost transitive (see \cite[Theorem~9.6.3]{RolewiczStefan1985Mls}). Therefore, \Cref{theorem:GDPrOnUniformlyConvexSpace} yields that $L^p[0,1]$ satisfies the $\Iso(L^p[0,1])$-DPr. This shows (a).
		
		For item (b), let $U\in \Iso^+(L^p[0,1])$ and denote by $\mathbbm{1}$ the constant $1$ function. Since $U$ is positive, we have $U \mathbbm{1} \geq 0$ almost everywhere. Hence, $\Iso^+(L^p[0,1])\cdot \mathbbm{1} \subseteq \{f\in S_{L^p[0,1]}: f\ge 0 \text{ a.e.}\}$. In particular, for every $f$ in the orbit of $\mathbbm{1}$, we have that 
		\begin{equation*} 
			\|f-(-\mathbbm{1})\|_p=\|f+\mathbbm{1}\|_p \ge \|\mathbbm{1}\|_p = 1,
		\end{equation*} 
		so the orbit of $\mathbbm{1}$ cannot be dense in $S_{L^p[0,1]}$. Therefore the action of $\Iso^+(L^p[0,1])$ is not almost transitive and \Cref{theorem:GDPrOnUniformlyConvexSpace} implies that $L^p[0,1]$ does not have the $\Iso^+(L^p[0,1])$-DPr.
		
		Finally, let $G\leq \Iso(\ell_p)$. By the Banach-Lamperti theorem (see, for instance, \cite[Theorem 3.2.5]{FlemingJamison2003}), every surjective linear isometry of $\ell_p$, $p\neq 2$, is given by a permutation of the canonical basis together with coordinatewise sign changes. Hence the $\Iso(\ell_p)$-orbit of $e_1$ is exactly $\{\pm e_n : n\in\N\}$. In particular, the $G$-orbit of $e_1$ is contained in $\{\pm e_n : n\in \N \}$ and therefore it cannot be dense in $S_{\ell_p}$. Thus the action of $G$ on $\ell_p$ is not almost transitive. A new application of  \Cref{theorem:GDPrOnUniformlyConvexSpace} gives item (c).
	\end{proof}

	\begin{remark}
		Let us note that from \Cref{theorem:GDPrOnUniformlyConvexSpace} it follows that if $X$ is a LUR Banach space then an action of a group by linear isometries on $X$ is almost transitive if and only if it is convex transitive. The implication from almost transitive to convex transitive is a direct consequence of the definitions. The other way, if the action is convex transitive, then by \Cref{cor:convexTranstiveImpliesGDP}, $X$ has the $G$-DPr with respect to this action and then by \Cref{theorem:GDPrOnUniformlyConvexSpace} the action must be almost transitive. This result is certainly not new and even stronger results can be found in \cite{BECERRAGUERRERO2009108}.
	\end{remark}

	We have the following consequence of \Cref{theorem:GDPrOnUniformlyConvexSpace} for strictly convex finite dimensional spaces with the $G$-Daugavet property.
	
	\begin{corollary}\label{corollary:finite-dimensional-strictly-convex} Let $X$ be a finite-dimensional strictly convex and let $G \le \iso(X)$. Denote by $\overline{G}$ the closure of $G$ in $\iso(X)$. The following statements are equivalent.
		\begin{itemize}
			\item[(1)] $X$ has the $G$-Daugavet property.
			\item[(2)] The action of $G$ on $S_X$ is almost transitive.
			\item[(3)] The action of $\overline{G}$ on $S_X$ is transitive.
			\item[(4)] $X$ is linearly isometric to a finite dimensional Hilbert space and $\overline{G}$ acts transitively on its unit sphere.
		\end{itemize}
		In particular, if $G$ is closed, then $X$ has the $G$-Daugavet property if and only if $X$ is Euclidean and $G$ acts transitively on $S_X$.
	\end{corollary}
	
	\begin{proof} Since $X$ is finite-dimensional and strictly convex, it is LUR. Thus, (1) $\Leftrightarrow$ (2) follows immediately from \Cref{theorem:GDPrOnUniformlyConvexSpace}. The equivalence (3) $\Leftrightarrow$ (4) is well-known, it is Mazur's rotation problem for finite-dimensional spaces. Now let us prove (2) $\Leftrightarrow$ (3). Since $X$ is finite-dimensional, $\Iso(X)$ is compact and so is $\overline{G}$, being a closed subset of $\Iso(X)$. For a fixed $x\in S_X$, the map $\Phi_x:\overline{G}\longrightarrow S_X$ given by $\Phi_x(g)=g\cdot x$ is continuous. Hence $\overline{G}\cdot x=\Phi_x(\overline{G})$ is compact and therefore closed. We claim that $\overline{G}\cdot x=\overline{G\cdot x}$. Indeed, since $G\subseteq \overline{G}$, we have $G\cdot x\subseteq \overline{G}\cdot x$. As $\overline{G}\cdot x$ is closed, it follows that $\overline{G\cdot x}\subseteq \overline{G}\cdot x$. For the reverse inclusion, take $h\in \overline{G}$. Then there exists a net $(g_\alpha)\subseteq G$ such that $g_\alpha\to h$. By the continuity of $\Phi_x$, we get $g_\alpha x\longrightarrow hx$. Thus $hx\in \overline{G\cdot x}$. Since $h\in \overline{G}$ was arbitrary, we obtain $\overline{G}\cdot x\subseteq \overline{G\cdot x}$. Therefore, $\overline{G}\cdot x=\overline{G\cdot x}$. Now, $G$ acts almost transitively on $S_X$ if and only if $G\cdot x$ is dense in $S_X$ for every $x\in S_X$, that is, $\overline{G\cdot x}=S_X$ for every $x\in S_X$. By the equality we have obtained before, this is equivalent to $\overline{G}\cdot x=S_X$ for every $x\in S_X$, which is precisely the meaning of $\overline{G}$ acting transitively on $S_X$. This proves (2) $\Leftrightarrow$ (3). Finally, if $G$ is closed, then $\overline{G}=G$. Therefore, by the equivalences already proved, $X$ has the $G$-Daugavet property if and only if $X$ is Euclidean and $G$ acts transitively on the unit sphere $S_X$.
	\end{proof}

	\subsection{In the opposite direction of the DPr and aDPr} Although we have seen this phenomenon already happening in the previous section, in the present one, we will be dealing with examples that show that the $G$-Daugavet property gives much more freedom than the DPr and then aDPr thanks to the groups and their actions on the involved Banach spaces (see once again \Cref{theorem:GDPrOnUniformlyConvexSpace}, \Cref{cor:Lp-ellp-examples} and \Cref{corollary:finite-dimensional-strictly-convex}). In particular, these examples show that compactness or non-compactness of $G$ is not, in itself, decisive for the $G$-Daugavet property. Indeed, both compact and non-compact groups may give rise to the property and both may also fail to do so.

	It is known that every reflexive real Banach space with the aDPr must be finite-dimensional (see \cite{LMP1999}). The next example shows that no analogous conclusion holds for the $G$-Daugavet property in full generality. Although this is already implied by \Cref{cor:Lp-ellp-examples}, the example below is more informative. Indeed, it exhibits an infinite-dimensional reflexive $G$-Banach space whose $G$-DPr is non-trivial in the sense that it is not merely a consequence of the transitivity of the action. This also illustrates that, in the group setting, the geometry of the underlying space may also play a substantial role.

	\begin{example} \label{inf-dim-reflexive-with-GDPr-not-almost-transitive}
		Let $Y$ be a~reflexive infinite-dimensional Banach space such that $\Iso(Y)$ acts (almost) transitively on $Y$ (for instance, $Y = \ell_2$). Define $X = Y \oplus_1 Y$ and let $G := \Z_2 \times \Iso(Y)$ act on $X$ by
		\begin{equation*}
			(a, T) (x, y)
			= \begin{cases}
				(T x, T y), & a = 1,\\
				(T y, T x), & a = -1
			\end{cases}
		\end{equation*}
		for every $(x,y) \in X$ and $(a,T) \in G$. Then, the following holds true.
		\begin{itemize}
			\item[(1)] $X$ is a reflexive infinite-dimensional $G$-Banach space.
			\item[(2)] The action of $G$ on $X$ is not almost transitive.
			\item[(3)] $X$ has the $G$-DPr.
		\end{itemize} 
	\end{example}
	
	\begin{proof} We prove directly the case in which $\Iso(Y)$ acts almost transitively on $S_Y$ since the transitive case is a particular case. Since $Y$ is reflexive, so is $X$. To see that the action of $G$ on $X$ is not almost transitive, consider $y = (y_1, 0) \in S_X$ for some $y_1 \in S_Y$. Then, the orbit of $y$ is
		\begin{equation*}
			G \cdot y = \{ (T y_1, 0), (0, T y_1) : T \in \Iso(Y) \}.
		\end{equation*}
		Since $\Iso(Y)$ acts almost transitively on $S_Y$, we have that
		\begin{equation*}
			\overline{G \cdot y} = S_Y \times \{0\} \cup \{0\} \times S_Y
		\end{equation*}
		and this set is not dense in $S_X$. Therefore, the action is not almost transitive. Let us show that $X$ has the $G$-DPr.
		Let $T: X \rightarrow X$ be a rank-one operator with $\norm{T} = 1$.
		Then, there are $x^* = (x_1^*, x_2^*) \in S_{X^*}$ and $x = (x_1, x_2) \in S_X$ such that
		\begin{equation*}
			T(y_1, y_2) = x^*(y_1, y_2) x = (x_1^*(y_1) + x_2^*(y_2)) (x_1, x_2)
		\end{equation*}
		for every $(y_1, y_2) \in X$. We have that $\norm{x^*} = \max{\norm{x_1^*}, \norm{x_2^*}} = 1$. Assume that $\norm{x_1^*} = 1$ (the case $\norm{x_2^*} = 1$ is analogous).
		Then, by the reflexivity of $Y$, there is $z \in S_Y$ such that $x_1^*(z) = 1$.
		We have two cases we need to treat.

		\noindent
		\textbf{Case 1:} Suppose that $\norm{x_1} \neq 0$.
		Let $\eta>0$. Since $\Iso(Y)$ acts almost transitively on $S_Y$, we can find $g \in \Iso(Y)$ such that
		\begin{equation*}
			\left\|g z - \frac{x_1}{\norm{x_1}}\right\| < \eta.
		\end{equation*}
		Then, we have
		\begin{align*}
			\norm{(1, g) + T}
			&\geq \norm{((1, g) + T)(z, 0)}\\
			&= \norm{\left(gz, 0\right) + (x_1^*(z) + x_2^*(0)) (x_1, x_2)} \\
			&= \norm{\left(gz + x_1, x_2\right)}\\
			&= \norm{gz+x_1}+\norm{x_2}\\
			&\geq
			\left\|\frac{x_1}{\norm{x_1}}+x_1\right\|
			-
			\left\|gz-\frac{x_1}{\norm{x_1}}\right\|
			+
			\norm{x_2} >
			1+\norm{x_1}-\eta+\norm{x_2}
			=
			2-\eta.
		\end{align*}
		This gives us that $\sup_{g\in G}\norm{g+T}=2$.
		
		\noindent
		\textbf{Case 2:} Suppose that $\norm{x_1} = 0$.
		In this case, we have $\norm{x_2}=1$. Let $\eta>0$. Since $\Iso(Y)$ acts almost transitively on $S_Y$, we can find $g\in \Iso(Y)$ such that $\|gz-x_2\|<\eta$. Then, we have
		\begin{eqnarray*}
			\norm{(-1, g) + T}
			&\geq& \norm{((-1, g) + T)(z, 0)} \\
			&=& \norm{\left(0, gz\right) + (x_1^*(z) + x_2^*(0)) (x_1, x_2)} \\
			&=& \norm{\left(0, gz+x_2\right)} \\
			&=& \norm{gz+x_2}\\
			&\geq& 2-\eta.
		\end{eqnarray*}
		This gives once again that $\sup_{g\in G}\norm{g+T}=2$.
	\end{proof}

	Let $K$ be a compact Hausdorff space, let $X$ be a Banach space and let $(\Omega,\Sigma,\mu)$ be a measure space with $\mu$ positive. It is known that the Banach space $C(K,X)$ has the Daugavet property if and only if either $K$ is perfect or $X$ has the Daugavet property (see \cite[Remark 6]{MartinPaya2000}). Likewise, $L^1(\mu,X)$ has the Daugavet property if and only if either $\mu$ is atomless or $X$ has the Daugavet property (see \cite[Remark 9]{MartinPaya2000}). The analogous statements for the aDPr also hold true and it can be found in \cite[Theorem 3.4]{MO2004}. In Section~\ref{section:L1-and-CK}, we will show that these equivalences remain valid under suitable restrictions on the group action and, as a consequence, we will recover both of the aforementioned results. Notice first that if $\mu$ is atomless, then $L^1(\mu,X)$ has the Daugavet property and hence, as a consequence, it has the $G$-Daugavet property (see Corollary \ref{DPr-implies-G-DPr}), regardless of the action of $G$ on $L^1(\mu,X)$. Similarly, if $K$ is perfect, then $C(K,X)$ has the Daugavet property and therefore also the $G$-Daugavet property, independently of the action considered on $C(K,X)$. The difficulty arises in the converse direction: in the absence of a natural compatibility\footnote{Here, compatibility means that the action on the function space is obtained by applying to each value of the function the same operator $g$ which acts on $X$, that is, $(gf)(t)=g(f(t))$ for $\mu$-almost every $t$ in the case of $L^1(\mu,X)$ and $(gf)(s)=g(f(s))$ for every $s\in K$ in the case of $C(K,X)$.
	} between the action on the ambient function space and the action on $X$, one cannot expect the corresponding characterization results to remain valid.
	
	\begin{example} \label{examples-L1-C(K)-independent-actions}  Let $X$ be the space $\R^2$ equipped with the norm
		\begin{equation*}
			\norm{(a,b)}_X = \max \left\{ \abs{a}, \abs{b}, \frac{\abs{a} + \abs{b}}{\sqrt{2}} \right\}
		\end{equation*}
		for every $(a,b) \in \R^2$ (see Figure \ref{fig:ballForVecValL1example}). Let $G = \Z_8$ and equip $X$ with the action of $G$ by rotations by multiples of $\pi/4$.
		Denote $Y = \ell_1(G, X)$, where we take the counting measure on $G$ and equip $Y$ with the action of $G$ by translations\footnote{The action of $G$ on $Y$ by translations means that $G$ shifts the coordinates, that is, $(hf)(i) = f(i-h)$ for every $i,h \in \Z_8$, where $i-h$ is taken modulo 8. This action is by linear isometries as it only permutes the eight elements of the group. Notice also that this action on $Y$ is not the action induced pointwise from the rotation action on $X$.}.
		Then, the following holds true.
		\begin{enumerate}
			\item $X$ has the $G$-Daugavet property.
			\item $Y$ does not have the $G$-Daugavet property.
		\end{enumerate}
		
		\begin{figure}
			\begin{tikzpicture}[scale=1.9]
				\def\a{0.4142}
				
				\fill[blue!15] (0,0) -- (\a, 1) -- (-\a, 1) -- cycle;
				
				\draw[->, gray!80] (-1.5, 0) -- (1.5, 0) node[right, text=black] {$x$};
				\draw[->, gray!80] (0, -1.5) -- (0, 1.5) node[above, text=black] {$y$};
				
				\draw[thick, black] 
				(1, \a) -- (\a, 1) -- (-\a, 1) -- (-1, \a) -- 
				(-1, -\a) -- (-\a, -1) -- (\a, -1) -- (1, -\a) -- cycle;
				
				\draw[thick, blue] (0,0) -- (\a, 1);
				\draw[thick, blue] (0,0) -- (-\a, 1);
				\draw[thick, blue] (\a, 1) -- (-\a, 1); 
				
				\filldraw (1, \a) circle (0.9pt) node[right] {$(1, \sqrt{2}-1)$};
				\filldraw (\a, 1) circle (0.9pt) node[above right] {$(\sqrt{2}-1, 1)$};
				\filldraw (-\a, 1) circle (0.9pt) node[above left] {$(1-\sqrt{2}, 1)$};
				\filldraw (-1, \a) circle (0.9pt) node[left] {$(-1, \sqrt{2}-1)$};
				\filldraw (-1, -\a) circle (0.9pt) node[left] {$(-1, 1-\sqrt{2})$};
				\filldraw (-\a, -1) circle (0.9pt) node[below left] {$(1-\sqrt{2}, -1)$};
				\filldraw (\a, -1) circle (0.9pt) node[below right] {$(\sqrt{2}-1, -1)$};
				\filldraw (1, -\a) circle (0.9pt) node[right] {$(1, 1-\sqrt{2})$};
				
				\filldraw (0,0) circle (0.9pt) node[below right] {$(0,0)$};
			\end{tikzpicture}
			\caption{The unit ball in $X$}
			\label{fig:ballForVecValL1example}
		\end{figure}
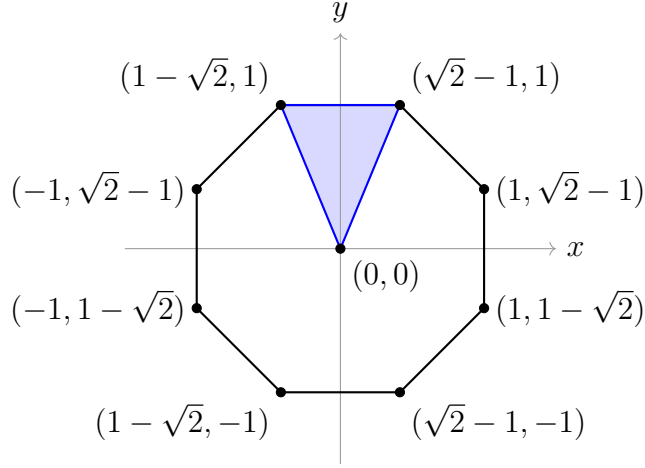
	\end{example}

	\begin{proof}
		Indeed, to see that $X$ has the $G$-DPr, let $x \in S_X$. If $x$ is one of the eight extreme points of $S_X$, then $x \in Q(x, \eps)$ for all $\eps >0$. If $x$ is not an extreme point, then $Q(x,\eps)$ contains the entire edge of the sphere containing $x$ and thus $Q(x, \eps)$ contains at least two extreme point for any $\eps > 0$. In both cases, $Q(x, \eps)$ contains at least one extreme point and since the action on $X$ is transitive for the set of extreme points of $B_X$, we have $\clco Q_G(x, \eps) = B_X$ for any $x \in S_X$ and $\eps > 0$.

		To see item (2), let us assume that $Y$ has the $G$-DPr. We know that $Y^* = \ell_\infty(G, X^*)$. It is easy to verify that the element $(0, 1)$ represents a~norm one functional on $X$. Hence we can define $y \in S_Y$ to be the constant function $(1/8, 0)$ and $y^* \in S_{Y^*}$ to be the constant function $(0,1)$. Then, we have that $T: Y \to Y$ defined by $T = y^* \otimes y$ is a rank-one linear operator on $Y$ with $\norm{T} = 1$.
		Since we assume $Y$ to have the $G$-DPr, we have that
		\begin{equation*}
			\sup_{g \in G} \norm{\id + T \circ g} = 2.
		\end{equation*}
		Notice also that $y^*$ is $G$-invariant. Indeed, when we say that $y^* \in S_{Y^*}$ is the constant function $(0,1)$ we mean the following. Let $\varphi \in X^*$ defined by $\varphi(a,b) = b$. Then, $y^* = (\varphi, \varphi, \ldots, \varphi) \in \ell_{\infty}(G, X^*)$. This implies that, for $z = ((a_i, b_i))_{i=0}^7 \in Y$, we have that $y^*(z) = \sum_{i=0}^7 b_i$. And now we can show directly that $y^*$ is invariant under this translation action: for every $h \in G$, we have
		\begin{equation*}
			y^*(hz) = \sum_{i=0}^7 \varphi((hz)(i)) = \sum_{i=0}^7 \varphi(z(i-h)) = \sum_{i=0}^7 \varphi(z(i)) = y^*(z).
		\end{equation*}    
		
		Now, since $y^*$ is $G$-invariant, we get $\norm{\id + T} = 2$. By compactness of $B_Y$, we see that there must exist $z \in Y$ such that 
		\begin{equation*}
			\norm{z + y^*(z)y} = 2.
		\end{equation*}
		For ease of notation, denote $z = ((a_i, b_i))_{i=0}^7$. From the above equality, it follows that $|y^*(z)|=1$ and so (potentially switching $-z$ for $z$)
		\begin{equation*}
			1 = \abs{y^*(z)} = \abs{\sum_{i=0}^7 b_i} = \sum_{i=0}^7 b_i
			\le \sum_{i=0}^7 \abs{b_i} \le \sum_{i=0}^7 \norm{z_i} = \norm{z} = 1.
		\end{equation*}
		That is, other than $\sum_{i=0}^7 b_i = 1$ we also know that $b_i \ge 0$ and for every $i \in \{0, \dots, 7\}$
		\begin{equation*}
			\max \left\{ \abs{a_i}, \frac{\abs{a_i} + \abs{b_i}}{\sqrt{2}} \right\} \le b_i.
		\end{equation*}
		This means that all $z_i$ are contained in the cone highlighted in~\Cref{fig:ballForVecValL1example}. Therefore, for all $0 \le i \le 7$, we have $\abs{a_i} \le (\sqrt{2}-1)b_i$. Now we look at
		\begin{equation*}
			\norm{z + y^*(z)y}_Y
			= \norm{z + y}_Y
			= \sum_{i=0}^7 \norm{(a_i + 1/8, b_i)}_X.
		\end{equation*}
		From the considerations above, we have that $\abs{a_i + 1/8} \le b_i + 1/8$ and if $b_i > 0$, then the inequality is strict; also $\abs{b_i} < b_i + 1/8$ and 
		\begin{equation*}
			\frac{\abs{a_i + 1/8} + \abs{b_i}}{\sqrt{2}}
			\le \frac{(\sqrt{2}-1+1)b_i + 1/8}{\sqrt{2}}
			< b_i + \frac{1}{8}.
		\end{equation*}
		Since $b_i > 0$ for at least one $i$, for this $i$, in all three cases the inequality is strict and we obtain that
		\begin{equation*}
			\norm{z + y^*(z)y}_Y
			= \sum_{i=0}^7 \norm{(a_i + 1/8, b_i)}_X
			< \sum_{i=0}^7 \left( b_i + \frac{1}{8} \right)
			= 1 + \sum_{i01}^7 b_i
			= 2.
		\end{equation*}
		This is, however, a~contradiction with our choice of $z$.
	\end{proof}
	
	Now let us consider the same phenomenon in $C(K,X)$ spaces.
	
	\begin{example} \label{C(K,X)-action}
		Let $G = \Z_8$ and $X$ be the $G$-Banach space from~\Cref{examples-L1-C(K)-independent-actions}. Then, we have that $X^*$ has the $G$-DPr and if we equip $Y = C(G, X^*)$ with the action by translations, then $Y$ does not have the $G$-DPr. 
	\end{example}
	\begin{proof}
		Since $X$ is finite-dimensional, it is reflexive. It is easy to see that the action on $X^{**}$ induced as the dual action to the dual action on $X^*$ is the same action as the original action on $X$. Hence, it follows from \Cref{prop:inheritanceFromDual} and the fact that $X$ has the $G$-DPr (which was shown in~\Cref{examples-L1-C(K)-independent-actions}), that $X^*$ has the $G$-DPr.
		Now, since $G$ is discrete, we have that actually $Y = \ell_\infty(G, X^*)$. It follows then that $Y = \ell_1(G, X)^*$. If we show that the dual action to the action of $G$ on $\ell_1(G, X)$ by translations is the action by translations on $\ell_\infty(G, X^*)$, then we will know that the latter cannot have the $G$-DPr for then, by virtue of~\Cref{prop:inheritanceFromDual}, so would the former. Let $x = (x_i)_{i=0}^7 \in \ell_1(G, X)$, $x^* = (x^*_i)_{i=0}^7 \in \ell_\infty(G, X)$ and $g \in G$. Then, considering the dual action on $\ell_\infty(G, X^*)$, we have that 
		\begin{eqnarray*}
			(gx^*)(x) = x^*(g^{-1}x) &=& x^*\big( (x_{i+g \textrm{ mod } 8})_{i=0}^7 \big) \\
			&=& \sum_{i=0}^7 x^*_i (x_{i+g \textrm{ mod } 8}) \\
			&=& \sum_{i=0}^7 x^*_{i-g \textrm{ mod } 8} (x_i).
		\end{eqnarray*}
		However, the expression on the right is nothing else but the action of $x^*$ translated by $g$ on $x$.
	\end{proof}

	One might hope for a connection between the $G$-Daugavet property and the possibility of embedding a space into another one with an unconditional basis. However, the classical theory of the alternative Daugavet property already shows that such a statement cannot hold in full generality: indeed, the space $c_0$ has the alternative Daugavet property (recall that $c_0$ has numerical index equals one) and yet embeds trivially into itself. One may argue that the $G$-Daugavet property is a weaker requirement, so perhaps one should only expect a correspondingly weaker statement. A natural refinement would then be to ask whether a space with the $G$-Daugavet property can embed into a space with an unconditional basis $(e_n)$ which is preserved by the action, in the sense that
	\begin{equation*} 
		G\cdot \{e_n : n\in \mathbb N\}=\{e_n : n\in \mathbb N\}.
	\end{equation*}
	Right away, this excludes the previous pathology coming from the alternative Daugavet property, since if $e_n$ is a basis vector, then $-e_n$ cannot belong to the same basis. Nevertheless, \Cref{ex:Iso^+(ell_1)} below provides a simple counterexample even to this variant of the statement as $\Iso^+(\ell_1)$, the group of all positive surjective linear isometries of $\ell_1$, preserves the canonical basis of $\ell_1$.
	
	\begin{proposition} \label{ex:Iso^+(ell_1)}
		The $G$-Banach space $\ell_1$ has the $G$-DPr for $G = \Iso^+(\ell_1)$.
	\end{proposition}
	\begin{proof}
		For any $x \in S_{\ell_1}$ and $\eps > 0$, find $i \in \N$ such that $\abs{x_i} < \eps/2$. Then, we have that $\pm e_i \in Q(x,\eps)$ and hence $\{\pm e_n : n \in \N\} \subseteq G Q(x,\eps)$, so $B_X = \clco_G(Q(x,\eps))$.
	\end{proof}
	
	In general, there is no ``minimal group'' (even if we restrict ourselves to closed subgroups of $\Iso(X)$) for which a $G$-Banach space has the $G$-DPr. In particular, we have the following consequence of \Cref{ex:Iso^+(ell_1)}.
	
	\begin{corollary} \label{cor:intersectionInstability}
		Let $X$ be a $G$-Banach space. If $G, H \leq \Iso(X)$ are such that $X$ has both the $G$-DPr and $H$-DPr, then $X$ need not have the $(G \cap H)$-DPr.
	\end{corollary}
	\begin{proof}
		Indeed, $\ell_1$ has the $\{\id, -\id\}$-DPr (which is just the aDPr) and $\Iso^+(\ell_1)$-DPr, both of which are closed subgroups of $\Iso(\ell_1)$ with the SOT topology. Nevertheless, $\{\id, -\id\} \cap \Iso^+(\ell_1) = \{\id\}$ and $\ell_1$ does not have the classical Daugavet property.
	\end{proof}
	
	Although $\ell_1$ does not have the DPr, it does have the $G$-DPr for suitably large groups, for instance for $\Iso^+(\ell_1)$ by \Cref{ex:Iso^+(ell_1)}. The following result shows that the $G$-DPr of $\ell_1$ is driven not merely by positivity of the action but by the size of the acting group. Indeed, while the full group $\Iso^+(\ell_1)$ is large enough to force the property, every SOT-compact subgroup is too small to do so.
	
	\begin{proposition}
		The $G$-Banach space $\ell_1$ equipped with the action of an SOT-compact subgroup $G \leq \Iso^+(\ell_1)$ does not have the $G$-Daugavet property.
	\end{proposition}
	\begin{proof}
		Let $G$ be a~compact subgroup of $(\Iso^+(\ell_1), \mathrm{SOT})$. The map $T \mapsto T(-e_1)$ is SOT-to-norm continuous, 
		so the orbit $G(-e_1)$ must be a norm-compact subset of $\ell_1$. Since $G(-e_1)$ consists entirely of negative standard basis vectors, which are uniformly separated by a distance of 2, this norm-compact set must be finite.
		Together with the fact that $G \leq \Iso^+(\ell_1)$, we get that there is $n \in \N$ and $k_1, \dots, k_n$ such that $G(-e_1) = \{-e_{k_i}: 1 \leq i \leq n\}$. Let $\eps = \min \left\{ 2^{-k_i} : 1 \leq i \leq n \right\}$. Define $x = (2^{-k}) \in S_{\ell_1}$. Then, for every $1 \leq i \leq n$, we have that
		\begin{eqnarray*}
			\norm{x + (-e_{k_i})}
			&=& \abs{2^{-k_i} - 1} + \sum_{k \neq k_i} 2^{-k} \\
			&=& 1 - 2^{-k_i+1} + \sum_{k \in \N} 2^{-k} 
			= 2 - 2^{-k_i+1}
			\leq 2 - 2\eps.
		\end{eqnarray*}
		Now, for every $u \in B_{\ell_1}$ such that $\|u+e_{k_i}\| < \e$, we have that $\|x+u\| \leq \|x - e_{k_i}\| + \|u+e_{k_i}\| < 2 - \e$. This shows that $B(-e_{k_i}, \e) \cap Q(x, \e) = \emptyset$ for every $1 \leq i \leq n$. Therefore, one can show that $G Q(x, \e) \cap B(-e_1, \e) = \emptyset$. Since $-e_1$ is strongly exposed in $B_{\ell_1}$ by the functional $-e_1^*$, there exists $\eta > 0$ such that $S(B_{\ell_1}, -e_1^*, \eta) \subseteq B(-e_1, \e)$. This means that $GQ(x,\e)$ is disjoint from this new slice, that is, $-e_1^*(w) \leq 1 - \eta$ for every $w \in GQ(x, \e)$. By convexity and norm-closedness of $\{u \in \ell_1: -e_1^*(u) \leq 1 - \eta\}$, we get that $\clco_G Q(,x\e) = \overline{\co}(GQ(x, \e)) \subseteq \{ u \in \ell_1: -e_1^*(u) \leq 1 - \eta\}$. Since $-e_1^*(-e_1) = 1$, we get that $-e_1 \not\in \clco_G Q(x, \e)$.
		
	\end{proof}

	The preceding examples might suggest that the $G$-Daugavet property is essentially produced almost transitivity (as in $L_p$-spaces) or by the presence of a very large non-compact group (as in the case of $\Iso^+(\ell_1)$). One could therefore hope that, after excluding such mechanisms, the $G$-Daugavet property would force some of the familiar isomorphic consequences of the classical DPr, for instance non-reflexivity or non-Asplundness. The following example shows that this is also not the case. In fact, even for a compact action with no nontrivial invariant subspaces and which is not convex-transitive, the $G$-Daugavet property may still hold on a finite-dimensional space (see also \Cref{inf-dim-reflexive-with-GDPr-not-almost-transitive}).
	
	\begin{example}
		Let \(X=(\mathbb R^2,\|\cdot\|)\) and let \(G=\mathbb Z_8=\langle R\rangle\) be as in \Cref{examples-L1-C(K)-independent-actions}. Then, the following assertions hold true.
		\begin{itemize}
			\item[(a)] $G$ is compact and $X$ is a reflexive space which satisfies the $G$-DPr.

			\item[(b)] The action has no nontrivial proper invariant subspaces.
			
			\item[(c)] The action is not convex-transitive on $X$.
		\end{itemize}
		\begin{proof} Item (a) was established in \Cref{examples-L1-C(K)-independent-actions}. For (b), note that a nonzero proper invariant subspace of the real 2-dimensional space $X$ would necessarily be 1-dimensional. Suppose that $E=\R v$ is $G$-invariant. Since $G=\Z_8=\langle R\rangle$, where $R$ acts by rotation through angle $\pi/4$, the invariance of $E$ implies that $Rv\in E$, so $Rv=\lambda v$ for some real scalar $\lambda$. This is impossible, since the rotation $R$ has no real eigenvalues. Therefore the only $G$-invariant subspaces are $\{0\}$ and $X$. Finally, let us prove (c). Take $x_0=(0,1)\in S_X$ and consider the functional $f(a,b)=b$. For every $g\in G$, we have $f(gx_0)\leq 1$, and equality holds if and only if $gx_0=x_0$ as the orbit $Gx_0$ consists of the eight points obtained by rotating $(0,1)$ (see Figure \ref{fig:ballForVecValL1example}), and only the original point has second coordinate exactly $1$. Let $u=(\xi,1)$ be an extreme point of $B_X$. Then $f(u)=1$. Suppose, towards a contradiction, that $u\in \co(Gx_0)$. Then
			\begin{equation*}
				u=\sum_{k=1}^n \lambda_k g_kx_0
			\end{equation*}
			for some $\lambda_k\geq 0$ with $\sum_{k=1}^n \lambda_k=1$. Applying $f$, we get
			\begin{equation*}
				1=f(u)=\sum_{k=1}^n \lambda_k f(g_kx_0).
			\end{equation*}
			Since $f(g_kx_0)\leq 1$ for every $k$, this forces $f(g_kx_0)=1$ whenever $\lambda_k>0$. Hence $g_kx_0=x_0$ for every such $k$, which implies $u=x_0$. But this is impossible, since $u=(\xi,1)\neq (0,1)=x_0$. Therefore $u\notin \co(Gx_0)$, and consequently $\overline{\co}(Gx_0)\neq B_X$.
		\end{proof}
	\end{example}
	
	To end this section, let us provide one more result in spirit to \Cref{cor:intersectionInstability} but now taking unions instead.
	\begin{example}
		Let $X = \C$, $G_n = \langle e^{2^{-n} \pi i}\rangle \le S_\C$ and $G = \bigcup_{n=1}^\infty G_n$ and consider the natural actions of $G_n$'s and $G$ on $X$ by rotations. Then $X$ does not have the $G_n$-DPr for any $n \in \N$, yet it has the $G$-DPr.
	\end{example}
	
	\begin{proof}
		Since $\C$ is LUR, this is a~simple consequence of \Cref{theorem:GDPrOnUniformlyConvexSpace} -- the action of $G$ is almost transitive since $G = \langle e^{2^{-n} \pi i} : n \in \N \rangle$ is dense in $S_\C$. However, the orbit of any point of $S_X$ under the action of $G_n$ is finite for every $n \in \N$ and hence the orbits cannot be dense in the unit circle.
	\end{proof}

	\section{Generalizations of the classical theory}

	In the previous section we observed that the $G$-Daugavet property may behave very differently from its classical counterparts even on reflexive or finite-dimensional spaces, due to the additional flexibility provided by the group action. We now turn to situations in which the group setting remains closer to the classical Daugavet theory. The aim of this section is to identify conditions on the action or on the interaction between the action and the geometry of the involved Banach space, under which one can recover consequences of the classical Daugavet property and the alternative Daugavet property.

	\subsection{The $G$-Banach spaces $L^1(\mu,X)$ and $C(K,X)$}
	\label{section:L1-and-CK} In this section, we extend to the context of the $G$-Daugavet property the classical characterizations from \cite{MO2004, MartinPaya2000} asserting that
	\begin{itemize}
		\item $L^1(\mu,X)$ has the (a)DPr $\Leftrightarrow$ $\mu$ is atomless or $X$ has the (a)DPr, and
		\item $C(K,X)$ has the (a)DPr $\Leftrightarrow$ $K$ is perfect or $X$ has the (a)DPr.
	\end{itemize}
	Our aim is to show that these equivalences remain valid in the group setting provided one equips $L^1(\mu,X)$ and $C(K,X)$ with the natural pointwise actions induced by the action on $X$.

	It is important to emphasize that the choice of these actions is not arbitrary. As we have seen before, if the action on $L^1(\mu,X)$ or on $C(K,X)$ is not related to the original action on $X$, then there is no reason to expect any characterization of the above form to remain true. This is precisely what happens in Examples \ref{examples-L1-C(K)-independent-actions} and \ref{C(K,X)-action}. This choice is also natural from the point of view of the structure of linear isometries on these spaces. In the scalar-valued case (see, for instance, \cite{FlemingJamison2003}), the Banach-Stone theorem shows that surjective linear isometries on $C(K)$ are determined by homeomorphisms of $K$ together with modulus-one multipliers (in the real case, this reduces to continuous changes of sign). In the vector-valued setting (see, for instance, \cite{ FlemingJamison2007}), the corresponding description of surjective linear isometries on $C(K,X)$ has the same form, except that the scalar multiplier is replaced by a continuous change of surjective linear isometries of $X$. Likewise, for vector-valued $L^p$-spaces (including $L^1$), Banach-Lamperti type theorems describe, under suitable assumptions on $X$, natural classes of surjective linear isometries in terms of measure-space transformations combined with isometries of $X$. 
	
	The role of the next two stability results is precisely to explain why the induced actions appearing in Theorems \ref{L1(X,mu)-iff} and \ref{thm:c-of-K-X} are the correct ones.
	
	The first proposition concerns coordinatewise actions on $\ell_1$-sums and will be the basic tool for the $L^1$-case, while the second one concerns coordinatewise actions on $\ell_\infty$-sums and will be used in the $C(K,X)$-case. As expected, their proofs are based on the classical cases and we will include them for the sake of completeness. 
	
	In what follows, for a family of groups $(G_i)_{i\in\N}$, we denote by $\prod_{i=1}^\infty G_i$ their direct product, that is, the group of all sequences $(g_i)_{i\in\mathbb N}$ with $g_i\in G_i$ for every $i$, endowed with coordinatewise multiplication. We denote by $\bigoplus_{i=1}^\infty G_i$ the direct sum, that is, the subgroup of $\prod_{i=1}^\infty G_i$ consisting of all sequences with finite support.
	
	\begin{proposition} \label{stability-ell1}
		Let $X_i$ be $G_i$-Banach spaces for $i \in I$, where $I$ is any index set, and define $X = \left(\sum_{i \in I} X_i\right)_{\ell_1}$. Consider the coordinatewise action given by
		\begin{equation*}
			(g_i)_{i \in I}(x_i)_{i \in I} = (g_i x_i)_{i \in I}
		\end{equation*}
		for every $(x_i)_{i \in I} \in X$ and $(g_i)_{i \in I} \in G$, where $G$ is defined below. Then, we have that the following are equivalent.
		\begin{itemize}
			\item[(a)] $X$ has the $G$-DPr for $G = \prod_{i\in I} G_i$.
			\item[(b)] $X$ has the $G$-DPr for $G = \bigoplus_{i \in I} G_i$, w.r.t. the restriction of the same action.
			\item[(c)] Each $X_i$ has the $G_i$-DPr.
		\end{itemize}
	\end{proposition}
	
	\begin{proof} (b) $\Rightarrow$ (a). Since the direct sum is a subgroup of the direct product, this implications follows from \Cref{fact-DP-implies-GDP}.  
		
		\noindent
		(a) $\Rightarrow$ (c). Assume that $X$ has the $G$-DPr and fix an arbitrary $k \in I$. We show that $X_k$ has the $G_k$-DPr. Let $T_k: X_k \to X_k$ be a rank-one operator and without loss of generality assume that $\norm{T_k} = 1$. Choose $\eps > 0$ arbitrarily. We define an extension $\hat{T}: X \to X$ by
		\[
		\hat{T}((x_i)_{i \in I}) = (0, \dots, 0, T_k(x_k), 0, \dots)
		\]
		where the non-zero term is in the $k$-th position. Then $\hat{T}$ is a rank-one operator on $X$ with $\|\hat{T}\| = \norm{T_k} = 1$. Since $X$ has the $G$-DPr, we have
		\[
		\sup_{g \in G} \norm{\id_X + g \circ \hat{T}} = 1 + \|\hat{T}\| = 2.
		\]
		So there are $g = (g_i)_{i \in I} \in G$ and $x = (x_i)_{i \in I} \in S_X$ such that
		\begin{equation*}
			\norm{x + g \circ \hat{T}(x)} > 2 - \eps.
		\end{equation*}
		Observe that
		\[
		(\id_X + g \circ \hat{T})(x) = x + g(\hat{T}x) = (x_1, \dots, x_k + g_k T_k x_k, \dots).
		\]
		Calculating the norm in $X = (\sum X_i)_{\ell_1}$, we get
		\begin{align*}
			2 - \eps
			&< \norm{(\id_X + g \circ \hat{T})(x)}
			= \sum_{i \neq k} \norm{x_i}_{X_i} + \norm{x_k + g_k T_k x_k}_{X_k}\\
			&= (1 - \norm{x_k}_{X_k}) + \norm{x_k + g_k T_k x_k}_{X_k}
			\leq (1 - \norm{x_k}_{X_k}) + \norm{x_k} (1 + \norm{T_k})\\
			&= 1 + \norm{x_k}_{X_k},
		\end{align*}
		implying that $\norm{x_k}_{X_k} > 1 - \eps$ and consequently $\sum_{i \neq k} \norm{x_i}_{X_i} < \eps$.
		Now, consider the vector $y \in X$ defined by the embedding of $x_k$ (i.e., $y_k = x_k$ and $y_i = 0$ for $i \neq k$). Note that $\norm{x - y}_X = \sum_{i \neq k} \norm{x_i} < \e$. By the triangle inequality, we get that
		\begin{align*}
			\norm{x_k + g_k T_k x_k}_{X_k}
			&= \norm{y + g \circ \hat{T}(y)}_X \\
			&\geq \norm{x + g \circ \hat{T}(x)}_X - \norm{(x - y) + g \circ \hat{T}(x - y)}_X.
		\end{align*}
		Since $\norm{\id + g \circ \hat{T}} \leq 2$, we have that
		\[
		\norm{(\id_{X_k} + g_k T_k)(x_k)}_{X_k} > (2 - \e) - 2\e = 2 - 3\e.
		\]
		Finally, let $u = x_k / \norm{x_k} \in S_{X_k}$. Since $\norm{x_k} \leq 1$, we get
		\[
		\norm{(\id_{X_k} + g_k T_k)(u)}_{X_k} = \frac{1}{\norm{x_k}} \norm{(\id_{X_k} + g_k T_k)(x_k)} > 2 - 3\e.
		\]
		Since $\e$ was arbitrary, $\sup_{h \in G_k} \norm{\id_{X_k} + h \circ T_k} = 2 = 1 + \norm{T_k}$, completing the proof.
		
		\noindent
		(c) $\Rightarrow$ (b). Assume that each $X_i$ has the $G_i$-DPr. Let $T: X \to X$ be a~rank-one operator and again without loss of generality assume that $\norm{T} = 1$. We can write $T = \phi \otimes y$ for some $\phi \in S_{X^*}$ and $y \in S_X$.
		Recall that $X^*$ is isometric to $(\sum X_i^*)_{\ell_\infty}$.
		Thus, we can write $\phi = (\phi_i)_{i \in I}$ with $\phi_i \in X_i^*$ and $\norm{\phi} = \sup_i \norm{\phi_i}_{X_i^*}$.
		Similarly, we can write $y = (y_i)_{i \in I}$ with $y_i \in X_i$ and $\norm{y} = \sum_i \norm{y_i}_{X_i}$.
		Let $\eps > 0$ and find $\eta \in (0, 1)$ such that $3\eta < \eps$. Choose an index $k \in I$ such that $\norm{\phi_k} \ge 1 - \eta = \norm{\phi} - \eta$.
		We need to handle two cases.
		
		\noindent
		\textbf{Case 1:} $\norm{y_k} < \eta$.
		
		In this case, we can find $u \in S_{X_k}$ such that $\phi_k(u) > 1 - \eta$. Then, define $g \in G$ to be the identity on each coordinate and $x \in S_X$ by embedding $u$ into the $k$-th coordinate (i.e., $x_k=u$ and $x_i=0$ for $i \neq k$). We compute
		\begin{align*}
			\norm{x + gTx}
			&= \norm{x + \phi_k(u)y} 
			= \sum_{j \neq k} \norm{\phi_k(u) y_j}_{X_j} + \norm{u + \phi_k(u) y_k}_{X_k}.
		\end{align*}
		For the first term, we use $\sum_{j \neq k} \norm{y_j} = 1 - \norm{y_k}$:
		\[
		\sum_{j \neq k} \norm{\phi_k(u) y_j} = \abs{\phi_k(u)} (1 - \norm{y_k}) \ge (1-\eta)(1-\eta) > 1 - 2\eta.
		\]
		For the second term, using the triangle inequality to get
		\[
		\norm{u + \phi_k(u) y_k} \ge \norm{u} - \abs{\phi_k(u)}\norm{y_k} > 1 - \eta.
		\]
		Summing these gives $\norm{x + gTx} > 2 - 3\eta > 2 - \eps$.

		\noindent
		\textbf{Case 2:} $\norm{y_k} \ge \eta$.
		
		Consider the operator $T_k = \phi_k \otimes y_k$ on $X_k$.
		Since $X_k$ has the $G_k$-DPr, there exist $g_k \in G_k$ and $u \in S_{X_k}$ such that
		\[
		\norm{u + g_k T_k u}_{X_k} \ge 1 + \norm{T_k} - \eta^2.
		\]
		Substituting $T_k u = \phi_k(u)y_k$, this becomes
		\begin{equation} \label{eq:comp_k}
			\norm{u + g_k \phi_k(u) y_k}_{X_k}
			\ge 1 + \norm{\phi_k}\norm{y_k} - \eta^2
			\ge1 + \norm{\phi_k}\norm{y_k} - \eta.
		\end{equation}
		Using the triangle inequality, we obtain $\norm{g_k\phi_k(u)y_k} \ge 1 + \norm{\phi_k}\norm{y_k} - \eta^2 - 1$, which, using the fact that $G$ acts by isometries and dividing by $\norm{y_k}$, implies
		\begin{equation*}
			\abs{\phi_k(u)}
			\ge \norm{\phi_k} - \frac{\eta^2}{\norm{y_k}}
			\ge \norm{\phi_k} - \eta
			\ge 1 - 2\eta.
		\end{equation*}
		
		Now, define $x \in S_X$ by embedding $u$ into the $k$-th coordinate (i.e., $x_k=u$ and $x_i=0$ for $i \neq k$). Define $g \in G$ such that its $k$-th component is $g_k$ and $g_j$ is the identity element of $G_j$ for $j \neq k$. We compute
		\begin{align*}
			\norm{x + gTx}
			&= \norm{x + g(\phi(x)y)} 
			= \norm{x + g(\phi_k(u)y)} \\
			&= \sum_{j \neq k} \norm{\phi_k(u) y_j}_{X_j} + \norm{u + g_k \phi_k(u) y_k}_{X_k}.
		\end{align*}
		Thus, using \eqref{eq:comp_k} above, we get that
		\begin{align*}
			\norm{x + gTx} 
			&\ge |\phi_k(u)| \sum_{j \neq k} \norm{y_j} + (1 + \norm{\phi_k}\norm{y_k} - \eta)\\
			&\ge (1-2\eta) \sum_{j \neq k} \norm{y_j} + 1 + (1-\eta)\norm{y_k} - \eta\\
			&= 1 + \sum_{j \neq k} \norm{y_j} + \norm{y_k}
			- \left(2\sum_{j \neq k} \norm{y_j} + \norm{y_k} + 1\right)\eta\\
			&\ge 1 + \norm{y} - (2\norm{y} + 1) \eta\\
			&= 2 - 3\eta \ge 2-\eps.
		\end{align*}
		As $\e$ was chosen arbitrarily, $\sup_{g \in G} \norm{\id + g \circ T} \ge 1 + \norm{T}$ and $X$ has the $G$-DPr.
	\end{proof}
	
	One can also get a~``diagonal'' version of the previous result.
	\begin{proposition}[Diagonal $G$-$\ell_1$ stability] \label{prop:diagonal-ell1-stability}
		Let $X_i$ be Banach spaces and let a single group $G$ act by isometries on each $X_i$. Let $X = \left( \sum X_i \right)_{\ell_1}$ and let $G$ act diagonally, that is, $g(x_i)_i := (g x_i)_i$ for every $g \in G$ and every $(x_i)_i \in X$. Then, $X$ has the $G$-DPr if and only if $X_i$ has the $G$-DPr for every $i$.
	\end{proposition}
	The proof of this is similar to the proof of \Cref{stability-ell1} and is henceforth omitted.
	
	As we shall see below, Proposition \ref{prop:diagonal-ell1-stability} is exactly the ingredient that allows one to pass from the classical decomposition of $L^1(\mu,X)$ into atomless and atomic parts to the group characterization in Theorem \ref{L1(X,mu)-iff} below. Let us recall that atomless $L^1(\nu, X)$ has the classical DPr \cite{Kadets2000}. This implies that, for atomless $\nu$, $L^1(\nu, X)$ has the $G$-DPr for every $G$ acting by isometries on $X$.
	We can follow \cite[Remark 9]{MartinPaya2000} to get the promised result.

	\begin{theorem} \label{L1(X,mu)-iff} Let $X$ be a $G$-Banach space and $\mu$ a positive measure. Equip $L^1(\mu, X)$ with the action given by $(gf)(t) = g(f(t))$ for every $g \in G$, $f \in L^1(\mu, X)$ and $\mu$-almost every $t$. Then, $L^1(\mu, X)$ has the $G$-DPr if and only if $\mu$ is atomless or $X$ has the $G$-DPr.
	\end{theorem}
	
	\begin{proof} Suppose that $X$ has the $G$-DPr and let $\mu$ be a positive measure. Then, we can write $L^1(\mu, X) = L^1(\nu, X) \oplus_1 [ \oplus_{i \in I} X ]_{\ell_1}$ for some set $I$ and atomless $\nu$. Since $\nu$ is atomless, $L^1(\nu, X)$ has the DPr and a fortiori it has the $G$-DPr. As $X$ is assumed to have the $G$-DPr, by the diagonal $G$-$\ell_1$ stability, $L^1(\mu, X)$ has the $G$-DPr. Conversely, if $L^1(\mu, X)$ has the $G$-DPr and $\mu$ has an atom, then one can write $L^1(\mu, X) = X \oplus_1 Z$ for some Banach space $Z$. Again by the diagonal $G$-$\ell_1$ stability, $X$ has the $G$-DPr.
	\end{proof}

	\begin{remark} 
		The compatibility (once again, for the meaning of the word ``compatibility'' here, we send the reader to the comments immediately before Example \ref{examples-L1-C(K)-independent-actions} above) between the action on $X$ and the induced action on $L^1(\mu,X)$ in Theorem~\ref{L1(X,mu)-iff} is (once again) essential. One can produce trivial counterexamples by taking a measure $\mu$ with atoms and equipping $L^1(\mu,X)$ with the trivial action. In that case, the $G$-DPr reduces to the classical DPr and therefore the conclusion of the theorem fails whenever $X$ is finite-dimensional. More interestingly, the theorem may also fail for natural non-trivial actions if the action on $L^1(\mu,X)$ is not compatible with the action on $X$. To see this, let $X$ and $Y$ be as in \Cref{examples-L1-C(K)-independent-actions} and let $G=\mathbb Z_8$. Denote by $\alpha$ the action of $G$ on $Y$ by translations, by $\beta$ the action of $G$ on $Y$ induced by the action on $X$ as in Theorem~\ref{L1(X,mu)-iff}, and by $\gamma$ the action of $G$ on $L^1(\mu,Y)$ induced by $\beta$. If $\mu$ has atoms, then \Cref{examples-L1-C(K)-independent-actions} shows that $Y$ does not have the $G$-DPr with respect to $\alpha$, while it does have the $G$-DPr with respect to $\beta$. Consequently, Theorem~\ref{L1(X,mu)-iff} yields that $L^1(\mu,Y)$ has the $G$-DPr with respect to $\gamma$, even though $\mu$ is not atomless and $Y$ does not have the $G$-DPr with respect to the action $\alpha$.
	\end{remark}

	Now we turn to the $C(K, X)$ characterization.
	
	\begin{proposition} \label{prop:ellInftySumStability}
		Let $X_i$ be $G_i$-Banach spaces for $i \in \N$, and define $X = \left(\sum_{i=1}^{\infty} X_i\right)_{\ell_\infty}$. Consider the coordinatewise action given by
		\begin{equation*}
			(g_i)_{i \in \N}(x_i)_{i \in \N} = (g_i x_i)_{i \in \N}
		\end{equation*}
		for every $(x_i)_{i \in \N} \in X$ and $(g_i)_{i \in \N} \in G$, where $G$ is defined below. Then, we have that the following are equivalent.
		\begin{itemize}
			\item[(a)] $X$ has the $G$-DPr for $G = \prod_{i=1}^{\infty} G_i$.
			\item[(b)] $X$ has the $G$-DPr for $G = \bigoplus_{i=1}^{\infty} G_i$, w.r.t. the restriction of the same action.
			\item[(c)] Each $X_i$ has the $G_i$-DPr.
		\end{itemize}
	\end{proposition}

	\begin{proof} (b) $\Rightarrow$ (a). Since the direct sum is a subgroup of the direct product, this implication follows from \Cref{fact-DP-implies-GDP}.
		
		\noindent
		(a) $\Rightarrow $(c). Assume $X$ has the $G$-DPr and fix an arbitrary $k \in \N$. We show that $X_k$ has the $G_k$-DP. Let $T_k: X_k \to X_k$ be a rank-one operator, without loss of generality assume that $\norm{T_k} = 1$, and choose $0<\eps<1$. We define an extension $\hat{T}: X \to X$ by
		\[
		\hat{T}((x_i)_{i \in \N}) = (0, \dots, 0, T_k(x_k), 0, \dots),
		\]
		where the non-zero term is in the $k$-th position. Then $\hat{T}$ is a rank-one operator on $X$ with $\|\hat{T}\| = \norm{T_k} = 1$. Since $X$ has the $G$-DP, we have
		\[
		\sup_{g \in G} \norm{\id_X + g \circ \hat{T}} = 1 + \|\hat{T}\| = 2.
		\]
		Thus there are $g = (g_i)_{i \in \N} \in G$ and $x = (x_i)_{i \in \N} \in S_X$ such that
		\begin{equation*}
			\norm{x + g \circ \hat{T}(x)} > 2 - \eps.
		\end{equation*}
		Observe that
		\[
		(\id_X + g \circ \hat{T})(x) = x + g(\hat{T}x) = (x_1, \dots, x_k + g_k T_k x_k, \dots).
		\]
		
		Calculating the norm in $X = (\sum_{i=1}^\infty X_i)_{\ell_\infty}$, we get
		\begin{equation*}
			2 - \eps < \max \left\{ \norm{x_k + g_k T_k x_k}_k, \sup_{n \in \N \setminus\{k\}} \norm{x_n}_n \right\}.
		\end{equation*}
		Since $\norm{x} = 1$, it must be that case that $\norm{x_k + g_k T_k x_k}_k > 2-\eps$. But this shows that
		\begin{equation*}
			\norm{\id + g_k \circ T_k} > 2-\eps
		\end{equation*}
		and, as $0<\eps <1$ was arbitrary, we have the desired equality.
		
		\noindent
		(c) $\Rightarrow$ (b). Assume that each $X_i$ has the $G_i$-DPr. We use the slice characterization (2) from Proposition \ref{prop:characterization-G-DP}. Let $x \in S_X$, $\eps > 0$, and let $S_G = S_G(B_X, \Phi, \delta)$ be a $G$-slice defined by a functional $\Phi \in S_{X^*}$ and some $\delta>0$. We must find $y \in S_G$ such that $\norm{x+y} > 2 - \eps$. Since $\norm{x} = \sup_{i\in\N} \norm{x_i} = 1$, there exists an index $k \in \N$ such that $\norm{x_k} > 1 - \eps/2$. We can decompose the functional $\Phi$ as $\Phi = \phi_k + \Psi$, where $\phi_k \in X_k^*$ acts on the $k$-th coordinate (viewed as a subspace of $X$) and $\Psi \in X^*$ vanishes on the $k$-th coordinate. Specifically, for any $z=(z_i) \in X$, $\Phi(z) = \phi_k(z_k) + \Psi(z)$. Since $X$ is an $\ell_\infty$-sum, we have the identity $\|\Phi\| = \|\phi_k\| + \|\Psi\| = 1$. 
		
		If $\|\phi_k\|=0,$ since $\Phi\in S_X{^*},$ we can choose $w=(w_i)_{i\in \N}\in S_X$ to be such that
		\begin{equation*}
			\re \Phi(w)>1-\delta.
		\end{equation*}
		Now, we can take $y\in X$ defined by 
		\begin{equation*}
			y_i := \begin{cases} 
				x_k, & \text{if } i = k, \\ 
				w_i, & \text{if } i \neq k. 
			\end{cases}
		\end{equation*}
		Then, $y\in B_X$ and 
		\begin{equation*}
			\re \Phi(y)=\re \Phi(w)>1-\delta,
		\end{equation*}
		so $y\in S_G$ and, moreover,
		\begin{equation*}
			\|x+y\|\geq\|x_k+y_k\|=2\|x_k\|>2-\eps.
		\end{equation*}
		Otherwise, if $\|\phi_k\|>0$ take a constant $\eta > 0$  such that $\eta < \min\left\{\frac{\delta}{4}, \frac{\e}{2}, 1\right\}$.
		Since $\|\Phi\|=1$, there exists $w \in S_X$ such that $\re \Phi(w) > 1 - \eta$. This implies
		\begin{equation} \label{eq:psi_estimate}
			\re \Psi(w) = \re \Phi(w) - \re \phi_k(w_k) > 1 - \eta - \|\phi_k\|.
		\end{equation}
		Now, we focus on the space $X_k$. Let $u = x_k / \norm{x_k} \in S_{X_k}$ and consider the normalized functional $\varphi = \phi_k / \|\phi_k\| \in S_{X_k^*}$. Since $X_k$ has the $G_k$-DPr, applying Proposition \ref{prop:characterization-G-DP} to the slice $S_{G_k}(B_{X_k}, \varphi, \eta)$ and the element $u$, there exists $v \in S_{G_k}(B_{X_k}, \varphi, \eta)$ such that $\norm{u+v} > 2 - \eta$.
		By definition of the $G_k$-slice, there exists $h \in G_k$ such that $\re \varphi(hv) > 1 - \eta$, which means
		\begin{equation} \label{eq:phi_estimate}
			\re \phi_k(hv) > \|\phi_k\|(1-\eta).
		\end{equation}
		We now construct the desired element $y = (y_i)_{i \in \N} \in B_X$ and the group element $g = (g_i)_{i \in \N} \in G$. Define
		\begin{equation*}
			y_i = \begin{cases} v & \text{if } i = k, \\ w_i & \text{if } i \neq k, \end{cases}
			\quad \text{and} \quad
			g_i = \begin{cases} h & \text{if } i = k, \\ e_i & \text{if } i \neq k, \end{cases}
		\end{equation*}
		where $e_i$ is the identity in $G_i$. Note that $g \in G = \bigoplus_{i=1}^\infty G_i$. First, we check the norm condition. Since $X$ is an $\ell_\infty$-sum:
		\begin{align*}
			\norm{x+y} &\geq \norm{x_k + y_k} = \norm{x_k + v} \\
			&\geq \norm{u+v} - \norm{u - x_k} \\
			&> 2 - \eta - (1 - \norm{x_k}) \\
			&> 2 - \eta - \eps/2 > 2 - \eps.
		\end{align*}
		Finally, we verify that $y \in S_G(B_X, \Phi, \delta)$. Using that $\Psi$ vanishes on the $k$-th coordinate and $g_i$ is the identity for $i \neq k$:
		\begin{align*}
			\re \Phi(gy) &= \re \phi_k(g_k y_k) + \re \Psi( (g_i y_i)_{i} ) \\
			&= \re \phi_k(hv) + \re \Psi(w) \\
			&> \|\phi_k\|(1-\eta) + (1 - \eta - \|\phi_k\|) \quad \text{by \eqref{eq:phi_estimate} and \eqref{eq:psi_estimate}} \\
			&= 1 - \eta(1 + \|\phi_k\|) \\
			&\geq 1 - 2\eta > 1 - \delta.
		\end{align*}
		Thus, $y \in S_G$ and $\norm{x+y} > 2-\eps$, which concludes the proof.
	\end{proof}

	In the same way as before, \Cref{prop:ellInftySumStability} will provide the stability principle needed in the proof of \Cref{thm:c-of-K-X} below, once one isolates the contribution of the isolated points of $K$. Notice that here we do not need a diagonal version of  \Cref{prop:ellInftySumStability} as in the $L^1$ case.
	
	\begin{theorem} \label{thm:c-of-K-X}
		Let $K$ be a~compact Hausdorff space and $X$ be a~$G$-Banach space. Equip $C(K, X)$ with the action of $G$ defined by $(gf)(t) = g(f(t))$ for every $t \in K, f \in C(K, X)$ and $g \in G$. Then $C(K, X)$ has the $G$-DPr if and only if $K$ is perfect or $X$ has the $G$-DPr.
	\end{theorem}
	\begin{proof} First, assume that $C(K,X)$ has the $G$-DPr, but $K$ is not perfect. Then $K$ contains an isolated point $s\in K$ and the map $U:C(K,X)\longrightarrow X\oplus_\infty C(K\setminus\{s\},X)$ defined by  $Uf=(f(s),f|_{K\setminus\{s\}})$ is an onto isometry. Now define an action by isometries of $G^2$ on $C(K,X)$ by setting
		\begin{equation*}
			((g,h)f)(t)
			=
			\begin{cases}
				g(f(t)), & t=s,\\
				h(f(t)), & t\neq s,
			\end{cases}
		\end{equation*}
		for every $(g,h)\in G^2$ and every $f\in C(K,X)$. The original action of $G$ on $C(K,X)$ is precisely the restriction of this $G^2$-action to the diagonal subgroup $\{(g,g):g\in G\}$. Hence, since $C(K,X)$ has the $G$-DPr with respect to the original action, it also has the $G^2$-DPr with respect to the action above. Under the isometry $U$, this $G^2$-action becomes the coordinatewise action on $X\oplus_\infty C(K\setminus\{s\},X)$, that is, $(g,h)(x,F)=(gx,hF)$, where $(hF)(t)=h(F(t))$ for every $t\in K\setminus\{s\}$. Indeed, for every $f\in C(K,X)$ and every $(g,h)\in G^2$, we have
		\begin{equation*}
			U((g,h)f)
			=
			(g f(s),h(f|_{K\setminus\{s\}}))
			=
			(g,h)Uf.
		\end{equation*}
		Thus $U$ is a $G^2$-equivariant onto isometry. Therefore $X\oplus_{\infty} C(K\setminus\{s\},X)$ has the $G^2$-DPr with respect to the coordinatewise action. By \Cref{prop:ellInftySumStability}, applied to the first coordinate, we conclude that $X$ has the $G$-DPr.
		
		On the other hand, if $K$ is perfect, then $C(K,X)$ has the DPr and therefore also the $G$-DPr by \Cref{DPr-implies-G-DPr}. So the last remaining case is to assume that $X$ has the $G$-DPr and show that so does $C(K,X)$. To that end, let $T:C(K,X)\rightarrow C(K,X)$ be a rank-one operator with $\norm{T}=1$ and fix $\eps>0$. Proceeding as in the standard case (see \cite[Theorem 5]{MartinPaya2000}), we construct $x_0\in S_X$, $t_0\in K$ and a linear isometry $\phi:X\to C(K,X)$ such that $\|T(\phi(x_0))(t_0)\|>1-2\eps$
		and $\phi(x)(t_0)=x$ for every $x\in X$. Define $S:X\rightarrow X$ by $Sx=T(\phi(x))(t_0)$. Since $T$ is rank-one, so is $S$, and
		\begin{equation*}
			\|S\| \geq \|Sx_0\| = \|T(\phi(x_0))(t_0)\| > 1 -2\e.
		\end{equation*}
		Since $X$ has the $G$-DPr, there exist $x\in S_X$ and $g\in G$ such that
		\begin{equation*}
			\|gx+Sx\| \geq 1+ \|S\|-\eps > 2-3\eps.
		\end{equation*}
		Define $h\in C(K,X)$ by $h(t)=\phi(x)(t)$ for every $t\in K$. Then
		\begin{eqnarray*}
			\|g+T\|
			&\geq& \|gh+Th\| \\
			&\geq& \|(gh)(t_0)+(Th)(t_0)\|\\
			&=& \|g(h(t_0))+(Th)(t_0)\|\\
			&=& \|gx+Sx\|\\
			&>& 2-3\eps.
		\end{eqnarray*}
		Since $\eps>0$ is arbitrary, this proves that $C(K,X)$ has the $G$-DPr.
	\end{proof}

	\subsection{Group separable determination}
	
	The aim of this section is to show that the $G$-Daugavet property is in some sense separably determined as a counterpart of the classical case. We will provide a result for a general group $G$. We start with the following result whose proof is exactly as in the standard case (see, for instance, \cite[Proposition 4.1.2]{KMRW}) and so we omit it. This provides another characterization for the $G$-DPr.
	
	\begin{proposition} \label{prop:denseSubsetG-DP}
		Let $X$ be a $G$-Banach space. Then $X$ has the $G$-Daugavet property if and only if there is a~dense subset $A \subseteq S_X$ such that for every $x \in A$ and every $\delta > 0$, we have that
		\begin{equation*}
			\clco_G(A \setminus (x + (2 - \delta) B_X)) \supseteq A.
		\end{equation*}
	\end{proposition}
	
	We emphasize that the proofs below are not obtained by simply transplanting the classical arguments to the group setting. A new feature, for instance, is that one must construct, in parallel, both a net of subspaces and a compatible net of groups. Recall that a~family $\mathcal{E}$ of subsets of a~Banach space $X$ is said to be \emph{upward filtering} if for every $A, B \in \mathcal{E}$, there is $C \in \mathcal{E}$ such that $A \cup B \subseteq C$. Similarly, a~family $\mathcal{G}$ of subgroups of a~group $G$ is said to be \emph{upward filtering} if for every $G_1, G_2 \in \mathcal{G}$, there is $H \in \mathcal{G}$ such that $G_1, G_2 \leq H$.
	
	\begin{lemma} \label{lem:sepDeterminationLemma}
		Let $X$ be a~$G$-Banach space, $(E_i, G_i)_{i \in I}$ be a~net where $E_i$ are subspaces of $X$ such that $E_i \leq E_j$ for $i \leq j$ and $G_i$ are subgroups of $G$ such that $G_i \leq G_j$ for $i \leq j$. Suppose that $E_i$ is $G_i$-invariant for every $i \in I$ and denote $Y = \overline{\bigcup_{i \in I} E_i}$ and $H = \bigcup_{i \in I} G_i$. If for every $\delta > 0$ and every $i \in I$, there is $j \in I$ such that for every $x \in S_{E_i}$, we have that
		\begin{equation*}
			\clco_{G_j}(S_{E_j} \setminus (x + (2 - \delta) B_{X})) \supseteq S_{E_i},
		\end{equation*}
		then $Y$ is an~$H$-invariant subspace with the $H$-Daugavet property.
	\end{lemma}

	\begin{remark}
		One might initially expect that $H$ will instead be defined as the intersection of all the $G_i$'s, because in general if $Y$ is an~$H$-invariant subspace and $Z$ is a~$K$-invariant subspace, then $\clspan(Y \cup Z)$ is invariant under $H \cap K$.
		However, when applying this lemma, we will indeed want to take the union. It is worth stressing out that the upward filterness is
		the reason why we can take the union of groups and not intersections.
	\end{remark}

	\begin{proof}[Proof of Lemma \ref{lem:sepDeterminationLemma}]
		The upward filtering of $(E_i)$ implies that $Y$ is a~subspace and the upward filtering of $(G_i)$ implies that $H$ is a~group acting on $Y$ by linear isometries. To see that $Y$ is $H$-invariant, pick any $y \in \bigcup_{i \in I} E_i$ and $h \in H$. Then there is $i \in I$ such that $y \in E_i$ and there is $j \in I$ such that $h \in G_j$. Since $I$ is directed, there is $k \in I$ such that $i, j \leq k$. Thus, $y \in E_k$ and $h \in G_k$ and so $hy \in E_k \subseteq Y$. By density of $\bigcup_{i \in I} E_i$ in $Y$, we get that $Y$ is $H$-invariant. To show that $Y$ has the $H$-Daugavet property, we will make use of \Cref{prop:denseSubsetG-DP}.
		Define $A = \bigcup_{i \in I} S_{E_i}$, which is dense in $S_Y$. Fix $x \in A$ and $\delta > 0$ and let $y \in A$ be arbitrary.
		Then, there is $i \in I$ such that $x, y \in S_{E_i}$. By hypothesis, there is $j \in I$ such that
		\begin{equation*}
			\clco_{G_j}(A \setminus (x + (2 - \delta) B_X))
			\supseteq \clco_{G_j}(S_{E_j} \setminus (x + (2 - \delta) B_X))
			\supseteq S_{E_i}
			\ni y.
		\end{equation*}
		Since $y \in A$ was arbitrary, we may enlarge $G_j$ to $H$ in the above inclusion and get
		\begin{equation*}
			\clco_{H}(A \setminus (x + (2 - \delta) B_X)) \supseteq A.
		\end{equation*}
		By \Cref{prop:denseSubsetG-DP}, this implies that $Y$ has the $H$-Daugavet property.
	\end{proof}
	
	We have the following result. 
	
	\begin{theorem}
		Let $X$ be a $G$-Banach space for some topological group $G$ and assume that the action of $G$ on $X$ is continuous. Then, the following are equivalent.
		\begin{enumerate}
			\item $X$ has the $G$-Daugavet property.
			\item There exists a~net $((E_i, G_i))_{i \in I}$, where $G_i$ is a~separable subgroup of $G$ and $E_i$ is a~closed separable $G_i$-invariant subspace of $X$ with the $G_i$-Daugavet property for every $i \in I$, such that $E_i \subseteq E_j$ and $G_i \leq G_j$ for all $i \leq j$ and for every separable subspace $Y$ of $X$ and every separable subgroup $H$ of $G$, there is $i \in I$ such that $Y \subseteq E_i$ and $H \leq G_i$.
		\end{enumerate}
	\end{theorem}
	\begin{proof}
		$(2) \implies (1)$: We apply \Cref{lem:sepDeterminationLemma} for the net $((E_i, G_i))_{i \in I}$. Since each $E_i$ has the $G_i$-DPr, the hypothesis of the lemma is satisfied by taking $j=i$ for every $i \in I$. Let 
		\begin{equation*}
			Y_0:= \overline{\bigcup_{i\in I} E_i} \ \ \ \mbox{and} \ \ \ H_0 := \bigcup_{i \in I} G_i.
		\end{equation*}
		By Lemma \ref{lem:sepDeterminationLemma}, the space $Y_0$ has the $H_0$-DPr. We now show that $Y_0 = X$ and $H_0 = G$. Indeed, given $x \in X$, the space $\spann\{x\}$ is separable and $\{1_G\}$ is a separable group of $G$. Hence, by asumption, there exists $i \in I$ such that $x \in E_i$. Therefore, $X = \bigcup_{i \in I} E_i$ and so $Y_0 = X$. Similarly, given $g \in G$, the cyclic subgroup $\langle g \rangle$ is countable, hence separable. Again by the assumption, there is $i \in I$ such that $\langle g \rangle \leq G_i$. In particular, $G_i \subseteq H_0$ and therefore $G = \bigcup_{i \in I} G_i = H_0$.

		The hypothesis of the lemma is satisfied since each $E_i$ has the $G_i$-Daugavet property and thus we can take $j = i$ for every $i \in I$. Given any $x \in X$, $\spann\{x\}$ is a~separable subspace of $X$ and $\{e\}$ is a~separable subgroup of $G$, so there is $i \in I$ such that $x \in E_i$ and $e \in G_i$. Thus, by the lemma, $X = \overline{\bigcup_{i \in I} E_i}$ has the $G$-Daugavet property.

		\noindent
		$(1) \implies (2)$: First, we prove the following claim.
		
		\emph{Claim:}
		Let $Y$ be a~separable subspace of $X$ and let $H$ be a~countable subgroup of $G$. Then there is a~countable subgroup $H \leq H' \leq G$ and a~separable $H'$-invariant subspace $Y \leq Y' \leq X$ such that $Y'$ has the $H'$-Daugavet property.
		
		\emph{Proof of the claim:}
		Denote $E_0 = Y$ and $G_0 = H$. We shall inductively construct an~increasing sequence of separable subspaces $(E_n)$ of $X$ and an~increasing sequence of separable subgroups $(G_n)$ of $G$ such that $E_n$ is $G_n$-invariant for every $n \geq 1$ and
		\begin{equation*}
			\clco_{G_{n+1}} \left( S_{E_{n+1}} \setminus \left( x + \left( 2 - \frac{1}{n} \right) B_X \right) \right) \supseteq S_{E_n}
		\end{equation*}
		for every $n \geq 1$ and every $x \in S_{E_n}$. Assume $E_n$ and $G_n$ have been constructed. Take a~dense sequence $(x_k)_{k \in \N}$ in $S_{E_n}$.
		By item (7) of \Cref{prop:characterization-G-DP}, for every $(k,j) \in \N \times \N$, there is a~sequence $(y_{k,j}^m)_{m \in \N}$ in $S_X \setminus (x_k + (2 - 1/(2n)) B_X)$ such that $x_j \in \clco_G(\{y_{k,j}^m : m \in \N\})$. Since we can approximate $x_j$ by a~countable sequence of $G$-convex combinations of finitely many $y_{k,j}^m$'s, we may actually find a~countable subgroup $G_{n+1}$ of $G$  such that 
		\begin{equation*}
			x_j \in \clco_{G_{n+1}}(\{y_{k,j}^m : m \in \N\})
			\quad \text{for all } j,k \in \N,
		\end{equation*}
		and without loss of generality we can assume that $G_{n}\subseteq G_{n+1}$ (indeed, if $n \in \N$, if it is necessary, we can substitute $G_{n+1}$ by $\langle G_{n+1}, G_n\rangle$, which is still countable).
		Define
		\begin{equation*}
			E_{n+1} = \clspan \{ g y_{k, j}^m \colon g \in G_{n+1}, j, k, m \in \N \}.
		\end{equation*}
		
		Since $G_{n+1}$ is countable, $E_{n+1}$ is separable and it is clearly $G_{n+1}$-invariant.
		It remains to check that the required inclusion holds. Fix $x \in S_{E_n}$ and find $k_0 \in \N$ such that $\norm{x - x_{k_0}} < 1/2n$. Then, for every $j,m \in \N$, we have that
		\begin{equation*}
			y_{k_0, j}^m \in S_{E_{n+1}} \setminus (x + (2 - 1/n) B_X)
		\end{equation*}
		and $x_j \in \clco_{G_{n+1}}(\{y_{k_0, j}^m : m \in \N\})$.
		Since $j \in \N$ was arbitrary, we get that
		\begin{equation*}
			\clco_{G_{n+1}}(S_{E_{n+1}} \setminus (x + (2 - 1/n) B_X))
			\supseteq \overline{\{x_j: j \in \N\}}
			\supseteq S_{E_n}.
		\end{equation*}
		Notice that, by our construction, the sequence $(E_n)_{n\in \N}$ is increasing, so we may employ \Cref{lem:sepDeterminationLemma} to see that the space $Y' = \overline{\bigcup_{n \in \N} E_n}$ and the group $H' = \bigcup_{n \in \N} G_n$ satisfy the requirements of the claim, which can be shown as in the standard case.
		
		Now, to finish the proof of the theorem, let $\mathcal{S}$ be the family of all pairs $(Y,H)$, where $Y$ is a~closed separable subspace of $X$ and $H$ is a~separable subgroup of $G$ such that $Y$ is $H$-invariant and $Y$ has the $H$-Daugavet property.
		Define an~ordering on $\mathcal{S}$ by setting $(Y_1, H_1) \leq (Y_2, H_2)$ if $Y_1 \subseteq Y_2$ and $H_1 \leq H_2$.
		First we show that $\mathcal{S}$ is directed. Indeed, let $(Y_1, H_1), (Y_2, H_2) \in \mathcal{S}$. Consider the separable subspace $Z = \clspan(Y_1 \cup Y_2)$ and let $K$ be a countable dense subgroup of the  group generated by $H_1 \cup H_2$. Using the claim, we can find a~countable group $K \leq H \leq G$ and a~separable $H$-invariant subspace $Z \leq Y \leq X$ such that $Y$ has the $H$-Daugavet property. This implies that $(Y, \overline{H}) \in \mathcal{S}$ and $(Y_1, H_1), (Y_2, H_2) \leq (Y, \overline{H})$.
		
		Finally, let $Y$ be an~arbitrary separable subspace of $X$ and let $H$ be an~arbitrary separable subgroup of $G$. Let $H'$ be a~countable dense subgroup of $H$. Using the claim again, we can find a~countable group $H' \leq K \leq G$ and a~separable $K$-invariant subspace $Y \leq Z \leq X$ such that $Z$ has the $K$-Daugavet property. Since the action is assumed to be continuous, $Z$ is also $\overline{K}$-invariant and since $\overline{K} \geq K$, $Z$ also has the $\overline{K}$-Daugavet property.
		Thus, $(Z, \overline{K}) \in \mathcal{S}$ and $Y \subseteq Z$ and $H \leq \overline{K}$.
		This shows that the net $(Y, H)_{(Y,H) \in \mathcal{S}}$ satisfies the requirements of item (2).
	\end{proof}

	If we impose some additional requirements on the acting group, we can use a~variation of the same idea to get a leaner formulation.
	\begin{theorem}
		Let $G$ be either a~separable or a~$\sigma$-compact group and $X$ be a~$G$-Banach space. If the action of $G$ on $X$ is continuous, then the following are equivalent.
		\begin{enumerate}
			\item $X$ has the $G$-Daugavet property.
			\item every separable subspace of $X$ is contained in a $G$-invariant separable subspace with the $G$-Daugavet property.
		\end{enumerate}
	\end{theorem}
	\begin{proof}
		$(2) \implies (1)$:
		Let $\mathcal{S}$ be the family of all separable $G$-invariant subspaces of $X$ having the $G$-Daugavet property. We order $\mathcal{S}$ by inclusion.
		First, we show that $\mathcal{S}$ is directed. Let $E_1, E_2 \in \mathcal{S}$. Then $Y = \clspan(E_1 \cup E_2)$ is a separable subspace of $X$. By hypothesis (2), there exists $E \in \mathcal{S}$ such that $Y \subseteq E$. Thus $E_1, E_2 \subseteq E$.
		
		Since every $x \in X$ is contained in the separable subspace $\spann\{x\}$, hypothesis (2) implies that $X = \overline{\bigcup_{E \in \mathcal{S}} E}$.
		We can now apply \Cref{lem:sepDeterminationLemma} considering the net $(E)_{E \in \mathcal{S}}$ indexed by itself, and the constant net of groups $G_E = G$ for all $E \in \mathcal{S}$, yielding that $X$ has the $G$-Daugavet property.

		$(1) \implies (2)$:
		Let $Y$ be a separable subspace of $X$. Denote $E_0 = Y$ We shall inductively construct an~increasing sequence of separable subspaces $(E_n)$ of $X$ such that $E_n$ is $G$-invariant for every $n \in \N$ and
		\begin{equation*}
			\clco_{G} \left( S_{E_{n+1}} \setminus \left( x + \left( 2 - \frac{1}{n} \right) B_X \right) \right) \supseteq S_{E_n}
		\end{equation*}
		for every $n \in \N$ and every $x \in S_{E_n}$. Assume $E_n$ has been constructed and take a~dense sequence $(x_k)_{k \in \N}$ in $S_{E_n}$.
		By item (7) of \Cref{prop:characterization-G-DP}, for every $(k,j) \in \N \times \N$, there is a~sequence $(y_{k,j}^m)_{m \in \N}$ in $S_X \setminus (x_k + (2 - 1/(2n)) B_X)$ such that
		\begin{equation*}
			x_j \in \clco_G(\{y_{k,j}^m : m \in \N\})
			\quad \text{for all } j,k \in \N.
		\end{equation*}
		Define
		\begin{equation*}
			E'_{n+1} = \clspan \{ y_{k, j}^m \colon j, k, m \in \N \} \ \ \mbox{and} \ \ 
			E_{n+1} = \clspan (G \cdot E'_{n+1}).
		\end{equation*}
		First, we note that $E'_{n+1}$ is clearly separable.
		Since $G$ is separable or $\sigma$-compact, the subspace $E_{n+1}$ is also separable. Indeed, if $G$ is separable, this is immediate as the continuous image of a separable space is separable. If $G$ is $\sigma$-compact, then $G = \bigcup_{n \in \N} K_n$ where each $K_n$ is compact. For any $y \in E'_{n+1}$, the orbit $G\cdot y = \bigcup_{n \in \N} K_n y$ is a~countable union of compact (hence separable) sets. Since $E'_{n+1}$ has a~countable dense subset $\{z_k : k \in \N\}$, the space $E_{n+1} = \clspan(\bigcup_{k \in \N} Gz_k)$ is separable.
		
		It remains to check that the required inclusion holds. Fix $x \in S_{E_n}$ and find $k_0 \in \N$ such that $\norm{x - x_{k_0}} < 1/2n$. Then, for every $j,m \in \N$, we have that
		\begin{equation*}
			y_{k_0, j}^m \in S_{E_{n+1}} \setminus (x + (2 - 1/n) B_X)
		\end{equation*}
		and $x_j \in \clco_{G}(\{y_{k_0, j}^m : m \in \N\})$.
		Since $j \in \N$ was arbitrary, we get that
		\begin{equation*}
			\clco_{G}(S_{E_{n+1}} \setminus (x + (2 - 1/n) B_X))
			\supseteq \overline{\{x_j: j \in \N\}}
			\supseteq S_{E_n}.
		\end{equation*}
		With the sequence $(E_n)_{n=1}^\infty$ constructed, we may employ \Cref{lem:sepDeterminationLemma} to see that the space $Y' = \overline{\bigcup_{n \in \N} E_n}$ is the desired $G$-invariant subspace with the $G$-Daugavet property.
	\end{proof}

	\subsection{Group numerical radius} As we have mentioned before, it is well-known that the alternative Daugavet property is closely related to the theory of numerical ranges and numerical index (see, for instance, Chapter 12 of \cite{KMRW}). Recall that the {\it spatial numerical range} of an operator $T \in \mathcal{L}(X)$ is given by $V(T) = \{x^*(T(x)): \|x^*\|=\|x\|=x^*(x)=1\}$ and its {\it numerical radius} given by $v(T) = \sup_{\lambda \in V(T)} |\lambda|$. It is immediate to see that $v(\cdot)$ is a seminorm on $\mathcal{L}(X)$ and $v(T) \leq \|T\|$ for every $T \in \mathcal{L}(X)$. The {\it numerical index} of a Banach space $X$ is defined to be the best constant $k \in [0,1]$ such that $k\|T\| \leq v(T)$ for every $T \in \mathcal{L}(X)$, that is, $n(X) = \inf 
	\{v(T): T \in \mathcal{L}(X), \|T\|=1 \}$. There are real Hilbert spaces with numerical index 0, for instance, $\R^2$ endowed with the $\ell_2$-norm, while complex Banach spaces $X$ satisfy $n(X) \geq 1/e$ by using \cite[\S 4, Theorem 1]{BonsallDuncanV2}. It turns out that $n(X) = 1$ if and only if every $T \in \mathcal{L}(X)$ satisfies the alternative Daugavet equation (see, for instance, \cite[Corollary 12.3.2]{KMRW}). In particular, if $n(X) = 1$, then $X$ satisfies the aDPr. For the classical theory about numerical ranges, we send the reader to the monographs \cite{BonsallDuncanV1, BonsallDuncanV2}.

	In this short section, we extend these concepts for the group-action framework. As a consequence, and we will see this reflected in the upcoming sections, we will have the tools to connect them to the $G$-DPr, thereby recovering as particular cases some classical results concerning DPr, aDPr and RNP. We start by defining the numerical radius and numerical index in this new context. Recall that the symbol $\mathcal{L}^G(X)$ stands for all $G$-equivariant operators on an $G$-Banach space $X$.

	\begin{definition}
		Let $X$ be a~$G$-Banach space. For $T \in \mathcal{L}(X)$, we define the {\it $G$-numerical radius} of $T$ by 
		\begin{equation*}
			v_G(T) = \sup \{ \re x^*(T(gx)) : x \in S_X, x^* \in S_{X^*}, x^*(x) = 1, g \in G \}.
		\end{equation*}
		If $H$ is another group also acting on $X$ by linear isometries, we define the {\it $(G,H)$-numerical index} of $X$ by
		\begin{equation*}
			n_{G, H}(X) = \inf \{ v_G(T) : T \in \mathcal{L}^H(X), \norm{T} = 1 \}.
		\end{equation*}
	\end{definition}
	Notice that $v_{S_\K}(T)$ is the classical numerical radius, that is, the following equality $v_{S_{\K}}(T) = v(T)$ holds true for every $T \in \mathcal{L}(X)$. Moreover, if $S_{\K} \subseteq G$, then we have that $v_G(T) \geq v(T)$ as the elements of $G$ might increase the supremum. In particular, if $n(X) = 1$, then $v(T) = \|T\|$ for every $T \in \mathcal{L}(X)$ and then $v(T) = v_G(T) = \|T\|$. Regarding the numerical index, we have that $n_{S_\K, \{\id\}}(X) = n_{S_\K, S_\K}(X)$ is the numerical index $n(X)$ of $X$. 
	
	\begin{remark} \label{rem:GnumIndex}
		
		Notice that for any group $G$ and any $G$-Banach space $X$, we have
		\begin{align*}
			n_{\{\id_X\}, G}(X)
			&= \inf \{ v_{\{\id_X\}}(T) : T \in \mathcal{L}^G(X), \norm{T} = 1 \} \\
			&\le v_{\{\id_X\}}(-\id)\\
			&= \sup \{ \re x^*(-x) : x \in S_X, x^* \in S_{X^*}, x^*(x) = 1 \}\\
			&= -1.
		\end{align*}
		In particular, $n_{\{\id_X\}, G}(X) = - 1$ even when $X$ is a real Hilbert space. In fact, it is not difficult to see that $n_{G,H}(X) \in [-1,1]$ for every group $G$ and every $G$-Banach space so that $H$ acts on $X$ by linear isometries.
	\end{remark}
	
	Since we will usually be interested in the case when $H = \{\id\}$, we also introduce a shorthand notation
	\begin{equation} \label{G-numerical-index}
		n_G(X)
		= n_{G, \{\id_X\}}(X) = n_{G, S_\K}(X)
		= \inf \{v_G(T): T \in \mathcal{L}(X), \norm{T} = 1 \}
	\end{equation}
	and we call it the $G$-{\it numerical index} of $X$.
	
	\begin{remark}
		In general, $v_G$ is not a seminorm as in the classical case. Indeed, if $G=\{\id\}$, then $v_G(\id_X)=1$ while $v_G(-\id_X)=-1$, so absolute homogeneity fails. On the other hand, if $S_{\K}\subseteq G$ (in particular, if $-\id\in G$ in the real case), then for every scalar $\lambda$ and every $T\in\mathcal{L}(X)$ we have $v_G(\lambda T)=|\lambda|\,v_G(T)$. Since subadditivity is always satisfied, it follows that in this case $v_G$ is a seminorm on $\mathcal{L}(X)$.
	\end{remark}

	It is not difficult to construct nontrivial groups $G$ and $H$ acting on a real Hilbert space $X$ by linear isometries and such that $n_{G,H}(X) = 0$ as in the classical case as we can see in \Cref{example:G-numerical-index-equals-0}.
	
	\begin{example} \label{example:G-numerical-index-equals-0}
		Let
		\begin{equation*}
			X = \left( \bigoplus_{k=1}^{\infty} \mathbb{R}^2 \right)_{\ell_2} 
		\end{equation*}
		so that $X$ is a real Hilbert space. For every sequence $\e = (\e_k)_{k \in \N} \in \{\pm 1\}^{\N}$, define the operator $g_\e \in \Iso(X)$ by $g_\e \big( (x_k)_{k \in \N} \big) = (\e_k x_k)_{k \in \N}$. Let $G = H = \{g_\e : \e \in \{\pm 1\}^{\N}\}$. Then $G$ and $H$ are groups acting on $X$ by linear isometries and $n_{G,H}(X)=0$. Indeed, as $X$ is a Hilbert space we can identify $X^*$ with $X$ by means of the inner product and then, for every $T \in \mathcal{L}(X)$, we have $v_G(T) = \sup \big\{ \langle T(gx), x \rangle : x \in S_X,\ g \in G \big\}$. Since $-\id_X \in G$, it follows that for every $T \in \mathcal{L}^H(X)$ with $\|T\|=1$ we have $v_G(T)\geq 0$, and therefore $n_{G,H}(X)\geq 0$. Now define $J \in \mathcal{L}(X)$ blockwise by $J\big( (a_k,b_k)_{k \in \N} \big) = \big( (-b_k,a_k) \big)_{k \in \N}$. Then $\|J\|=1$. Moreover, $J$ is $H$-equivariant. Indeed, both $J$ and every $g_\e \in H$ act independently on each $\R^2$-block: on the $k$-th block, $J$ maps $(a,b)\mapsto(-b,a)$ while $g_\e$ acts as multiplication by $\e_k \id_{\mathbb{R}^2}$. Since scalar multiples of the identity commute with every linear operator, we get $J g_\e = g_\e J$ for every $\e$, that is, $J \in \mathcal{L}^H(X)$. Finally, for every $g_\e \in G$ and every $x=(x_k)_{k \in \N} \in S_X$, we have
		\begin{equation*}
			\langle J(g_\e x), x \rangle
			= \sum_{k=1}^{\infty} \e_k \langle Jx_k, x_k \rangle
			= 0,
		\end{equation*}
		and then $v_G(J)=0$, which implies that $n_{G,H}(X)\leq 0$.
	\end{example}

	It is a well-known fact (see \cite[Chapter 2, \S 2, Theorem 5]{BonsallDuncanV1}) that
	\begin{equation*}
		\sup \{ \re x^*(T x) : x \in S_X, x^* \in S_{X^*}, x^*(x) = 1 \} = \inf_{\alpha > 0} \frac{\norm{\id + \alpha T} - 1}{\alpha}.
	\end{equation*}
	From this, it easily follows that
	\begin{equation*}
		v_G(T) = \sup_{g \in G} \inf_{\alpha > 0} \frac{\norm{\id + \alpha Tg} - 1}{\alpha}.
	\end{equation*}
	In the following proof, we will need only one inequality from the above expression however with the supremum and infimum interchanged, and we will prove it. We will show this by following the classical proof (which can be found in \cite[Chapter 3, \S 9, Lemma 2]{BonsallDuncanV1}) and keeping track of the group action.

	\begin{lemma} \label{lem:G-DP-numerical-radius}
		Let $X$ be a~$G$-Banach space and let $T \in \mathcal{L}(X)$. Then,
		\begin{equation*}
			\sup_{g \in G} \norm{\id + Tg} = 1 + \norm{T} \iff v_G(T) = \norm{T}.
		\end{equation*}
	\end{lemma}
	\begin{proof}
		The base case with $G = \{\id\}$ is well known and it can be found, for instance, in \cite[Lemma 12.3.1]{KMRW}. First assume that $v_G(T) = \norm{T}$. Then for any $\eps > 0$, there are $x \in S_X$, $x^* \in S_{X^*}$ with $x^*(x) = 1$ and $g \in G$ such that $\re x^*(T(gx)) > \norm{T} - \eps$. Thus,
		\begin{align*}
			\norm{\id + Tg}
			&\ge \norm{x + T(gx)}
			\ge \re x^*(x + T(gx))\\
			&= \re x^*(x) + \re x^*(T(gx))
			> 1 + \norm{T} - \eps.
		\end{align*}
		
		Now, assume that $\sup_{g \in G} \norm{\id + Tg} = 1 + \norm{T}$. For this direction, we need the following inequality as we have mentioned previously.
		
		\noindent
		\textbf{Claim:}
		\begin{equation*}
			\liminf_{\alpha\downarrow 0} \sup_{g \in G} \frac{\norm{\id + \alpha Tg} - 1}{\alpha} \le v_G(T).
		\end{equation*}
		
		Before proving the claim, let us see how it finishes the proof of the lemma. For a contradiction, assume that there is some $\eps > 0$ such that $v_G(T) < \norm{T} - \eps$. By the claim, there is $\alpha \in (0,1)$ such that
		\begin{equation} \label{eqn:G-DP-numerical-radius}
			\sup_{g \in G} \frac{\norm{\id + \alpha Tg} - 1}{\alpha} < \norm{T} - \eps.
		\end{equation}
		
		By our assumption, there is $g \in G$ such that $\norm{\id + Tg} > 1 + \norm{T} - \alpha \eps$. An easy computation then shows that
		\begin{align*}
			\norm{\id + \alpha Tg}
			&= \norm{\id + Tg - (1-\alpha)Tg}
			\ge \norm{\id + Tg} - (1-\alpha)\norm{T}\\
			&> 1 + \norm{T} - \alpha \eps - (1-\alpha)\norm{T}
			= 1 + \alpha\norm{T} - \alpha \eps.
		\end{align*}
		This, however, contradicts \eqref{eqn:G-DP-numerical-radius} as
		\begin{align*}
			\norm{T} - \eps
			= \frac{(1 + \alpha\norm{T} - \alpha\eps) - 1}{\alpha}
			< \frac{\norm{\id + \alpha Tg} - 1}{\alpha}
			\le \sup_{g \in G} \frac{\norm{\id + \alpha Tg} - 1}{\alpha}.
		\end{align*}
		
		All that remains is to prove the claim. To that end, let $\alpha \in (0,1)$ be small enough, $g \in G$ and $x \in S_X$. By the Hahn-Banach theorem, there is $x^* \in S_{X^*}$ such that $x^*(x) = 1$. Then, $\re x^*(T(gx)) \le v_G(T)$ and thus
		\begin{align*}
			\norm{(\id - \alpha Tg)x}
			\ge \re x^*(x - \alpha T(gx))
			= 1 - \alpha \re x^*(T(gx))
			\ge 1 - \alpha v_G(T).
		\end{align*}
		This means that for any $x \in X$ and for a fixed $g \in G$, we have $\norm{(\id - \alpha Tg)x} \ge (1 - \alpha v_G(T))\norm{x}$ and hence, substituting $(\id + \alpha Tg)x$ for $x$, we get
		\begin{equation*}
			\norm{(\id +\alpha Tg)x} \le \frac{1}{1 - \alpha v_G(T)} \norm{(\id - \alpha^2 (Tg)^2)x}, \quad x \in X.
		\end{equation*}
		It follows that
		\begin{equation*}
			\norm{\id + \alpha Tg}
			\le \frac{1 + \alpha^2\norm{T}^2}{1 - \alpha v_G(T)},
		\end{equation*}
		which, after rearranging, gives
		\begin{equation*}
			\frac{\norm{\id + \alpha Tg} - 1}{\alpha}
			\le \frac{v_G(T) + \alpha \norm{T}^2}{1 - \alpha v_G(T)}.
		\end{equation*}
		Since $g \in G$ was arbitrary and the right-hand side does not depend on $g$, we get
		\begin{equation*}
			\sup_{g \in G} \frac{\norm{\id + \alpha Tg} - 1}{\alpha}
			\le \frac{v_G(T) + \alpha \norm{T}^2}{1 - \alpha v_G(T)}.
		\end{equation*}
		Finally, we obtain the desired claim as
		\begin{equation*}
			\liminf_{\alpha\downarrow 0} \sup_{g \in G} \frac{\norm{\id + \alpha Tg} - 1}{\alpha}
			\le \liminf_{\alpha\downarrow 0} \frac{v_G(T) + \alpha \norm{T}^2}{1 - \alpha v_G(T)}
			= v_G(T).
	\end{equation*}\end{proof}

	As we will see in the next sections, the condition $n_G(X)=1$ will play a central role in recovering classical Daugavet-type behavior from the $G$-DPr (see, for instance, \Cref{theorem:all-together} below). In fact, the equality $n_G(X)=1$ may arise naturally from the action itself and not only from the underlying geometry of the Banach space $X$ as we will see in the next result. Recall the definition of convex-transitivity at the beginning of \Cref{section:transitivity}.
	
	\begin{proposition} \label{prop:convex-transitivity-G-numerical-index} Let $X$ be a $G$-Banach space. If the action of $G$ on $X$ is convex-transitive, then $n_G(X) = 1$. 
	\end{proposition}
	
	\begin{proof} We will prove that, for every $T \in \mathcal{L}(X)$, we have that $v_G(T) = \|T\|$. Let us fix $T \in \mathcal{L}(X)$ with $\|T\| = 1$. Then, all we need to do is to prove that $v_G(T)  \geq 1$. Let $\e \in (0,1)$ be arbitrary. Since $\|T^*\|=\|T\|=1$, we can choose a norm-one functional $x^* \in S_{X^*}$ such that $\|T^* x^*\| > 1 - \e$. In fact, by using the Bishop-Phelps theorem \cite{BishopPhelps1961}, we can assume that $x^*$ attains its norm, say $x^*(x) = 1$ for some $x \in S_X$. Also, we can choose $y \in S_X$ to be such that $\re T^*x^*(y) > 1 - \e$. Since the action is convex-transitive, there are $g_1, \ldots, g_n \in G$ and $\lambda_1, \ldots, \lambda_n \in \R$ with $\lambda_k \geq 0$ for every $k=1,\ldots,n$ such that $\lambda_1 + \ldots + \lambda_n = 1$ and the element
		\begin{equation*}
			z:= \sum_{k=1}^n \lambda_k g_k x 
		\end{equation*}
		is close to $y \in S_X$, that is, $\|y-z\| < \e$. As this holds, we have that $|T^*x^*(y-z)| \leq \|y-z\| < \e$ and 
		\begin{equation*}
			\re T^* x^*(z) \geq \re T^*x^*(y) - |T^*x^*(y-z)| > 1 -2\e.
		\end{equation*}
		As $z$ is a convex combination of the vectors $g_k x$, by the last inequality, there exists $k_0 \in \{1, \ldots, n\}$ such that $\re T^* x^*(g_{k_0} x) > 1 - 2\e$. Finally, having $x^*(x) = 1$, we obtain that 
		\begin{equation*}
			v_G(T) \geq \re x^*(T(g_{k_0}x)) = \re T^* x^*(g_{k_0} x) > 1 - 2\e.
		\end{equation*}
		Since $\e>0$ was arbitrary and $x^*(x)=1$, it follows that $v_G(T) = 1$ and $n_G(X) = 1$.
	\end{proof}

	Let us notice that $n_G(X) = 1$ can happen for reasons that have no relation with the action being convex transitive, which implies in particular that the converse of \Cref{prop:convex-transitivity-G-numerical-index} does not hold in general. Indeed, \Cref{ex:Iso^+(ell_1)} tells us that $\ell_1$ has the $\iso^+(\ell_1)$-DPr. By \Cref{thm:G-DPandRNP} below, it turns out that $n_{\iso^+(\ell_1)}(\ell_1)=1$. Nevertheless, $\ell_1$ with $\iso^+(\ell_1)$ is not convex-transitive as the positive isometries preserve positivity, so the orbit of $e_1$ is just $\{e_n: n \in \N\}$ and 
	\begin{equation*}
		\overline{\co}(\iso^+(\ell_1) \cdot e_1) = \left\{ x \in \ell_1: x_n \geq 0, \sum_{n=1}^{\infty} x_n = 1 \right\} 
	\end{equation*}
	which is far from being the whole unit ball $B_{\ell_1}$.
	
	\subsection{Strong Radon-Nikodým operators and the $G$-DPr} Let $X$ and $Y$ be Banach spaces. Recall that an operator $T \in \mathcal{L}(X,Y)$ is said to be {\it strong-Radon-Nikodým} if the subset $\overline{T(B_X)}$ satisfies the Radon-Nikodým property. It is well-known that weakly compact operators are strong Radon-Nikodým. We have the following group invariant version of \cite[Theorem 3.2.6 and Proposition 12.3.7]{KMRW}. As a consequence, we recover the results which relate the DPr, aDPr, RNP and numerical index (see \Cref{classical-RNPs} below).
	
	\begin{proposition} \label{prop:sRN-Operators}
		Let $X$ be a $G$-Banach space. The following statements are equivalent. 
		\begin{enumerate}
			\item $X$ has the $G$-Daugavet property.
			
			\item For every operator $T: X \to X$ such that $\overline{T(B_X)} = \clco\left(\dent{\overline{T(B_X)}}\right)$ we have
			\begin{equation*}
				\sup_{g \in G} \norm{g + T} = 1 + \norm{T}.
			\end{equation*}
			
			\item For every strong Radon-Nikodým operator $T : X \to X$ we have
			\begin{equation*}
				\sup_{g \in G} \norm{g + T} = 1 + \norm{T}.
			\end{equation*}
		\end{enumerate}
	\end{proposition}
	\begin{proof}
		The implications $(2) \implies (3)$ follows immediately from the definition of strong Radon-Nikod\'ym operators and $(3) \implies (1)$ is clear. Let us prove $(1) \implies (2)$. Let $T : X \to X$ be such that for $K = \overline{T(B_X)}$, $K = \clco(\dent{K})$ and, without loss of generality, assume that $\norm{T} = 1$.
		Let $\eps > 0$. Since $K = \clco(\dent{K})$, there is a~denting point $y_0 \in K$ such that $\norm{y_0} > 1 - \eps$.
		There is a~slice $S$ of $K$, such that $\diam(S) < \eps$ and $y_0 \in S$.     By \cite[Lemma 2.6.5]{Kadets2000}, $\tilde{S} := T^{-1}(S) \cap B_X$ is a~slice of $B_X$. Let $x^* \in X^*$ and $\delta > 0$ be such that $\tilde{S} = S(B_X, x^*, \delta)$. For every $x \in \tilde{S}$, we have that $Tx \in S$ and hence $\norm{Tx - y_0} < \eps$. Using \Cref{prop:characterization-G-DP}.(ii) applied to $ y_0 / \norm{y_0} \in S_X$ and the $G$-slice $\tilde{S}_G = S_G(B_X, x^*, \delta)$, there is $z \in \tilde{S}$ and $g \in G$ such that $\norm{y_0 / \norm{y_0} + gz} > 2 - \eps$. Thus,
		\begin{equation*}
			\norm{gz + y_0}
			\geq \norm{gz + \frac{y_0}{\norm{y_0}}} - \abs{1 - \norm{y_0}}
			> 2 - \eps - \eps = 2 - 2\eps.
		\end{equation*}
		Finally, we have
		\begin{align*}
			\norm{g + T}
			&\geq \norm{(g + T)z}
			= \norm{gz + Tz} \\
			&\geq \norm{gz + y_0} - \norm{Tz - y_0}
			> 2 - 2\eps - \eps = 2 - 3\eps.
	\end{align*}\end{proof}

	\begin{theorem} \label{thm:G-DPandRNP}
		Let $X$ be a~$G$-Banach space. If $X$ has both the RNP and the $G$-DPr, then $n_{G}(X) = 1$. In particular, we have that $n_{\iso(L^p[0,1])}(L^p[0,1]) = 1$  for every $1 < p < \infty$.
	\end{theorem}
	
	\begin{proof} The result follows from Theorem \ref{thm:GDPandGRNP} below. Nevertheless, the reader who does not want to read up all the way to the $G$-RNP may simply drop the $H$-equivariance in the proof of \Cref{thm:GDPandGRNP} and replace \cite[Theorem 6.4]{DantasDouchaJungRaunig2025} with the well-known equivalence of dentability and the RNP to directly obtain a proof for the present statement. The ``in particular'' part follows from \Cref{theorem:GDPrOnUniformlyConvexSpace}.
	\end{proof}

	We recover the following classical results.
	
	\begin{corollary} \label{classical-RNPs} Let $X$ be a Banach space.
		\begin{itemize}
			\item[(a)] If $X$ has the RNP and the aDPr, then $n(X) = 1$.
			\item[(b)] If $X$ has the DPr, then $X$ does not have the RNP.
		\end{itemize}
	\end{corollary}
	
	\begin{proof} The point (a) is a direct consequence of \Cref{thm:G-DPandRNP} by choosing $G = S_\K$. For (b), assume that $X$ has the RNP and suppose that $X$ has the DPr (that is, the $G$-DPr for $G =\{\id_X\}$). Then, by \Cref{thm:G-DPandRNP}, we have that $n_{\{\id_X\}}(X) = 1$. However, by \Cref{rem:GnumIndex}, we have that $n_{\{\id\}}(X) = -1$, a~contradiction.
	\end{proof}
	
	The classical RNP will likely be of main interest, but given the spirit of this paper, we feel obliged to present the following result in its full ``group-generality'' and only then deduce the classical RNP as a~special case. Since the definition of the group-equivariant RNP takes some work, we do not repeat it here and instead refer the interested reader to~\cite{DantasDouchaJungRaunig2025}. Notice that the action below is required to be continuous.
	
	\begin{theorem} \label{thm:GDPandGRNP}
		Let $G$ and $H$ be two groups acting continuously on a~Banach space $X$ by linear isometries. If $X$ has the $G$-Daugavet property and every $H$-invariant bounded set is dentable, then $n_{G, H}(X) = 1$.
		In particular, if $H$ is locally compact and $\sigma$-compact, $X$ has the $G$-Daugavet property and the $H$-RNP, then $n_{G, H}(X) = 1$.
	\end{theorem}
	\begin{proof}
		For every $T \in \mathcal{L}^H(X)$ we have that $K = \overline{T(B_X)}$ is bounded and $H$-invariant and hence dentable, implying that $K = \clco (\dent K)$ (see \cite[III.3]{GGMS87}). From \Cref{prop:sRN-Operators}.(2) and \Cref{lem:G-DP-numerical-radius}, we get that $v_G(T) = \norm{T}$ for every $H$-equivariant operator $T$. Thus, $n_{G, H}(X) = 1$.
		The ``in particular'' part follows from \cite[Theorem~6.4]{DantasDouchaJungRaunig2025}.
	\end{proof}
	
	\subsection{SCD operators and the $G$-DPr} Next, we show that in spaces with the $G$-Daugavet property, even wider classes of operators satisfy the Daugavet equation. The following are $G$-invariant adaptations of \cite{AKMMS}.
	
	Let us denote by
	\begin{equation*}
		K(X^*) = S_{X^*} \cap \overline{\ext(B_{X^*})}^{w^*}.
	\end{equation*}
	For every slice $S$ of $B_X$ and every $\e>0$, consider the set
	\begin{eqnarray*}
		D_G(S, \e)
		&:=&
		\{y^* \in K(X^*): S \cap G \cdot S(B_X, y^*, \e) \not= \emptyset\} \\
		&\subseteq&
		\{y^* \in K(X^*): S \cap \overline{\co}_G (S(B_X, y^*, \e)) \not= \emptyset \}.
	\end{eqnarray*}
	
	\begin{proposition} \label{prop:GDPandSCDslices}
		Let $X$ be a $G$-Banach space. The following statements are equivalent.
		\begin{itemize}
			\item[(i)] $X$ has the $G$-DPr.
			\item[(ii)] For every $x \in S_X$, for every $\e > 0$ and for every slice $S$ of $B_X$, there exists $y^* \in K(X^*)$ such that $x \in S(B_X, y^*, \e)$ and $S \cap G \cdot S(B_X, y^*, \e) \not= \emptyset$.
			\item[(iii)] For every $x \in S_X$, for every $\e > 0$ and for every slice $S$ of $B_X$, there exists $y^* \in D_G(S, \e)$ such that $x \in S(B_X, y^*, \e)$.
			\item[(iv)] For every $\e > 0$ and for every slice $S$ of $B_X$, the set $D_G(S, \e)$ is $w^*$-dense in $K(X^*)$.
			\item[(v)] For every $\e > 0$ and for every sequence $\{S_n: n \in \N\}$ of slices of $B_X$, the set $\bigcap_{n \in \N} D_G(S_n, \e)$ is $w^*$-dense in $K(X^*)$.
		\end{itemize}
	\end{proposition}
	
	\begin{proof}
		The equivalence of (i), (ii) and (iii) is an immediate consequence of item (2) of Proposition \ref{prop:characterization-G-DP}.
		
		Let us prove $\textup{(iii)} \Rightarrow \textup{(iv)}$.
		Fix $\e > 0$ and a slice $S$ of $B_X$.
		As in \cite[Proposition 4.2]{AKMMS}, it is enough to show that $\overline{D_G(S,\e)}^{w^*} \supseteq \ext(B_{X^*})$. Indeed, if this holds, then
		\begin{equation*}
			\overline{D_G(S,\e)}^{w^*}
			\supseteq
			\overline{\ext(B_{X^*})}^{w^*} \cap S_{X^*}
			=
			K(X^*).
		\end{equation*}
		So let $x^* \in \ext(B_{X^*})$, let $x \in S_X$ and let $\delta \in (0,\e)$.
		We will show that the $w^*$-slice
		\begin{equation*}
			S(B_{X^*}, x, \delta)
			=
			\{z^* \in B_{X^*}: \re z^*(x) > 1-\delta\}
		\end{equation*}
		intersects $D_G(S,\e)$.
		By (iii), applied to $x$, $\delta$ and the slice $S$, there exists $y^* \in D_G(S,\delta)$ such that $x \in S(B_X, y^*, \delta)$, i.e. $\re y^*(x) > 1-\delta$. Hence $y^* \in S(B_{X^*}, x, \delta)$, and since $\delta < \e$, we clearly have $D_G(S,\delta) \subseteq D_G(S,\e)$. Therefore,  $S(B_{X^*}, x, \delta) \cap D_G(S,\e) \not= \emptyset$, and so $x^* \in \overline{D_G(S,\e)}^{w^*}$.
		
		Next, let us prove $\textup{(iv)} \Rightarrow \textup{(iii)}$.
		Fix $x \in S_X$, $\e > 0$ and a slice $S$ of $B_X$.
		Since $D_G(S,\e)$ is $w^*$-dense in $K(X^*)$ and $K(X^*)$ is norming for $X$, we can choose $y^* \in D_G(S,\e) \cap S(B_{X^*}, x, \e)$. Then $\re y^*(x) > 1-\e$, that is, $x \in S(B_X, y^*, \e)$, proving (iii).
		
		To prove $\textup{(iv)} \Rightarrow \textup{(v)}$, it is enough to verify that $D_G(S,\e)$ is relatively $w^*$-open in $K(X^*)$.
		Indeed, once this is done, the conclusion follows from Baire's theorem, exactly as in \cite[Proposition 4.2]{AKMMS}, because $K(X^*)$ is a Baire space. So fix $\e > 0$, a slice $S$ of $B_X$, and let $y^* \in D_G(S,\e)$. Then there exists $y \in S \cap G \cdot S(B_X, y^*, \e)$. Hence there exist $g \in G$ and $\eta > 0$ such that $\re y^*(gy) > 1-\e+\eta$. Consider the $w^*$-neighbourhood of $0$
		\begin{equation*}
			U:=U_{gy,\eta}
			:=
			\{x^* \in X^*: |x^*(gy)| < \eta\}.
		\end{equation*}
		We claim that $(y^*+U)\cap K(X^*) \subseteq D_G(S,\e)$. Indeed, if $z^* \in (y^*+U)\cap K(X^*)$, then  $|(z^*-y^*)(gy)|<\eta$  and therefore
		\begin{equation*}
			\re z^*(gy)
			\geq
			\re y^*(gy) - |(z^*-y^*)(gy)|
			> 1-\e.
		\end{equation*}
		Thus $gy \in S(B_X, z^*, \e)$, which means precisely that $y \in S \cap G \cdot S(B_X, z^*, \e)$. Hence $z^* \in D_G(S,\e)$, as claimed.
		
		Finally, $\textup{(v)} \Rightarrow \textup{(iv)}$ is immediate.
	\end{proof}
	
	Let us recall the definition of a \emph{slicely countably determined} (SCD, for short) operator introduced in~\cite{AKMMS}.
	Given two Banach spaces $X$ and $Y$, a~convex bounded set $A \subseteq Y$ and a~bounded linear operator $T: X \to Y$, we say that $A$ is an SCD set if there is a~countable family $\{S_n : n \in \N\}$ of slices of $A$ such that $A \subseteq \clco(B)$ for every $B \subset A$ intersecting all the sets $S_n$. We call $\{S_n\}$ a determining sequence of slices for $A$. We call $T$ an SCD operator if $T(B_X)$ is an SCD set.
	
	\begin{proposition} \label{prop:G-DPandSCDoperators}
		Assume that $X$ has the $G$-DPr and consider $T \in \mathcal{L}(X)$ to be an SCD-operator. Then,
		\begin{equation*}
			\sup_{g \in G} \|g + T\| = 1 + \|T\|.
		\end{equation*}
	\end{proposition}
	
	\begin{proof}
		We may assume that $\|T\|=1$.
		Let $(S_n)$ be a determining sequence of slices of $T(B_X)$.
		By \cite[Proposition 2.6.5]{KMRW}, the sets $T^{-1}(S_n)\cap B_X$ are slices of $B_X$. Fix $\e>0$, and choose $a \in S_X$ such that $\|T(a)\| > 1-\e$. By Proposition \ref{prop:GDPandSCDslices}, the set
		\begin{equation*}
			\bigcap_{n \in \N} D_G(T^{-1}(S_n)\cap B_X,\e)
		\end{equation*}
		is $w^*$-dense in $K(X^*)$.
		Since $K(X^*)$ is norming for $X$, we can find
		\begin{equation*}
			y^* \in \bigcap_{n \in \N} D_G(T^{-1}(S_n)\cap B_X,\e)
		\end{equation*}
		such that
		\begin{equation} \label{eq:GDP-SCD-op-1}
			\re y^*(T(a)) \geq \|T(a)\|-\e > 1-2\e.
		\end{equation}
		By the definition of $D_G$, for every $n \in \N$ we have $\bigl(T^{-1}(S_n)\cap B_X\bigr)\cap G\cdot S(B_X,y^*,\e)\not=\emptyset$. Hence, for every $n \in \N$, $S_n \cap T\bigl(G\cdot S(B_X,y^*,\e)\bigr)\not=\emptyset$. Since $(S_n)$ is determining for $T(B_X)$, it follows that $T(B_X)\subseteq \overline{\co}\bigl(T(G\cdot S(B_X,y^*,\e))\bigr)$. In particular, $T(a)\in \overline{\co}\bigl(T(G\cdot S(B_X,y^*,\e))\bigr)$. Therefore, there exists $z \in \co\bigl(T(G\cdot S(B_X,y^*,\e))\bigr)$ such that $\|T(a)-z\|<\e$. Together with \eqref{eq:GDP-SCD-op-1}, this gives
		\begin{equation} \label{eq:GDP-SCD-op-2}
			\re y^*(z) > 1-3\e.
		\end{equation}
		
		We may write
		\begin{equation*}
			z=\sum_{k=1}^m \lambda_k T(g_kx_k)
		\end{equation*}
		where $x_k \in S(B_X,y^*,\e)$, $g_k \in G$, $\lambda_k \geq 0$ for every $k=1,\ldots,m$, and $\sum_{k=1}^m \lambda_k=1$.
		By \eqref{eq:GDP-SCD-op-2}, there exists $k_0 \in \{1,\ldots,m\}$ such that $\re y^*(T(g_{k_0}x_{k_0})) > 1-3\e$. Since $x_{k_0} \in S(B_X,y^*,\e)$, we also have $\re y^*(x_{k_0}) > 1-\e$. Hence $\re y^*\bigl(x_{k_0}+T(g_{k_0}x_{k_0})\bigr) > 2-4\e$. Using Proposition \ref{fact:G-DP-equation-commutes}, we obtain
		\begin{eqnarray*}
			\sup_{g \in G} \|g + T\|
			= \sup_{g \in G} \|\id + T \circ g\|
			&\geq&
			\|\id + T\circ g_{k_0}\| \\
			&\geq&
			\|x_{k_0}+T(g_{k_0}x_{k_0})\| \\
			&\geq&
			\re y^*\bigl(x_{k_0}+T(g_{k_0}x_{k_0})\bigr)
			>
			2-4\e.
		\end{eqnarray*}
		Since $\e>0$ was arbitrary, the result follows.
	\end{proof}
	
	Since SCD operators must have separable range, the use-cases of the above proposition are somewhat limited. However, we can employ the same trick as in \cite[Corollary~5.5]{AKMMS} to obtain information about operators with non-separable range.
	
	Before doing so, we need a small lemma.
	\begin{lemma} \label{prop:stabilityForIncreasingSeqOfSubspaces}
		Let $X$ be a~$G$-Banach space and $X_n$, $n \in \N$, be a sequence of increasing closed $G$-invariant subspaces of $X$. If every $X_n$ has the $G$-DPr, then so does $Y = \overline{\bigcup_{n=1}^\infty X_n}$.
	\end{lemma}
	\begin{proof}
		First of all, notice that $Y$ is $G$-invariant, so it makes sense to talk about it having the $G$-DPr as a~$G$-Banach space. Let $T: Y \to Y$ be a~rank-one operator with $\norm{T} = 1$, that is, there are $y \in S_Y$ and $y^* \in S_{Y^*}$ such that $T = y^* \otimes y$. Let $\eps > 0$ be arbitrary. Then there is $x \in \bigcup_{n=1}^\infty X_n$ such that $\norm{y-x} < \eps$. Find $n \in \N$ such that $x \in X_n$ and also $\norm{y^*|_{X_n}} > 1-\eps$ (to do so, one finds $z \in Y$ with $y^*(z) > 1-\eps/2$ and $\eps/2$-approximates $z$ with an element from $\bigcup_{n=1}^\infty X_n$). Define $S: X_n \to X_n$ by $S = y^*|_{X_n} \otimes x$. Then
		\begin{equation*}
			\norm{S} = \norm{y^*|_{X_n}} \norm{x} > (1-\eps)(1-\eps) > 1-2\eps.
		\end{equation*}
		Since $X_n$ has the $G$-DPr, there are $z \in S_{X_n}$ and $g \in G$ such that $\norm{gz + Sz} > 2 - 3\eps$.
		We clearly have
		\begin{equation*}
			\norm{Sz - Tz}
			= \norm{y^*(z)x - y^*(z)y}
			= \abs{y^*(z)}\norm{x-y}
			< \eps.
		\end{equation*}
		Putting these together, we obtain
		\begin{equation*}
			\norm{gz + Tz} \geq \norm{gz + Sz} - \norm{Sz - Tz} > 2-3\eps - \eps
		\end{equation*}
		which shows that $\norm{g+T} > 2 - 4\eps$. Taking the supremum over $g \in G$ we get the desired equality $\sup_{g \in G} \norm{g + T} = 2$.
	\end{proof}
	
	One may hope to remove the assumption that the subspaces $X_n$ are increasing and instead take $Y = \clspan \bigcup_{n=1}^\infty X_n$. A~simple example shows that this is not possible. Let $X_n = \R$, $X = (\sum_{n=1}^\infty X_n)_{\ell_p}$ with $p \in (1, \infty) \setminus \{2\}$ and $G = \{\id, -\id\}$. Then the spaces $X_n$ can be naturally identified with $G$-invariant subspaces of $X$ and it is well-known that they have the $G$-DPr (which coincides with the aDP for this choice of $G$). It is also clear that $X = \clspan \bigcup_{n=1}^\infty X_n$. So, if this more general statement was true, it would imply that $X = \ell_p$ had the aDPr, however that is not the case by \Cref{cor:Lp-ellp-examples}.
	
	\begin{corollary} \label{cor:GDP-SCD-nonseparable}
		Let $X$ be a $G$-Banach space for some topological group $G$ and assume that the action is continuous. If $X$ has the $G$-DPr and $T \in \mathcal{L}(X)$ is an operator such that $T(B_Y)$ is an SCD set for every separable $G$-invariant subspace $Y$ of $X$, then
		\begin{equation*}
			\sup_{g \in G} \|g + T\| = 1 + \|T\|.
		\end{equation*}
	\end{corollary}
	
	\begin{proof}
		First, choose a separable subspace $Y_1 \subseteq X$ such that
		\begin{equation*}
			\|T|_{Y_1}\| = \|T\|.
		\end{equation*}
		Using the separable determination theorem for the $G$-DPr, we can recursively construct separable subspaces $Y_n$ of $X$ and separable subgroups $G_n \leq G$ such that, for every $n \in \N$,
		\begin{itemize}
			\item $Y_n$ is $G_n$-invariant,
			\item $Y_n$ has the $G_n$-DPr,
			\item $Y_n \subseteq Y_{n+1}$ and $G_n \leq G_{n+1}$,
			\item $T(Y_n) \subseteq Y_{n+1}$.
		\end{itemize}
		Let
		\begin{equation*}
			Y = \overline{\bigcup_{n=1}^\infty Y_n}
			\qquad\text{and}\qquad
			H = \bigcup_{n=1}^\infty G_n.
		\end{equation*}
		Then $Y$ is a separable $H$-invariant subspace, $T(Y)\subseteq Y$, and, by \Cref{prop:stabilityForIncreasingSeqOfSubspaces}, $Y$ has the $H$-DPr.
		Also, $\|T|_Y\| = \|T\|$. By hypothesis, $T(B_Y)$ is an SCD set, so Proposition \ref{prop:G-DPandSCDoperators} applied to the $H$-Banach space $Y$ gives
		\begin{equation*}
			\sup_{h \in H} \|h + T|_Y\| = 1 + \|T|_Y\| = 1 + \|T\|.
		\end{equation*}
		Since $H \leq G$, we finally get
		\begin{equation*}
			\sup_{g \in G} \|g + T\|
			\geq
			\sup_{h \in H} \|h + T|_Y\|
			=
			1+\|T\|.
		\end{equation*}
		The reverse inequality is automatic.
	\end{proof}
	
	In the same way as for the aDPr, we obtain the following corollary.
	
	\begin{corollary} \label{result-contain-copies-of-l1}
		Let $X$ be a $G$-Banach space and assume the action to be continuous. If $X$ has the $G$-DPr and does not contain copies of $\ell_1$, then $n_{G}(X) = 1$.
	\end{corollary}
	
	\begin{proof}
		Since $X$ does not contain copies of $\ell_1$, neither does any separable $G$-invariant subspace $Y$ of $X$. Hence, for every operator $T \in \mathcal{L}(X)$, the restriction $T|_Y$ does not fix copies of $\ell_1$, and therefore $T(B_Y)$ is an SCD set by \cite[Examples 10.4.2]{KMRW}.
		By Corollary \ref{cor:GDP-SCD-nonseparable}, every operator $T \in \mathcal{L}(X)$ satisfies
		\begin{equation*}
			\sup_{g \in G} \|g + T\| = 1 + \|T\|.
		\end{equation*}
		Using Lemma \ref{lem:G-DP-numerical-radius}, we get $v_G(T)=\|T\|$ for every $T \in \mathcal{L}(X)$, and so $n_{G}(X)=1$. 
	\end{proof}
	
	By simply setting $G = S_\K$, we recover the classical version of the statement for the aDPr. By setting $G = \{\id_X\}$ and combining with \Cref{rem:GnumIndex}, we recover the statement that a~Banach space with the DPr must contain a~copy of $\ell_1$.
	
	\subsection{Group lush spaces} In this short section, we prove that $G$-lush spaces satisfy $n_G(X) = 1$ and also that if $X$ has the $G$-DPr and its unit ball is an SCD set, then $X$ must be $G$-lush. This will help us to recover classical result on the theory.

	\begin{definition} Let $X$ be a $G$-Banach space. We say that $X$ is {\it $G$-lush} if for every $x, y \in S_X$ and every $\e > 0$, there is $y^* \in S_{Y^*}$ such that $y \in S(B_X, y^*, \e)$ and $\dist(x, \co_G(S(B_X, y^*, \e))) < \e$.
	\end{definition}

	Writing out the definitions, we see that $X$ is $G$-lush if and only if for every $x, y \in S_X$ and $\e>0$, there are $y^* \in S_{X^*}$, $x_1, \ldots, x_n \in S(B_X, y^*, \e)$, $g_1, \ldots, g_n \in G$ and $\lambda_1, \ldots, \lambda_n \geq 0$ with $\sum_{k=1}^n \lambda_k =1$ such that 
	\begin{equation*}
		\left\| x - \sum_{k=1}^n \lambda_k g_k x_k \right\| < \e.
	\end{equation*}

	The analogous to \cite[Proposition 2.2]{BoykoKadetsMartinWerner2007} is the following. As it had happened before in the present paper, the proof is analogous to the original one by keeping track of the action of the group on $X$. We present the proof for the sake of completeness.
	
	\begin{proposition}
		Let $X$ be a $G$-Banach space such that $X$ is $G$-lush. Then, $n_{G}(X)=1$.
	\end{proposition}
	
	\begin{proof} Let $T \in \mathcal{L}(X)$ with $\|T\|= 1$. It is enough to prove that $v_G(T) = 1$. By \Cref{lem:G-DP-numerical-radius}, it is enough to prove that
		\begin{equation*}
			\sup_{g \in G} \| \id_X + Tg\| = 1 + \|T\| = 2.
		\end{equation*}
		As in the original proof, pick $\e \in (0,\frac{1}{2})$ and $x_0 \in S_X$ to be such that $\|T(x_0)\| > 1 - \e$, and set $y_0 = \frac{T(x_0)}{\|T(x_0)\|} \in S_X$. We apply $G$-lushness to the pair $x_0, y_0$ to get a functional $y^* \in S_{X^*}$ such that $y_0 \in S(B_X, y^*, \e)$ and elements $x_1, \ldots, x_n \in S(B_X, y^*, \e)$, $g_1, \ldots, g_n \in G$ and convex coefficients $\lambda_1, \ldots, \lambda_n \geq 0$ such that 
		\begin{equation*}
			v:= \sum_{k=1}^n \lambda_k g_k x_k \ \ \ \mbox{is such that} \ \ \ \|x_0 - v\| < \e.
		\end{equation*}
		Since $y_0 \in S(B_X, y^*, \e)$, we have that $\re y^*(y_0) > 1 - \e$ and since 
		\begin{equation*}
			\re y^*(T(v)) = \re y^* \left( \frac{T(x_0)}{\|T(x_0)\|}\right) - \re y^* \left( T \left( \frac{x_0}{\|T(x_0)\|} - v \right) \right) 
		\end{equation*}
		we have that 
		\begin{equation*}
			\re y^*(T(v)) \geq \re y^*(y_0) - \left\| \frac{x_0}{\|T(x_0)\|} - v \right\|.
		\end{equation*}
		
		Once again as in the original proof, we get that $|y^*(T(v))| > 1 - 4 \e$ and by a convex argument, there exists $j \in \{1, \ldots, n\}$ such that $\re y^*(T(g_j x_j)) > 1 - 4\e$. As $x_j \in S(B_X, y^*, \e)$, we have that 
		\begin{equation*}
			\| \id_X + T g_j\| \geq \re y^*(x_j) + \re y^*(T(g_j x_j)) > 2 - 5\e.
		\end{equation*}
		Since $T \in \mathcal{L}(X)$ was arbitrary, we have that $n_{G}(X) = 1$.
	\end{proof}

	We now present the $G$-version of~\cite[Theorem~4.4]{AKMMS}. We will use it in \Cref{theorem:all-together} to prove that a $G$-Banach space with the $G$-DPr and an additional condition on the group $G$ cannot be such that its unit ball is SCD.

	\begin{proposition} \label{G-DPr-G-lush}
		Let $X$ be a~$G$-Banach space. If $X$ has the $G$-DPr and its unit ball is an~SCD set, then $X$ is $G$-lush. In particular, $n_G(X) = 1$.
	\end{proposition}
	We will again follow the classical argument presented in the aforementioned paper.
	\begin{proof}
		From the definition, it directly follows that it suffices to show that
		\begin{equation*}
			\forall \eps > 0 \ \forall x \in S_X \ \exists y^* \in S_{X^*} : x \in S(B_X, y^*, \eps) \text{ and } \clco_G S(B_X, y^*, \eps) = B_X.
		\end{equation*}
		
		So, let $\eps > 0$ and $x \in S_X$ be arbitrary. Let $\{S_n\}$ be a determining sequence of slices for $B_X$. By \Cref{prop:GDPandSCDslices}~(v), $\bigcap_{n\in\N} D_G(S_n, \eps)$ is $w^*$-dense in $K(X^*)$ which is norming and hence there is $y^* \in \bigcap_{n\in\N} D_G(S_n, \eps)$ such that $x \in S(B_X, y^*, \eps)$. From the definition of $D_G(S_n, \eps)$ it follows that $S_n \cap S_G(B_X, y^*, \eps) \neq \emptyset$ for all $n \in \N$. And since $\{S_n\}$ is a determining sequence of slices, the definition of SCD sets gives that $B_X = \clco S_G(B_X, y^*, \eps) = \clco_G S(B_X, y^*, \eps)$.
	\end{proof}
	
	We can, of course, recover the standard results from the previous propositions -- an SCD space with the aDPr must have numerical index 1 and the properties of being SCD and having the DPr are mutually exclusive (because they would together imply that the $\{\id\}$-numerical index of the space would be 1, which is impossible).
	
	\section{On some conditions on $G$ which provide DPr-like behavior}

	It is natural to ask under which additional restrictions on the action one can recover classical consequences of the Daugavet property from its group counterpart. The content of this section sheds some light on the answer.

	\subsection{Recovering classical results} The difficulty is that the group action itself may produce almost antipodal points and therefore yield what one might call ``fake Daugavet behaviour'', even in spaces with very regular geometry (see, for instance, \Cref{theorem:GDPrOnUniformlyConvexSpace} and its consequences). A convenient way to exclude this phenomenon is to measure how far the action can move unit vectors. To this end, for a $G$-Banach space $X$, let us define the number
	\begin{equation*}
		\alpha_G(X) := \sup \{ \|gx-x\| : g \in G,\ x \in S_X \}.
	\end{equation*}
	Notice that when $G=\{\id\}$, one has $\alpha_G(X)=0$ (in fact, if $\alpha_G(X) = 0$, then $g \cdot x = x$ for every $g \in G$ and every $x \in X$, which implies that the action is trivial) and by the triangle inequality we always have $\alpha_G(X) \leq 2$. On the other hand, if $S_{\K} \subseteq G$, then $-\id_X \in G$ and, therefore, for every $x \in S_X$, one has $\alpha_G(X) \geq \|-\id_X(x)-x\|=2$, that is, $\alpha_G(X) = 2$. Clearly if $H \leq G$, then $\alpha_H(X) \leq \alpha_G(X)$.

	The condition $\alpha_G(X)<2$ can be interpreted as a uniform non-antipodality assumption on the action: no orbit contains points arbitrarily close to the antipode of a unit vector. \Cref{thm:GDPr-alphaG} below extends the classical implication ``DPr implies failure of the RNP'' (see also \Cref{thm:G-DPandRNP}). Before proving it, let us note that the assumptions $G$-DPr together with $\alpha_G(X)<2$ are easy to satisfy.
	
	\begin{example} \label{example:C3}
		Let $C_3=\{1,\omega,\omega^2\}$ be the cyclic group of order $3$, where $\omega=e^{2\pi i/3}$ and let $X=C[0,1]$ over the complex scalars. Define an action of $C_3$ on $X$ by scalar multiplication, that is, $g\cdot f = gf$ for every $g\in C_3$ and $f\in X$. Each $g$ acts as a surjective linear isometry on $X$. Moreover, since $C[0,1]$ has the classical DPr, it has the $C_3$-DPr as well by \Cref{DPr-implies-G-DPr}. Finally,
		\begin{eqnarray*}
			\alpha_{C_3}(X)
			&=& \sup \{ \|g\cdot f-f\| : g\in C_3,\ f\in S_X \} \\
			&=& \max \{|g-1|: g\in C_3 \} \\
			&=& |\omega-1|
			= \sqrt{3}
			< 2.
		\end{eqnarray*}
	\end{example}
	
	Now we prove \Cref{thm:GDPr-alphaG}. As a particular case of this result, we will get that $\alpha_{\Iso(L^p[0,1])}(L^p[0,1])=2$ for every $1 < p < \infty$ by using \Cref{cor:Lp-ellp-examples}.(a). On the other hand, as excepted, \Cref{thm:GDPr-alphaG} does not apply to the aDPr as in that case $\alpha_{S_{\K}}(X) = 2$. Notice that item (a) of \Cref{thm:GDPr-alphaG} also follows from Proposition \ref{relation-between-alpha-and-numerical-index} below with a different proof.
	
	\begin{theorem}\label{thm:GDPr-alphaG}
		Let $X$ be a $G$-Banach space with the $G$-DPr and suppose $\alpha_G(X)<2$. 
		
		\begin{itemize} 
			\item[(a)] Then, $X$ does not have the RNP. In particular, $X$ is not reflexive.
			\item[(b)] If $X$ has an unconditional basis, then $X$ contains a copy of $c_0$.
		\end{itemize}
	\end{theorem}
	\begin{proof}
		Assume that $X$ has the RNP. Then $B_X$ is dentable, so there exists a denting point $x\in B_X$. Since every denting point of $B_X$ is extreme, necessarily $x\in S_X$. Set
		\begin{equation*}
			\eta := \frac{2-\alpha_G(X)}{3}>0.
		\end{equation*}
		Because $x$ is a denting point of $B_X$, there exists a slice
		\begin{equation*}
			S := S(B_X,x^*,\e)
			= \{ z\in B_X : \re x^*(z)>1-\e \}
		\end{equation*}
		such that $x\in S$ and $\diam(S)<\eta$. Consider now the associated $G$-slice $S_G := S_G(B_X,x^*,\e)$. We claim that $S_G \cap Q(-x,\eta)=\emptyset$, where
		\begin{eqnarray*}
			Q(-x,\eta)
			&=& \{ y\in B_X : \|-x+y\|>2-\eta \} \\
			&=& \{ y\in B_X : \|y-x\|>2-\eta \}.
		\end{eqnarray*}
		Indeed, let $y\in S_G$. Then there exists $g\in G$ such that $gy\in S$. Since $x\in S$ and $\diam(S)<\eta$, we have $\|gy-x\|<\eta$. Therefore,
		\begin{align*}
			\|y-x\|
			&\leq \|y-g^{-1}x\|+\|g^{-1}x-x\| \\
			&= \|gy-x\|+\|g^{-1}x-x\| \\
			&< \eta+\alpha_G(X) = 2-2\eta < 2-\eta.
		\end{align*}
		Hence $\|y-x\|<2-\eta$, so $y\notin Q(-x,\eta)$. This proves the claim. On the other hand, since $X$ has the $G$-DPr, the slice characterization yields $S_G \cap Q(-x,\eta)\neq\emptyset$, which contradicts the claim above. This proves (a). For (b), recall that if an unconditional basis is not boundedly complete, then the space must contain $c_0$ (\cite{{James1950}}, see also \cite[Theorem 1.c.10]{LindenstraussTzafriri1977}). On the other hand, if the space is boundedly complete, then it is a dual with an unconditional basis and therefore it has the RNP. This means that if $X$ is a $G$-Banach space with an unconditional basis satisfying the $G$-DPr and such that $\alpha_G(X) < 2$, by (a), $X$ cannot have the RNP and then it must contain $c_0$. 
	\end{proof}
	
	Note that \Cref{thm:GDPr-alphaG} provides yet another tool to deduce that a~space with the DPr cannot have the RNP since, as discussed above, for any $X$ we have $\alpha_{\{\id_X\}}(X) = 0$. Moreover, we also have the following consequence of \Cref{thm:GDPr-alphaG} (remember \Cref{thm:G-DPandRNP} as well).

	\begin{corollary} In any reflexive $G$-space $X$ with the $G$-Daugavet property, we have $\alpha_G(X) = 2$ and $n_G(X) = 1$.
	\end{corollary}

	There is no hope to turn \Cref{thm:GDPr-alphaG} into an equivalence as we can see below.
	
	\begin{remark} \label{example-of-alpha-equals-sqrt3} One might think that a Banach space $X$ with the $G$-DPr and without the RNP could yield $\alpha_G(X) < 2$. However, this is not true in general. Indeed, consider $X = C[0,1]$ and $G = S_{\C}$. Then, $X$ has the $G$-DPr (in fact, this is nothing but the aDPr), it fails the RNP and $\alpha_G(X)=2$ as $-1 \in S_{\C}$ and, for every $f \in S_X$, we have that $\|(-1)f - f\| = 2$. Another possibility would be to have a result which says that $\alpha_G(X) < 2$ together with non-RNP implies the $G$-DPr. This is not the case either. Indeed, let $G=C_3$ act on $X=C[0,1]\oplus_1 \mathbb C$ by $g\cdot(f,\lambda)=(gf,\lambda)$. Then $\alpha_G(X)=\sqrt3<2$ (see Example \ref{example:C3} above) and $X$ does not have the RNP because it contains $C[0,1]$ as an $\ell_1$-summand. Moreover, $X$ does not have the $G$-Daugavet property. Indeed, the operator $T(f,\lambda)=(0,-\lambda)$ is rank-one and satisfies $\|T\|=1$. Since the action on the second summand is trivial, we have $g\circ T=T$ for every $g\in G$. Therefore $(\id+g\circ T)(f,\lambda)=(f,0)$, so $\|\id+g\circ T\|=1$ for every $g\in G$. Hence $\sup_{g\in G}\|\id+g\circ T\|=1<2=1+\|T\|$, and thus $X$ fails the $G$-DPr.
	\end{remark}

	It is also natural to wonder the following. Assume that $\alpha_G(X) = 2$  and suppose that it is never attained. Could we still yield a Banach space without the RNP whenever it satisfies the $G$-DPr? This is false in general as we can see in the next straightforward example.
	
	\begin{example} \label{MG(X)-example} Let $X=\mathbb{R}^2$ with the Euclidean norm, $\theta\in\R$ be such that $\theta/\pi \notin \Q$, and let $G=\langle R_\theta\rangle=\{R_{n\theta}:n\in\Z\}$, where $R_\theta$ denotes rotation by angle $\theta$. Then the orbit of every $x \in S_X$ is dense in the unit circle, so the action is almost transitive. Since $X$ is a Hilbert space, hence LUR, it follows from \Cref{theorem:GDPrOnUniformlyConvexSpace} that $X$ has the $G$-DPr. On the other hand, $X$ is finite-dimensional, and therefore it has the RNP. Now, for every $x \in S_X$ and every $n \in \mathbb{Z}$, we have $\|R_{n\theta}x-x\|<2$, because the equality $\|R_{n\theta}x-x\|=2$ would imply $R_{n\theta}x=-x$, that is, $n\theta \equiv \pi \pmod{2\pi}$, which is impossible since $\theta/\pi$ is irrational. However, the set $\{n\theta \mod 2\pi : n \in \mathbb{Z}\}$ is dense in $[0,2\pi)$, so for every $x \in S_X$ one can find a sequence $(n_k)$ such that $R_{n_k\theta}x \to -x$. Hence $\alpha_G(X) = 2$ (another way of seeing this is applying \Cref{thm:GDPr-alphaG}.(a)) and this supremum is never attained.
	\end{example}
	
	In what follows, we use once again the condition $\alpha_G(X) < 2$ to recover back the classical result about the DPr which says that Banach spaces with the DPr should contain copies of $\ell_1$. The following is an easy consequence of \Cref{result-contain-copies-of-l1}. In particular, we get once again \Cref{thm:GDPr-alphaG}.(a).
	
	\begin{proposition} \label{relation-between-alpha-and-numerical-index} Let $X$ be a $G$-Banach space. If $\alpha_G(X) < 2$, then $n_G(X) < 1$. 
	\end{proposition}
	
	\begin{proof} Consider $T:= -\id_X$. Then $\|T\| = 1$ and, for every $x \in S_X$ and $x^* \in S_{X^*}$ with $x^*(x) = 1$, and for every $g \in G$, we have that 
		\begin{equation*}
			\re x^*(g \cdot x) = \re x^*(x) + \re x^*(g \cdot x - x) \geq 1 - \|g \cdot x - x\| \geq  1 - \alpha_G(X),
		\end{equation*}
		which shows that $\re x^*(T(x)) = \re x^*(-g \cdot x) \leq \alpha_G(X) - 1$. Taking the supremum over $g \in G$, we get that $v_G(T) \leq \alpha_G(X) - 1 < 1$, that is, $n_G(X) \leq v_G(T) < 1$. 
	\end{proof}
	
	Now we are able to put several results we have collected throughout the paper to get the following result. Notice that this recovers back the classical Daugavet spirit.
	
	\begin{theorem} \label{theorem:all-together} Let $X$ be a $G$-Banach space with the $G$-DPr. Suppose that $\alpha_G(X) < 2$. Then, the following holds true.
		\begin{itemize}
			\item[(1)] If the action is continuous, $X$ must contain a copy of $\ell_1$.
			\item[(2)] $X$ does not have the RNP.
			\item[(3)] $B_X$ is not an SCD set.
		\end{itemize}
	\end{theorem}
	
	\begin{proof} For all the items we apply \Cref{relation-between-alpha-and-numerical-index} above. Indeed, item (1) is a consequence of \Cref{result-contain-copies-of-l1}, item (2) is a consequence of \Cref{thm:G-DPandRNP} and item (3) is a consequence of \Cref{G-DPr-G-lush}.
	\end{proof}
	
	\begin{remark} Notice that the converse of \Cref{relation-between-alpha-and-numerical-index} does not hold in general even in simple situations. Indeed, consider a real Hilbert space $X$ with the scalar action $G = \{ \id_X, - \id_X \}$. Taking $g=-\id_X$, we can see easily that $\alpha_G(X) \geq \|(-x) - x\| = 2$ for every $x \in S_X$ while $v_G(T) = v(T)$ for every $T \in \mathcal{L}(X)$. This shows that $n_G(X) = n(X) = 0$ while $\alpha_G(X) = 2$. Another simple example is when one considers a complex Hilbert space $X$ and $G= S_{\C}$ acting by scalar multiplication. While $n(X) = 1/2$, we have that $\alpha_G(X) = 2$ as $-1 \in S_{\C}$ and once again $v_G(T) = v(T)$ for every $T \in \mathcal{L}(X)$, that is, $n_G(X) = n(X) = 1/2$. In fact, we have that 
		\begin{itemize}
			\item $\alpha_G(X) = 2$ and $n_G(X) = 0$ can happen (real Hilbert spaces with $G=\{\id_X, - \id_X\}$ acting by scalar multiplication).
			\item $\alpha_G(X) = 2$ and $0<n_G(X)<1$ can happen (infinite-dimensional complex Hilbert spaces with $G=S_{\C}$ acting by scalar multiplication).
			\item $\alpha_G(X) = 2$ and $n_G(X) = 1$ can happen (considering $X=L^p[0,1]$ and $G=\iso(L^p[0,1])$ and applying \Cref{thm:GDPr-alphaG} and \Cref{thm:G-DPandRNP}, respectively).
			\item $\alpha_G(X) < 2$ and $n_G(X) = 1$ cannot happen (this is just a directly consequence of \Cref{relation-between-alpha-and-numerical-index}).
		\end{itemize}
	\end{remark}
	
	\begin{remark}
		Even when $X$ satisfies the $G$-DPr, the equality $\alpha_G(X)=2$ does not force the $G$-numerical index to be $1$. Indeed, let $H$ be a real Hilbert space of dimension at least $2$, let $K$ be perfect  and consider the $G$-Banach space $X=C(K,H)$ endowed with the action of $G=\{\id_X,-\id_X\}$ by scalar multiplication. Since $K$ is perfect, $X$ has the classical Daugavet property and therefore also the $G$-DPr. Moreover, as $-\id_X \in G$, we have $\alpha_G(X)=2$. On the other hand, for this action the $G$-numerical radius coincides with the classical numerical radius, so $n_G(X)=n(X)$. Since $n(C(K,H))=n(H)$ (see \cite[Theorem 5]{MartinPaya2000}) and $n(H)=0$, it follows that $n_G(X)=0$
	\end{remark}

	\subsection{Quantitative results}  Let us notice the following estimation for $\alpha_G(X)$ - see \Cref{proposition-estimation-alpha-G} below. For every operator $T \in \mathcal{L}(X)$ and $g \in G$, we have that 
	\begin{equation*}
		\|(\id_X + gT) - (\id_X + T)\| = \|gT - T\| \leq \alpha_G(X) \|T\|.
	\end{equation*}
	This shows that 
	\begin{equation*}
		\|\id_X + gT\| \leq \|\id_X + T\| + \alpha_G(X) \|T\|. 
	\end{equation*}
	Taking the supremum over $g \in G$, we get that 
	\begin{equation*}
		\sup_{g \in G} \| \id_X + gT \| \leq \|\id_X + T\| + \alpha_G(X) \|T\|
	\end{equation*}
	for every $T \in \mathcal{L}(X)$ and every $g \in G$. Now, assume that $X$ has the $G$-DPr. Then, 
	\begin{equation*}
		1 + \|T\| = \sup_{g \in G} \|\id_X + gT\| \leq \|\id_X + T\| + \alpha_G(X) \|T\|
	\end{equation*}
	for every rank-one operator $T \in \mathcal{L}(X)$. Rearranging this last inequality, we get that
	\begin{equation*}
		\|\id_X + T\| \geq 1 + (1 - \alpha_G(X))\|T\| 
	\end{equation*}
	for every rank-one operator $T \in \mathcal{L}(X)$. In particular, if $\alpha_G(X) = 0$, then $X$ has the classical DPr. In other words, we have the following result, which might be interpreted as a way of measuring how far a space with the $G$-DPr is of having the classical DPr.
	
	\begin{proposition} \label{proposition-estimation-alpha-G} Let $X$ be a $G$-Banach space. If $X$ has the $G$-DPr, then 
		\begin{equation} \label{inequality-estimation-alpha-G}
			\| \id_X + T\| \geq 1 + (1 - \alpha_G(X)) \|T\| 
		\end{equation}
		for every rank-one $T \in \mathcal{L}(X)$.
	\end{proposition}

	Let us notice in what follows that the estimate (\ref{inequality-estimation-alpha-G}) does not seem to provide back the classical DPr by assuming that the coefficient $1-\alpha_G(X)$ is arbitrarily close to 1. The reason is that there exist Banach spaces which satisfy a uniform lower bound of the same form for all rank-one operators while still failing the DPr. This follows from \cite{ChoiJung2024} and \cite{HallerLangemetsLimaNadelRuedaZoca2021}. In \cite{HallerLangemetsLimaNadelRuedaZoca2021} the following definition of a Daugavet index was given. For a {\bf real} Banach space, we define 
	\begin{equation*}
		\mathcal{T}^s(X) := \inf \{ r > 0: \exists x \in S_X \ \mbox{and a slice} \ S \ \mbox{of} \ B_X \ \mbox{with} \ S \subseteq B(x,r) \}.
	\end{equation*}
	Observe that $0 \leq \mathcal{T}^s(X) \leq 2$ and by the geometric characterization of the Daugavet property through slices it is easy to see that $\mathcal{T}^{s}(X)=2$ if and only if $X$ has the DPr. On the other hand, in \cite[Proposition 2.14]{ChoiJung2024}, it was shown that
	\begin{equation*}
		\mathcal{T}^s(X) = \inf_{x \in S_X} dc(x) 
	\end{equation*}
	where $dc(x)$ is called the Daugavet constant of the point $x$ and it can be defined as
	\begin{equation*}
		dc(x)=\inf_{\substack{S\subseteq B_X\\ S\text{ slice}}}\ \sup_{y\in S}\|x-y\|.
	\end{equation*}
	They also prove (see \cite[Proposition 2.4]{ChoiJung2024}) that, for a real Banach space $X$ and $x \in B_X$, every rank-one operator $T = x^* \otimes x$ with $\|x^*\| \geq 1$ satisfies $\| \id_X + T\| \geq 1+(dc(x) - 1)\|x^*\|$. As $\|T\| = \|x^*\|$, this last inequality becomes 
	\begin{equation} \label{choi-jung-estimate}
		\| \id_X + T\| \geq 1 + (dc(x) - 1)\|T\|
	\end{equation}
	and this means that a uniform lower bound on $dc(x)$ over the unit sphere gives a uniform lower bound on $\| \id_X + T\|$ for every rank-one operator $T \in \mathcal{L}(X)$. Let $\eta \in (0,1)$ be arbitrary. Then, by using \cite[Theorem 2.7]{HallerLangemetsLimaNadelRuedaZoca2021}, there exists a real Banach space $X_{\eta}$ such that $\mathcal{T}^s(X_{\eta})=2 - \eta$. Since $2 - \eta < 2$, it follows that $X_{\eta}$ does not satisfy the classical DPr. As $\mathcal{T}^s(X_{\eta}) = \inf_{x \in S_{X_{\eta}}} dc(x)$, we have that $dc(x) \geq 2 - \eta$ for every $x \in S_{X_{\eta}}$. By (\ref{choi-jung-estimate}), it follows that 
	\begin{equation*}
		\| \id_X + T \| \geq 1 + (1-\eta) \|T\| 
	\end{equation*}
	for every rank-one operator $T \in \mathcal{L}(X_{\eta})$. In other words, for every $0 < \eta < 1$ arbitrary, there exists a real Banach space $X_{\eta}$ such that $X_{\eta}$ fails the DPr and $\| \id_X + T\| \geq 2 - \eta$ for every rank-one operator with norm one.

	Nevertheless, observe that this does not by itself yield a $G$-Banach space with small $\alpha_G(X)$ and satisfying the $G$-DPr without the classical DPr. What it does show is that the estimate (\ref{inequality-estimation-alpha-G}) is not strong enough to prove that these two assumptions together give back the Daugavet property. So, we have the following question.
	\begin{problem} For every $\eta > 0$, does there exist a $G$-Banach space such that $X$ has the $G$-DPr, $\alpha_G(X) < \eta$ and $X$ does not have the classical Daugavet property?
	\end{problem}

	We have the following relation between $\mathcal{T}^s(X)$ and $\alpha_G(X)$.
	
	\begin{theorem} \label{estimate-T-s-alpha-G} Let $X$ be a real $G$-Banach space with the $G$-DPr. Then,
		\begin{equation*}
			\mathcal{T}^s(X) \geq 2 - \alpha_G(X).
		\end{equation*}
		Equivalently, $dc(x) \geq 2 - \alpha_G(X)$ for every $x \in S_X$.
	\end{theorem}
	
	\begin{proof} For a fixed $x \in S_X$, we will show that $dc(x) \geq 2 - \alpha_G(X)$. Let $S \subseteq B_X$ be an arbitrary slice of $B_X$, say $S = S(B_X, x^*, \delta)$ for some $x^* \in S_{X^*}$ and $\delta > 0$ and let $\e > 0$ be given. Consider the associated $G$-slice 
		\begin{equation*}
			S_G := S_G(B_X, x^*, \delta) = \{ u \in B_X: \exists g \in G \ \mbox{such that} \ x^*(g \cdot u) > 1 - \delta \}.
		\end{equation*}
		Since $X$ has the $G$-DPr, we can apply the characterization through $G$-slices for the point $-x$ to get an element $u \in S_G$ such that $\|u-x\| > 2 - \e$. Since $u \in S_G$, there exists $g \in G$ such that $g\cdot u \in S$. Set $y:= g\cdot u \in S$. Now, we have that 
		\begin{equation*}
			\|y - x\| = \|g \cdot u - x\| \geq \|u-x\| - \|g \cdot u - u\|. 
		\end{equation*}
		Notice now that 
		\begin{equation*}
			\|g \cdot u - u\| = \left\| g \cdot \frac{u}{\|u\|} - \frac{u}{\|u\|} \right\| \cdot \|u\| \leq \alpha_G(X)
		\end{equation*}
		whenever $u\not=0$ (in the case $u=0$ the same holds true trivially). This shows that 
		\begin{equation*}
			\| y - x\| \geq \|u-x\| - \|g \cdot u - u\| \geq 2 - \alpha_G(X) - \e.
		\end{equation*}
		As $\e > 0$ is arbitrary, we get that $\sup_{z \in S} \|x-z\| \geq 2 - \alpha_G(X)$ and since the slice $S$ is also arbitrary, we have that 
		\begin{equation*}
			dc(x)=\inf_{\substack{S\subseteq B_X\\ S\text{ slice}}}\ \sup_{z\in S}\|x-z\| \geq 2 - \alpha_G(X) 
		\end{equation*}
		for every $x \in S_X$. This shows that 
		\begin{equation*}
			\mathcal{T}^s(X) = \inf_{x \in S_X} dc(x) \geq 2 - \alpha_G(X) 
		\end{equation*}
		as we wanted. 
	\end{proof}

	We have the following consequences. Item (a) below is a consequence of \cite[Proposition 2.10]{ChoiJung2024}. Indeed, if $x \in B_X$ is an LUR point, then $dc(x) = 1 - \|x\|$. In particular, for every $x \in S_X$ which is LUR, $dc(x) = 0$. So, if $X$ is LUR and has the $G$-DPr, then for every $x \in S_X$, we have that 
	\begin{equation*}
		0 = dc(x) \geq 2 - \alpha_G(X)
	\end{equation*}
	by \Cref{estimate-T-s-alpha-G}. For (b), it is immediate from \cite[Lemma 2.8]{ChoiJung2024} as if $y$ is a denting point on $S_X$, then $dc(y) = 0$. To sum it up, we have the following result. Item (a) should be compared to \Cref{theorem:GDPrOnUniformlyConvexSpace} and item (b) with \Cref{thm:GDPr-alphaG}.
	
	\begin{corollary} \label{corollary-LUR-denting} Let $X$ be a real $G$-Banach space with the $G$-DPr. The following holds true.
		\begin{itemize}
			\item[(a)] If $X$ is LUR, then $\alpha_G(X) = 2$.
			\item[(b)] If $X$ has a denting point on $S_X$, then $\alpha_G(X) = 2$.
		\end{itemize}
	\end{corollary}

	As a particular case of \Cref{corollary-LUR-denting}.(b), we have that if $X$ satisfies the $G$-DPr and $\alpha_G(X) < 2$, then $X$ cannot have the RNP and we recover back \Cref{thm:GDPr-alphaG}.

	One can consider weak-star versions of the Daugavet indices of thickness as in \cite[Section 3]{HallerLangemetsLimaNadelRuedaZoca2021}. Here, we will consider the weak-star version of $\mathcal{T}^s(X)$. Let $X$ be a real Banach space. We consider 
	\begin{equation*}
		\mathcal{T}_{w^*}^s(X) := \inf \{ r>0: \exists x^* \in S_{X^*} \ \mbox{and a weak}^*\mbox{-slice} \ S \ \mbox{of} \ B_{X^*} \ \mbox{with} \ S \subseteq B(x^*, r) \}.
	\end{equation*}

	As in \Cref{estimate-T-s-alpha-G}, we will get an estimate which relates $\mathcal{T}_{w^*}^s(X)$ and $\alpha_G(X)$, and as in \Cref{relation-between-alpha-and-numerical-index} we get that real $G$-Banach spaces with the $G$-DPr have duals which fail the RNP, now in a quantitative manner. For this, we equip the dual $X^*$ with the canonical dual action as in \Cref{prop:inheritanceFromDual}. First notice that $\alpha_G(X^*) \leq \alpha_G(X)$. Indeed, for every $g \in G$ and $x^* \in S_{X^*}$, we have that 
	\begin{equation*}
		\|g \cdot x^* - x^*\| = \sup_{x \in B_X} \|(g \cdot x^* - x^*)(x) \| = \sup_{x \in B_X} |x^*(g^{-1}x - x)| \leq \alpha_G(X).
	\end{equation*}
	
	We will use this in the proof of the next result.
	
	\begin{theorem} \label{theorem:relation-between-alphaG-dual} Let $X$ be a real $G$-Banach space with the $G$-DPr. Then,
		\begin{equation} \label{estimate-alphaG-dual}
			\mathcal{T}_{w^*}^s(X) \geq 2 - \alpha_G(X).
		\end{equation}
	\end{theorem}

	We get the following consequence of \Cref{theorem:relation-between-alphaG-dual}, which should be compared to \Cref{relation-between-alpha-and-numerical-index} above.
	
	\begin{corollary} Let $X$ be a real $G$-Banach space with the $G$-DPr. If $\alpha_G(X) < 2$, then $X^*$ fails the RNP.
	\end{corollary}

	\begin{proof}
		Assume towards a contradiction that $X^*$ has the RNP. Then $X$ is an Asplund space, and therefore the dual ball $B_{X^*}$ is weak$^*$-dentable. In particular, $B_{X^*}$ admits weak$^*$-slices of arbitrarily small diameter, so $\mathcal{T}_{w^*}^s(X)=0$. On the other hand, by \Cref{theorem:relation-between-alphaG-dual} we have $\mathcal{T}_{w^*}^s(X)\ge 2-\alpha_G(X)>0$, a contradiction.
	\end{proof}
	
	\begin{proof}[Proof of \Cref{theorem:relation-between-alphaG-dual}] We assume that $\alpha_G(X)<2$ as otherwise the inequality is trivial.   We show that no weak$^*$-slice of $B_{X^*}$ can be contained in a ball of radius strictly smaller than $2-\alpha_G(X).$
		For this, we fix $r < 2 - \alpha_G(X)$. More precisely, we will show that there do not exist $x_0^* \in S_{X^*}$ and a weak$^*$-slice $S$ of $B_{X^*}$ such that $S \subseteq B(x_0^*, r)$. Once this is done, by the definition of $\mathcal{T}_{w^*}^s(X)$, we will have that the estimate (\ref{estimate-alphaG-dual}) follows. Assume by contradiction that such $x_0^* \in S_{X^*}$ and weak$^*$-slice $S$ do exist. Since $S$ is a weak$^*$-slice of $B_{X^*}$, there are $x_0 \in S_X$ and $\e>0$ such that 
		\begin{equation*}
			S = S(B_{X^*}, x_0, \e) := \{ y^* \in B_{X^*}: y^*(x_0) > 1 - \e \}.
		\end{equation*}
		Let $\eta \in (0, \e)$ be so small that $2 - 2 \eta - \alpha_G(X) > r$ and consider the rank-one operator $T:=x_0^* \otimes x_0$. Since $X$ has the $G$-DPr, the rank-one operator $-T$ satisfies the $G$-Daugavet equation and by (\ref{G-Daugavet-equation}) we have that $\sup_{g \in G} \|g - T\| = 2$. Passing to adjoints, we have that $\sup_{g \in G} \|g^* - T^*\| = 2$. Now, as $T^*y^* = y^*(x_0) x_0^*$ for every $y^* \in X^*$, there are $g \in G$ and $y^* \in S_{X^*}$ such that 
		\begin{equation*}
			\|g^* y^* - y^*(x_0) x_0^*\| > 2 - \eta. 
		\end{equation*}
		Now, since  $\|g^* y^* - y^*(x_0) x_0^*\| \leq 1 + |y^*(x_0)|$, we have that $|y^*(x_0)| > 1 - \eta$. Let $\theta := \sign(y^*(x_0)) \in \{\pm 1\}$ and set $z^* := \theta y^* \in S_{X^*}$. We will show that $z^* \in S$ and $\|z^* - x_0^*\| > r$, which contradicts $S \subseteq B(x_0^*, r)$. Indeed, on the one hand, we have that $z^*(x_0) = |y^*(x_0)| > 1 - \eta > 1 -\e$. So, $z^* \in S$. On the other hand, we have that 
		\begin{eqnarray*}
			\|g^*z^* - x_0^*\| &\geq& \| g^* z^* - |y^*(x_0)| x_0^*\| - |1 - |y^*(x_0)| | \\
			&=& \| g^* \theta y^* - |y^*(x_0)| x_0^*\| - |1 - |y^*(x_0)|| \\
			&=& \|g^* y^* - y^*(x_0) x_0^*\| - |1 - |y^*(x_0)|| \\
			&>& 2 - 2 \eta.
		\end{eqnarray*}
		Thus,
		\begin{equation*}
			\|z^* - x_0^*\| \geq \|g^* z^* - x_0^*\| - \|g^* z^* - z^*\| > 2 - 2 \eta - \alpha_G(X^*). 
		\end{equation*}
		As $\alpha_G(X^*) \leq \alpha_G(X)$, we obtain 
		\begin{equation*}
			\|z^* - x_0^*\| > 2 - 2 \eta - \alpha_G(X).
		\end{equation*}
		By the choice of $\eta$, this shows that $z^* \in S$ is such that $\|z^* - x_0^*\| > r$, contradicting that $S \subseteq B(x_0^*, r)$. 
	\end{proof}

	\subsection{On subgroups of the complex unit sphere}  It is still not known about the existence of an infinite-dimensional complex reflexive Banach space with the aDPr (see \cite[Problem 12.6]{KMRW}). As the aDPr is the $S_{\K}$-DPr, one might think that it is possible to find a subgroup $H$ of $S_{\C}$ and a complex reflexive Banach space $X$ such that $X$ has the $H$-DPr. This would solve the open problem by applying \Cref{fact-DP-implies-GDP}. Nevertheless, this approach does not seem to work as we can see in \Cref{complex-reflexive-aDPr-problem} below as we know exactly what happens with subgroups (finite and infinite) of $S_{\C}$ acting by multiplication. For this, we define the following subset of $B_{\C}$.

	\begin{definition} Let $X$ be a complex $G$-Banach space. We define 
		\begin{equation*} 
			M_G(X) := \{x^*(g \cdot x): x \in S_X, x^* \in S_{X^*}, x^*(x) = 1, g \in G \}.
		\end{equation*}
	\end{definition}
	The set $M_G(X)$ can be seen as the numerical range of the action of $G$. Indeed, notice that
	\begin{equation*} 
		M_G(X) = \bigcup_{g\in G} \{x^*(g\cdot x): x\in S_X,\ x^*\in S_{X^*},\ x^*(x)=1\}.
	\end{equation*}
	In other words, $M_G(X)$ is the union of the spatial numerical ranges of the isometries induced by the elements of $G$. In this sense, $M_G(X)$ is a numerical range version of the parameter $\alpha_G(X)$. While $\alpha_G(X)$ only measures how far the action can move unit vectors in norm, the set $M_G(X)$ keeps track of the complex values $x^*(g \cdot x)$ which are seen by norming pairs. This distinction is natural in the complex setting, since the aDPr is closely connected
	with numerical range theory and with the action of the whole unit circle $S_{\C}$. More precisely, the estimate
	\begin{equation*} 
		|x^*(g\cdot x)-1| = |x^*(g\cdot x-x)| \leq \|g\cdot x-x\| \leq \alpha_G(X)
	\end{equation*} 
	shows that $M_G(X)\subseteq \{\lambda\in B_{\C}:|\lambda-1|\leq \alpha_G(X)\}$. Thus $M_G(X)$ represents, in scalar numerical range form, the same antipodal
	behavior measured by $\alpha_G(X)$. In fact, $\alpha_G(X)=2$ if and only if
	$-1\in\overline{M_G(X)}$. Indeed, if $-1\in\overline{M_G(X)}$, then the previous estimate gives
	$\alpha_G(X)\geq 2$, hence $\alpha_G(X)=2$. Conversely, if $\alpha_G(X)=2$,
	we can find $g_n\in G$ and $x_n\in S_X$ such that $\|g_nx_n-x_n\|\longrightarrow 2$. By the Bishop-Phelps-Bollobás theorem for functionals \cite{Bollobas1970}, there exist $x_n^*\in S_{X^*}$ with
	$x_n^*(x_n)=1$ and  $x_n^*(g_nx_n)\longrightarrow -1$. Hence $-1\in\overline{M_G(X)}$.

	We start with the following result.

	\begin{proposition} \label{MG(X)-result} Let $X$ be a complex $G$-Banach space. If $X$ has the $G$-DPr and the RNP, then $S_{\C} \subseteq \overline{M_G(X)}$.
	\end{proposition}
	
	\begin{proof} Fix $\omega \in S_{\C}$ and consider the operator $T:=\overline{\omega} \id_X$. Then, $\|T\| = 1$ and since $X$ has the $G$-DPr and the RNP, we have that $n_{G}(X) = 1$ (see \Cref{thm:G-DPandRNP}). In particular, we have that $v_G(T) = \|T\| = 1$. By the definition of $v_G(T)$ and how $T$ is defined, we have that 
		\begin{equation*}
			\sup \{ \re (\overline{\omega} x^*(g \cdot x)): x \in S_X, x^* \in S_{X^*}, x^*(x) = 1, g \in G \} = 1.
		\end{equation*}
		This means that, given $\e > 0$, there are $x \in S_X$ and $x^* \in S_{X^*}$ such that $x^*(x) = 1$, and $g \in G$ such that $\re (\overline{\omega} x^*(g \cdot x)) > 1 - \e$. Set $\omega_0:= x^*(g \cdot x) \in M_G(X) \subseteq B_{\C}$. Then,
		\begin{equation*}
			|\omega_0 - \omega|^2 = |\omega_0|^2 + |\omega|^2 - 2 \re (\overline{\omega} \omega_0) \leq 2 - 2 \re (\overline{\omega} \omega_0) < 2 \e. 
		\end{equation*}
		This implies that $|\omega_0 - \omega| < \sqrt{2\e}$ and therefore $\omega \in \overline{M_G(X)}$ as we wanted. 
	\end{proof}

	\begin{remark} 
		
		If $X$ is a real Banach space, then we would have that $M_G(X) \subseteq [-1,1]$ and the analogous of \Cref{MG(X)-result} would say that $S_{\R} = \{-1,1\} \subseteq \overline{M_G(X)}$. Notice that $1 \in M_G(X)$ always holds true as by taking $g = \id_X$ and $(x, x^*) \in S_X \times S_{X^*}$ with $x^*(x) = 1$, we have that $x^*(g \cdot x) = x^*(x) = 1$. Therefore, the only nontrivial case that one would get from \Cref{MG(X)-result} would be that $-1 \in \overline{M_G(X)}$. In fact, we are not interested in the real case as the only outcome from it would be either the DPr or the aDPr.
		
	\end{remark}
	
	We have the following immediate consequence of \Cref{MG(X)-result}.
	
	\begin{corollary} \label{MG(X)-corollary} Let $X$ be a complex $G$-Banach space. If there exists a closed set $F \subseteq B_{\C}$ such that $M_G(X) \subseteq F$ and $S_{\C}\not\subseteq F$, then $X$ cannot have both $G$-DPr and RNP.
	\end{corollary}

	For scalar actions, \Cref{MG(X)-corollary} depends only on the subgroup of scalars and not on the complex Banach space itself as we can see below. 
	
	\begin{corollary} \label{complex-reflexive-aDPr-problem}
		Let $X$ be a complex Banach viewed as an $S_{\C}$-Banach space via scalar multiplication. Let $H \leq S_{\C}$ be a subgroup.
		\begin{itemize}
			\item[(a)] If $X$ has the RNP and $H$ is finite, then $X$ does not have the $H$-DPr.
			\item[(b)] If $H$ is infinite, then $X$ has the $H$-DPr if and only if $X$ has the aDPr.
		\end{itemize}
	\end{corollary}
	
	\begin{proof} (a). Let $H \leq S_{\C}$ act on a complex Banach space $X$ by scalar multiplication, that is, $h\cdot x=hx$ for every $h\in H$ and $x\in X$. Then $M_H(X)=H$. In fact, if $x\in S_X$, $x^*\in S_{X^*}$ satisfy $x^*(x)=1$, and $h\in H$, then $x^*(h\cdot x)=h x^*(x)=h$. Therefore, if $\overline{H}\neq S_{\C}$, then \Cref{MG(X)-corollary} implies that $X$ cannot have both the Radon-Nikodým property and the $H$-Daugavet property. In particular, this applies to every finite subgroup $H\le S_{\C}$ acting on $X$ by scalar multiplication.

		\noindent
		(b). Suppose that $H$ is infinite. Then, $H$ is dense in $S_{\C}$. As the mapping $g \mapsto \| \id_X + g \circ T\|$ is continuous on $S_{\C}$ for every rank-one operator $T \in \mathcal{L}(X)$, we can take the supremum over $g \in G$ and find out that $X$ having the $H$-DPr is equivalent to $X$ having the $S_{\C}$-DPr.
	\end{proof}

	\Cref{complex-reflexive-aDPr-problem} motives us the following problem. A positive answer to it would give us a positive answer to \cite[Problem 12.6]{KMRW}.
	
	\begin{problem} Are there an infinite subgroup $H \leq S_{\C}$ and an infinite dimensional complex reflexive $G$-Banach space with the $H$-DPr?
	\end{problem}

	If $G = \{\id_X\}$ is the trivial group, then $M_G(X) = \{1\}$. Hence, $\overline{M_G(X)} = \{1\}$ clearly does not contain $S_{\C}$. In particular, this brings us back once again to the fact that if $X$ is a (complex) Banach space with the classical DPr, then $X$ cannot have the RNP (see \Cref{thm:GDPr-alphaG} and \Cref{thm:G-DPandRNP}).

	Notice that the proof in \Cref{MG(X)-result} actually shows that if $n_G(X)=1$ then $S_{\mathbb{C}}\subseteq \overline{M_G(X)}$. From this together with \Cref{result-contain-copies-of-l1} it one obtains the following result for complex $G$-Banach spaces.

	\begin{proposition} Let $X$ be a complex $G$-Banach space and assume that the action of $G$ on $X$ is continuous.
		\begin{itemize}
			\item[(a)] If $X$ has the $G$-DPr and does not contain copies of $\ell_1$, then $S_{\C} \subseteq \overline{M_G(X)}$.
			\item[(b)] If $H \leq S_{\C}$ is finite and acts on $X$ by scalar multiplication, and if $X$ has the $H$-DPr, then $X$ contains a copy of $\ell_1$. 
		\end{itemize}
	\end{proposition}
	
	\section{Acknowledgments and funding information}

	\textbf{Acknowledgments}: A substantial part of this work was carried out during a visit of Tomáš Raunig to the University of Granada in 2025. He is grateful for the hospitality and support received during his stay. The authors are especially grateful to Yoël Perreau for his careful reading of previous versions of this manuscript and for several valuable suggestions which proved very useful in its preparation. They are also grateful to Michal Doucha for valuable insights on group theory. Finally, they would like to thank Mingu Jung and Miguel Martín for some discussions related to the topics of the paper.

	\textbf{Funding information}: We describe the fundings of the three authors.
	
	Sheldon Dantas was supported by Grant PID2021- 122126NB-C33, funded by MICIU / AEI/ 10.13039/ 501100011033 and by ERDF/EU, and by Grant PID2021-122126NB-C31, funded by MICIU/AEI/10.13039/501100011033 and by ERDF/EU. 
	
	Helena del Río was supported by Grant PRE2022-103590, funded by MICIU/AEI/ 10.13039/ 501100011033 and by ESF+, Grant PID2021-122126NB-C31, funded by MICIU/AEI/10.13039/ 501100011033 and by ERDF/EU, by the "María de Maeztu" Excellence Unit IMAG, funded by MICIU/AEI/10.13039/501100011033 under reference CEX2020-001105-M and by Junta de Andaluc\'ia, grant FQM-0185. 
	
	Tomáš Raunig was supported by the GAČR project 25-15366S, the GAUK project 344226, and the Czech Academy of Sciences (RVO 67985840).

	\bibliographystyle{abbrv}
	\bibliography{references}

\end{document}